\documentclass[11pt,reqno]{amsart}
\usepackage{amsfonts}
\usepackage{amsmath}
\usepackage{amssymb}
\usepackage{graphicx}
\usepackage{nameref,hyperref,cleveref}
\usepackage[margin=.85in]{geometry}
\newtheorem{theorem}{Theorem}[section]

\newtheorem{corollary}[theorem]{Corollary}

\newtheorem{definition}[theorem]{Definition}

\newtheorem{lemma}[theorem]{Lemma}

\newtheorem{proposition}[theorem]{Proposition}
\newtheorem{remark}[theorem]{Remark}

\numberwithin{equation}{section}

\def\W{\tilde{W}}

\begin{document}

\title[Homogeneous traces]{Traces for Homogeneous Sobolev Spaces in Infinite Strip-Like Domains}

\author{Giovanni Leoni}
\address{
Department of Mathematical Sciences\\
Carnegie Mellon University\\
Pittsburgh, PA 15213, USA
}
\email[G. Leoni]{giovanni@andrew.cmu.edu}
\thanks{G. Leoni was supported by an NSF Grant (DMS \#1714098).}

\author{Ian Tice}
\address{
Department of Mathematical Sciences\\
Carnegie Mellon University\\
Pittsburgh, PA 15213, USA
}
\email[I. Tice]{iantice@andrew.cmu.edu}
\thanks{I. Tice was supported by an NSF CAREER Grant (DMS \#1653161). }



\begin{abstract}
In this paper we construct a trace operator for homogeneous Sobolev spaces defined on infinite strip-like domains.  We identify an intrinsic seminorm on the resulting trace space that makes the trace operator bounded and allows us to construct a bounded right inverse.  The intrinsic seminorm involves two features not encountered in the trace theory of bounded Lipschitz domains or half-spaces.  First, due to the strip-like structure of the domain, the boundary splits into two infinite disconnected components.  The traces onto each component are not completely independent, and the intrinsic seminorm contains a term that measures the difference between the two traces.  Second, as in the usual trace theory, there is a term in the seminorm measuring the fractional Sobolev regularity of the trace functions with a difference quotient integral.  However, the finite width of the strip-like domain gives rise to a screening effect that bounds the range of the difference quotient perturbation.  The screened homogeneous fractional Sobolev spaces defined by this screened seminorm on arbitrary open sets are of independent interest, and we study their basic properties.   We conclude the paper with applications of the trace theory to partial differential equations.
\end{abstract}

\maketitle

\section{Introduction}

\subsection{Motivation}

For an open set $\Omega \subseteq \mathbb{R}^N$, $m\in\mathbb{N}$, and $1\le p\le \infty$, the homogeneous Sobolev space $\dot{W}^{m,p}(\Omega)$ consists of all real-valued functions $u\in L_{\operatorname*{loc}}^{1}(\Omega)$ such that all the weak derivatives of order $m$ belong to $L^{p}(\Omega)$.  The goal of this paper is to characterize the trace operators associated to $\dot{W}^{m,p}(\Omega)$ in infinite strip-like domains of the form
\begin{equation}\label{omega two graphs}
\Omega = \left\{  x\in\mathbb{R}^{N}\,|\,\eta^{-}(x^{\prime})<x_{N}<\eta
^{+}(x^{\prime})\right\}, 
\end{equation}
where $\eta^{\pm}:\mathbb{R}^{N-1}\rightarrow\mathbb{R}$ are given Lipschitz functions such that $\eta^{-}<\eta^{+}$.  Note that we do not require $\eta^\pm$ to be bounded, which allows the domain $\Omega$ to be unbounded in the vertical $x_N$ direction in addition to the horizontal $x'$ directions.

It is well-known (see, for example, \cite{benci-fortunato1979},
\cite{berger-schechter1972}, \cite{edmunds-evans1973},
\cite{galdi-simader1990}, \cite{meyers-serrin1960}, \cite{simader-sohr1996},
\cite{thater2002} and the references therein) that the inhomogeneous Sobolev
spaces $W^{m,p}(\Omega)$ (which we recall require functions and  all their weak derivatives up to order $m$ to belong to $L^p(\Omega)$) are ill-suited to study partial differential equations in unbounded domains.  This is already seen at the level of the fundamental solution to Laplace's equation when $N \ge 2$.  Indeed, if $B[0,1] \subset \mathbb{R}^N$ denotes the closed unit ball, then the function $u: \mathbb{R}^N \backslash B[0,1] \to \mathbb{R}$ given by 
\begin{equation*}
u(x) = 
\begin{cases}
\log|x| &\text{if } N = 2\\
1- |x|^{2-N} &\text{if }N \ge 3
\end{cases}
\end{equation*}
is a solution of the homogeneous Dirichlet problem
\begin{equation*}
\begin{cases}
\Delta u=0 & \text{in }\mathbb{R}^{2}\setminus B[0,1]\\
u=0 & \text{on }\partial B(0,1),
\end{cases}
\end{equation*}
and while we have the inclusions $u \in \dot{W}^{1,p}(\mathbb{R}^{N}\setminus B[0,1]) \cap \dot{W}^{2,q}(\mathbb{R}^{N}\setminus B[0,1])$ for every $N/(N-1)<p \le \infty$ and $1 < q \le \infty$, it is clear that $u\notin L^{r}(\mathbb{R}^{2}\setminus B[0,1])$ for any $1\leq r < \infty$. 

To overcome this problem, the standard approach in the literature is to use weighted Sobolev spaces (see, for example, \cite{benci-fortunato1979}, \cite{berger-schechter1972}, \cite{edmunds-evans1973}, \cite{kufner1985book}, \cite{mazya-mitrea-shaposhnikova2010}, \cite{thater2002}, \cite{uspenskii1961}
and the references therein).   An alternate approach to solving boundary value problems on infinite domains is to have a better understanding of the trace operator associated to homogeneous spaces.  

In the recent paper \cite{strichartz2016} Strichartz proved, among other things, a new characterization of the trace space associated to the homogeneous Sobolev space $\dot{H}^1(\Omega) := \dot{W}^{1,2}(\Omega)$, where $\Omega = \mathbb{R}\times(0,\pi) \subseteq \mathbb{R}^{2}$ is a horizontal strip. More precisely, he
defined the fractional-type Sobolev space $\tilde{H}^{1/2}(\mathbb{R})$ of all functions $f\in L_{\operatorname*{loc}}^{1}(\mathbb{R})$ such that
\begin{equation}\label{seminorm H 1/2}
|f|_{\tilde{H}^{1/2}(\mathbb{R})}:= \left( \iint_{|x-y|\leq1}\frac{|f(x)-f(y)|^{2}}{|x-y|^{2}}dxdy \right)^{1/2}<\infty
\end{equation}
and proved the following result.

\begin{theorem}[Strichartz \cite{strichartz2016}]
Let $\Omega=\mathbb{R}\times(0,\pi)$. Then the following hold.
\begin{enumerate}
 \item There exists a constant
$c>0$ such that for all $u\in\dot{H}^{1}(\Omega)$,
\begin{gather*}
\int_{\mathbb{R}}|\operatorname*{Tr}(u)(x_{1},\pi)-\operatorname*{Tr}(u)(x_{1},0)|^{2}dx_{1}\leq c\int_{\Omega}|\nabla u(x)|^{2}dx,\\
|\operatorname*{Tr}(u)(\cdot,\pi)|_{\tilde{H}^{1/2}(\mathbb{R})}^{2}+|\operatorname*{Tr}(u)(\cdot,0)|_{\tilde{H}^{1/2}(\mathbb{R})}^{2}
\leq c\int_{\Omega}|\nabla u(x)|^{2}dx.
\end{gather*}

 \item Given $f,g\in\tilde{H}^{1/2}(\mathbb{R})$ such that $f-g\in
L^{2}(\mathbb{R})$, there exists $u\in\dot{H}^{1}(\Omega)$ such that
$\operatorname*{Tr}(u)(\cdot,\pi)=f$, $\operatorname*{Tr}(u)(\cdot,0)=g$, and
\begin{equation*}
\int_{\Omega}|\nabla u(x)|^{2}dx\leq c\int_{\mathbb{R}}|f(x_{1})-g(x_{1}%
)|^{2}dx_{1}+c|f|_{\tilde{H}^{1/2}(\mathbb{R})}^{2}+c|g|_{\tilde{H}%
^{1/2}(\mathbb{R})}^{2}%
\end{equation*}
for some constant $c>0$, independent of $f$ and $g$.
\end{enumerate}
\end{theorem}

The seminorm $|\cdot|_{\tilde{H}^{1/2}(\mathbb{R})}$ defined in \eqref{seminorm H 1/2} has an interesting characterization in terms of the Fourier transform.  Indeed, one can show that (see Theorem 2.2 in \cite{strichartz2016} and Proposition \ref{proposition fourier} below)  
\begin{equation*}
|f|_{\tilde{H}^{1/2}(\mathbb{R})}^{2}\asymp \int_{\mathbb{R}}\min\{|\xi
|,|\xi|^{2}\}|\hat{f}(\xi)|^{2}d\xi,
\end{equation*}
where $\hat{f}$ is the Fourier transform of $f$.  This should be contrasted with the Fourier characterization of the seminorm defining the classical fractional homogeneous Sobolev space $\dot{H}^{1/2}(\mathbb{R})$:
\begin{equation*}
|f|_{\dot{H}^{1/2}(\mathbb{R})}^{2}:=\iint_{\mathbb{R}^{2}}\frac{|f(x)-f(y)|^{2}%
}{|x-y|^{2}}dxdy=c\int_{\mathbb{R}}|\xi||\hat{f}(\xi)|^{2}d\xi,
\end{equation*}
for $c >0$ a constant independent of $f$.  From these expressions we see that the difference between the seminorms is determined by the low frequencies on the Fourier side, with the seminorm on $\tilde{H}^{1/2}(\mathbb{R})$ allowing for worse behavior of the low frequency modes.  These expressions also suggests that $\dot{H}^{1/2}(\mathbb{R}) \subsetneqq \tilde{H}^{1/2}(\mathbb{R})$, i.e. that the classical space is strictly smaller.  This is indeed the case, as can be seen directly by considering a function $f\in C^{\infty
}(\mathbb{R)}$ such that $\operatorname*{supp}f^{\prime}\subseteq[ 0,1]$,
$f(x)=1$ for $x\geq 1$, and $f(x)=0$ for all $x\leq 0$. An
important consequence of the strict inclusion $\dot{H}^{1/2}(\mathbb{R}
)\subsetneqq\tilde{H}^{1/2}(\mathbb{R})$ is that, in general, functions in
$\dot{H}^{1}(\Omega)$ may not be extended to functions in $\dot{H}^{1}(\mathbb{R}^{2})$.

A central component of this paper consists of extending Strichartz's theorem to the more general case of $\dot{W}^{m,p}(\Omega)$ for $m\in \mathbb{N}$ and $1 \le p < \infty$, where $\Omega \subset \mathbb{R}^N$ for $N \ge 2$ is a horizontal strip of the form
\begin{equation}\label{omega infinite strip}%
\Omega:=\left\{  x\in\mathbb{R}^{N}\,|\,b^{-}<x_{N}<b^{+}\right\}
\end{equation}
for constants $b^{-}<b^{+}$.  In order to refer to the connected components of $\partial \Omega$ individually, we will write 
\begin{equation}\label{gamma_infinite_strip}
\Gamma^{\pm}=\left\{  x\in\mathbb{R}^{N}\,|\,x_{N}=b^{\pm}\right\}  . 
\end{equation}

\subsection{Main results and discussion}

The main thrust of the paper is to characterize the trace spaces associated to the homogeneous Sobolev space $\dot{W}^{m,p}(\Omega)$ when $\Omega$ is of the form \eqref{omega two graphs} or \eqref{omega infinite strip}.  The trace spaces involve a generalization of the space $\tilde{H}^{1/2}(\mathbb{R})$ from \cite{strichartz2016}, which we call a screened homogeneous fractional Sobolev space.  These spaces are important for our trace results but are of independent interest in a more general setting.  As such, in Section \ref{section spaces} we define these spaces on general open sets $U \subseteq \mathbb{R}^N$ for $0 < s < 1$ and $1 \le p < \infty$ through the seminorm 
\begin{equation*}
 L^1_{{\operatorname*{loc}}}(U) \ni f \mapsto |f|_{\W_{(\sigma)}^{s,p}(U)} =  \left(\int_{U} \int_{B(x,\sigma(x)) \cap U} \frac{|f(y)-f(x)|^p}{|y-x|^{sp+N}} dy xy \right)^{1/p} \in [0,\infty],
\end{equation*}
where $\sigma : U \to (0,\infty]$ is a lower semi-continuous function that we call a screening function, due to the fact that it screens the range of the difference quotient.  The functions for which the seminorm is finite comprise the space $\W_{(\sigma)}^{s,p}(U)$.  When $p=2$ and $\sigma$ is a finite constant, variants of the screened spaces have appeared in the analysis of weak formulations of various nonlocal elliptic equations (see \cite{du_etal2012}, \cite{felsinger-mathieu-kassman2015}, \cite{zhou-du2010}).  For general $p$ there has been much recent interest in inhomogeneous fractional-type Sobolev spaces with seminorms of the form 
\begin{equation*}
 L^p(U) \ni f \mapsto  \left(\int_{U} \int_{U } \frac{|f(y)-f(x)|^p}{|y-x|^{p}}\rho(|x-y|) dy xy \right)^{1/p} \in [0,\infty]
\end{equation*}
for a given kernel $\rho :[0,\infty) \to [0,\infty)$: we refer to the seminal papers \cite{bourgain-brezis-mironescu2001} and \cite{bourgain-brezis-mironescu2002} and to the papers \cite{brezis-nguyen2018}, \cite{ponce2004}, \cite{ponce-spector2017} and the references therein for a more-recent survey of this literature.  When $\rho(r) = \chi_{[0,b)}(r) r^{(1-s)p-N}$ these seminorms correspond to the screened homogeneous fractional seminorm for constant screening function $\sigma =b$, but our focus in this paper is the seminorm defined on $L^1_{{\operatorname*{loc}}}(U)$ rather than $L^p(U)$.  To the best of our knowledge the screened homogeneous fractional spaces for general $p$ and $\sigma$ have not been previously studied.

For the screened spaces we prove, among other things, Poincar\'{e}-type estimates, sequential completeness, and some basic interpolation and embedding results.  We also construct some sets and choices of screening functions (see Theorems \ref{theorem strict inclusion} and \ref{theorem strict inclusion vanishing} for the precise forms) for which we have the strict inclusion 
\begin{equation*}
 \dot{W}^{s,p}(U) \subsetneqq \W_{(\sigma)}^{s,p}(U),
\end{equation*}
which shows that the screened spaces are generally strictly bigger than the standard homogeneous fractional Sobolev spaces.  We also consider the special case $U = \mathbb{R}^N$, which is the domain most relevant to our trace results when we swap $N \mapsto N-1$.  We prove that if $\sigma_1,\sigma_2$ are both bounded above and below, then 
\begin{equation}\label{intro equiv}
\W_{(\sigma_1)}^{s,p}(\mathbb{R}^N) = \W_{(\sigma_2)}^{s,p}(\mathbb{R}^N),
\end{equation}
which shows that the precise form of the bounded screening function is not important.  We also establish density of smooth functions and prove a Fourier characterization when $p=2$. We have not attempted an exhaustive study of the screened spaces and have left many basic questions open.  In particular, we believe it would be interesting to study these spaces using interpolation theory and to give an equivalent characterization in terms of the Littlewood-Paley theory (see \cite{bennett-sharpley1988book}, \cite{bergh-lofstrom-1976book},
\cite{besov-ilin-nikolskii1978book}, \cite{besov-ilin-nikolskii1979book}, \cite{grafakos2014book}, \cite{peetre1976book}, \cite{triebel1995book}).

To state our trace results, we first need some notation.  We will write points $x \in \mathbb{R}^N$ as  $x=(x^{\prime},x_{N})\in\mathbb{R}^{N-1}\times\mathbb{R}$, and we will let $B^{\prime}(x^{\prime},r)$ denote the open ball in $\mathbb{R}^{N-1}$, centered at $x^{\prime}$, and of radius $r$.  In $\mathbb{R}^{N-1}$ with screening function given by the constant $a>0$ and $s = 1-1/p$, the screened homogeneous fractional Sobolev seminorm can be rewritten as
\begin{equation*}
|f|_{\W_{(a)}^{1-1/p,p}(\mathbb{R}^{N-1})}:=\left(  \int_{\mathbb{R}^{N-1}}\int_{B^{\prime}(0,a)}\frac{|f(x^{\prime}+h^{\prime})-f(x^{\prime})|^{p}}{|h^{\prime}|^{p+N-2}}\,dh^{\prime}dx^{\prime}\right)^{1/p}.
\end{equation*}

We can now state our trace results.  We begin with the case $m=1$ and $1<p<\infty$.  

\begin{theorem}\label{theorem trace strip}
Let $\Omega$ be as in \eqref{omega infinite strip} and let $1<p<\infty$. There exists a unique linear operator
\begin{equation*}
 \operatorname*{Tr}:\dot{W}^{1,p}(\Omega)\rightarrow L_{\operatorname*{loc}}^{p}(\partial\Omega)
\end{equation*}
satisfying the following.
\begin{enumerate}
 \item $\operatorname*{Tr}(u)=u$ on $\partial\Omega$ for all $u\in\dot{W}^{1,p}(\Omega)\cap C^0(\overline{\Omega})$.

 \item There exists a constant
$c=c(N,p)>0$ such that
\begin{align}
\int_{\mathbb{R}^{N-1}}|\operatorname*{Tr}(u)(x^{\prime},b^{+})-\operatorname*{Tr}(u)(x^{\prime},b^{-})|^{p}dx^{\prime}\leq(b^{+} -b^{-})^{p-1}\int_{\Omega}|\partial_{N}u(x)|^{p}dx,  \label{theorem trace strip est 1} \\
|\operatorname*{Tr}(u)(\cdot,b^{\pm})|_{\W_{(b^{+}-b^{-})}^{1-1/p,p} (\mathbb{R}^{N-1})}^{p}\leq c\int_{\Omega}|\nabla u(x)|^{p}dx \label{theorem trace strip est 2}
\end{align}
for all $u\in\dot{W}^{1,p}(\Omega)$.  Here $b^- < b^+$ are the constants defining $\Omega$ in \eqref{omega infinite strip}.

\item The integration by parts formula
\begin{equation*}
\int_{\Omega}u\partial_{i}\psi\,dx=-\int_{\Omega}\psi\partial_{i}u\,dx+\int_{\partial\Omega}\psi\operatorname*{Tr}(u)\nu_{i}\,d\mathcal{H}^{N-1}
\end{equation*}
holds for all $u\in\dot{W}^{1,p}(\Omega)$, all $\psi\in C_{c}^{1}(\mathbb{R}^{N})$, and all $i=1,\ldots,N$.
\end{enumerate}
\end{theorem}

Theorem \ref{theorem trace strip} is complemented by the following lifting result.  We recall that $\Gamma^\pm$ are defined in \eqref{gamma_infinite_strip}.

\begin{theorem}\label{theorem lifting strip}
Let $\Omega$ be as in \eqref{omega infinite strip}, $a>0$, and $1<p<\infty$.  Suppose that $f^{\pm}\in L_{\operatorname*{loc}}^{1}(\mathbb{R}^{N-1})$ are such that
\begin{gather}
\int_{\mathbb{R}^{N-1}}|f^{+}(x^{\prime})-f^{-}(x^{\prime})|^{p}\,dx^{\prime} < \infty,\label{difference f plus of minus}\\
|f^{-}|_{\W_{(a)}^{1-1/p,p}(\mathbb{R}^{N-1})}<\infty,\quad|f^{+}|_{\W_{(a)}^{1-1/p,p}(\mathbb{R}^{N-1})}<\infty. \label{seminorm strip}%
\end{gather}
Then there exists $u\in\dot{W}^{1,p}(\Omega)$ such that $\operatorname*{Tr}(u)=f^{\pm}$ on $\Gamma^{\pm}$ and
\begin{align*}
\int_{\Omega}|\nabla u(x)|^{p}dx    \leq c(b^{+}-b^{-})^{p-1}
\int_{\mathbb{R}^{N-1}}|f^{+}(x^{\prime})-f^{-}(x^{\prime})|^{p} \, dx^{\prime}
  +c|f^{-}|_{\W_{(a)}^{1-1/p,p}(\mathbb{R}^{N-1})}^{p} + |f^{+}|_{\W_{(a)}^{1-1/p,p}(\mathbb{R}^{N-1})}^{p}
\end{align*}
for some constant $c=c(a,N,p)>0$.  Moreover, the map $(f^-,f^+) \mapsto u$ is linear.  
\end{theorem}

\begin{remark}
It is important to observe that in Theorem \ref{theorem lifting strip} the functions $f^{\pm}$ are not assumed to be individually in $L^{p}(\mathbb{R}^{N-1})$, but their difference is. 
\end{remark}

\begin{remark}
Notice that in Theorem \ref{theorem trace strip} the screening function is the constant $b^+ - b^->0$, but in Theorem \ref{theorem lifting strip} the screening function is any constant $a>0$.  This discrepancy is accounted for by \eqref{intro equiv} (see Theorem \ref{theorem screened equivalence}), which shows that these screening functions define the same space and give rise to equivalent seminorms.
\end{remark}

Theorems \ref{theorem trace strip} and \ref{theorem lifting strip} show that for horizontal strips the trace space of the homogeneous space $\dot{W}^{1,p}(\Omega)$ is strictly larger than the trace space of the inhomogeneous space $W^{1,p}(\Omega)$. Indeed, by a classical result of Gagliardo \cite{gagliardo1957}, the trace space of $W^{1,p}(\Omega)$ is given by the fractional Sobolev space $W^{1-1/p,p}(\partial\Omega)$, or equivalently, by the Besov space $B^{1-1/p,p}_p(\partial\Omega)$ (see also \cite{adams-fournier2003book}, \cite{burenkov1998book}, \cite{grisvard2011book}, \cite{leoni2017book}, \cite{mazya2011book}, \cite{necas2012book}). The norm defining  $W^{1-1/p,p}(\partial \Omega)$ is 
\begin{equation*}
\Vert f\Vert_{W^{1-1/p,p}(\partial\Omega)}:=\Vert f\Vert_{L^{p}(\partial \Omega)}
+\left(  \int_{\partial\Omega}\int_{\partial\Omega}\frac{|f(x)-f(y)|^{p}}{|x-y|^{p+N-2}} \, d\mathcal{H}^{N-1}(x) \, d\mathcal{H}^{N-1}(y)\right)^{1/p}.
\end{equation*}
The class of pairs satisfying \eqref{difference f plus of minus} and \eqref{seminorm strip} gives rise to a new interesting type of space.  At the end of Section \ref{section spaces} we study some of its properties.  

By synthesizing our trace and lifting results with our results about the screened spaces in $\mathbb{R}^{N-1}$ we also prove Proposition \ref{proposition no extension}, which shows that there exist functions $u \in \dot{W}^{1,p}(\Omega)$ that do not admit extensions to $\dot{W}^{1,p}(\mathbb{R}^N)$.  In particular, this shows that there is no bounded extension operator available for $\dot{W}^{1,p}(\Omega)$.

Theorems \ref{theorem trace strip} and \ref{theorem lifting strip} can be extended to domains of the form \eqref{omega two graphs}, and we do so in Theorems \ref{theorem trace m=1} and \ref{theorem lifting m=1} below. The main difference is that conditions \eqref{difference f plus of minus} and \eqref{seminorm strip} should now be replaced, respectively, by
\begin{gather*}
\int_{\mathbb{R}^{N-1}}\frac{|f^{+}(x^{\prime})-f^{-}(x^{\prime})|^{p}}%
{(\eta^{+}(x^{\prime})-\eta^{-}(x^{\prime}))^{p-1}}\,dx^{\prime}<\infty,\\
\int_{\mathbb{R}^{N-1}}\int_{B^{\prime}(0,a(\eta^{+}(x^{\prime})-\eta
^{-}(x^{\prime})))}\frac{|f^{\pm}(x^{\prime}+h^{\prime})-f^{\pm}(x^{\prime
})|^{p}}{|h^{\prime}|^{p+N-2}}\,dh^{\prime}dx^{\prime}<\infty.
\end{gather*}
Note that we allow $\eta^{+}-\eta^{-}$ to be unbounded, and we do not assume that $\eta^{+}-\eta^{-}$ is bounded away from zero, so the extension from \eqref{omega infinite strip} to \eqref{omega two graphs} is non-trivial.

Next we consider the case $p=1$ and $m=1$.  In \cite{gagliardo1957} (see also \cite{burenkov1998book}, \cite{leoni2017book}) Gagliardo proved that for open bounded domains $\Omega\subset\mathbb{R}^{N}$ with Lipschitz continuous boundary the trace space of $W^{1,1}(\Omega)$ is given by $L^{1}(\partial\Omega)$ (see also the recent paper of Mironescu \cite{mironescu2015} for a simpler proof). In contrast, we can prove the following result.

\begin{theorem}\label{theorem trace strip p=1}
Let $\Omega$ be as in \eqref{omega infinite strip}. There exists a unique linear operator
\begin{equation*}
\operatorname*{Tr}:\dot{W}^{1,1}(\Omega)\rightarrow L_{\operatorname*{loc}}^{1}(\partial\Omega)
\end{equation*}
satisfying the following.
\begin{enumerate}
 \item $\operatorname*{Tr}(u)=u$ on $\partial\Omega$ for all $u\in\dot{W}^{1,1}(\Omega)\cap C^0(\overline{\Omega})$.

\item There exists a constant $c=c(N)>0$ such that for every $0<\varepsilon<b^{+}-b^{-}$ and every $u\in\dot{W}^{1,1}(\Omega)$,
\begin{align}
\int_{\mathbb{R}^{N-1}}|\operatorname*{Tr}(u)(x^{\prime},b^{+}) - \operatorname*{Tr}(u)(x^{\prime},b^{-})|\,dx^{\prime}\leq\int_{\Omega}|\partial_{N}u(x)|\,dx,  \label{theorem trace strip p=1 eq 1} \\
\sup_{|h^{\prime}|\leq\varepsilon}\int_{\mathbb{R}^{N-1}}|\operatorname*{Tr}(u)(x^{\prime}+h^{\prime},b^{\pm})-\operatorname*{Tr}(u)(x^{\prime},b^{\pm})|\,dx^{\prime}\leq\int_{\Omega_{\varepsilon}}|\nabla u(x)|\,dx^{\prime}
\label{theorem trace strip p=1 eq 2},
\end{align}
where $\Omega_{\varepsilon} :=\mathbb{R}^{N-1}\times((b^{-},b^{-}+\varepsilon) \cup (b^{+}-\varepsilon,b^{+}))$. In particular,
\begin{equation*}
\lim_{\varepsilon\rightarrow0^{+}}\sup_{|h^{\prime}|\leq\varepsilon}%
\int_{\mathbb{R}^{N-1}}|\operatorname*{Tr}(u)(x^{\prime}+h^{\prime},b^{\pm
})-\operatorname*{Tr}(u)(x^{\prime},b^{\pm})|\,dx^{\prime}=0.
\end{equation*}

\item The integration by parts formula
\begin{equation*}
\int_{\Omega}u\partial_{i}\psi\,dx=-\int_{\Omega}\psi\partial_{i}u\,dx+\int_{\partial\Omega}\psi\operatorname*{Tr}(u)\nu_{i}\,d\mathcal{H}^{N-1}
\end{equation*}
holds for all $u\in\dot{W}^{1,1}(\Omega)$, all $\psi\in C_{c}^{1}(\mathbb{R}^{N})$, and all $i=1,\ldots,N$. 

\end{enumerate}

\end{theorem}

Theorem \ref{theorem trace strip p=1} is complemented by the following lifting result.

\begin{theorem}\label{theorem lifting strip p=1}
Let $\Omega$ be as in \eqref{omega infinite strip}.  Suppose that $f^{\pm}\in L_{\operatorname*{loc}}^{1}(\mathbb{R}^{N-1})$ are such that
\begin{gather}
\int_{\mathbb{R}^{N-1}}|f^{+}(x^{\prime})-f^{-}(x^{\prime})|\,dx^{\prime
}<\infty,\label{p=1 lift 1}\\
\lim_{\varepsilon\rightarrow0^{+}}\sup_{|h^{\prime}|\leq\varepsilon}%
\int_{\mathbb{R}^{N-1}}|f^{\pm}(x^{\prime}+h^{\prime})-f^{\pm}(x^{\prime
})|\,dx^{\prime}=0, \label{p=1 lift 2}.
\end{gather}
Then there exists $u\in\dot{W}^{1,p}(\Omega)$ such that $\operatorname*{Tr}(u)=f^{\pm}$ on $\Gamma^{\pm}$ (as defined in \eqref{gamma_infinite_strip}), and
\begin{align*}
\int_{\Omega}  &  |\nabla u(x)|\,dx\leq+c\int_{\mathbb{R}^{N-1}}%
|f^{+}(x^{\prime})-f^{-}(x^{\prime})|\,dx^{\prime}\\
&  +c\sup_{|h^{\prime}|\leq\varepsilon_{0}}\int_{\mathbb{R}^{N-1}}%
|f^{-}(x^{\prime}+h^{\prime})-f^{-}(x^{\prime})|\,dx^{\prime}+c\sup
_{|h^{\prime}|\leq\varepsilon_{0}}\int_{\mathbb{R}^{N-1}}|f^{+}(x^{\prime
}+h^{\prime})-f^{+}(x^{\prime})|\,dx^{\prime}%
\end{align*}
for some constant $c=c(a,b^{+},b^{-},N)>0$.    Moreover, the map $(f^-,f^+) \mapsto u$ is linear. 
\end{theorem}

Observe that functions $f^{\pm}\in L^{1}(\mathbb{R}^{N-1})$ satisfy \eqref{p=1 lift 1} and \eqref{p=1 lift 2} but one can construct non-integrable functions $f^{\pm}\in L_{\operatorname*{loc}}^{1}(\mathbb{R}^{N-1})$
satisfying \eqref{p=1 lift 1} and \eqref{p=1 lift 2}.  We have not been able to extend Theorems \ref{theorem trace strip p=1} and \ref{theorem lifting strip p=1} to more general domains of the form \ref{omega two graphs}. In particular, it is not clear what would be the
analog of condition \eqref{p=1 lift 2}.

Next we consider the case $m\geq2$ and $1<p<\infty$. For simplicity we present here only the case $m=2$ and refer to Section \ref{subsection m>2 strip} for the general case $m\geq2$. The main difference with the case $m=1$ is that in the higher order case the difference $\operatorname*{Tr}(u)(\cdot,b^{+}) - \operatorname*{Tr}(u)(\cdot,b^{-})$ does not have to belong to $L^{p}(\mathbb{R}^{N-1})$. This makes the proof of the lifting theorems significantly more involved.

Given a function $u\in\dot{W}^{2,p}\left(  \Omega\right)  $, we have that $\partial_{i}u\in\dot{W}^{1,p}\left(  \Omega\right)  $ for every $i=1,\ldots,N$ and thus by Theorem \ref{theorem trace strip} there exists $\operatorname*{Tr}(\partial_{i}u)$. By a density argument, it can be shown that $\operatorname*{Tr}(u)(\cdot,b^{\pm})\in W_{\operatorname*{loc}}^{1,p}(\mathbb{R}^{N-1})$ and that for every $i=1,\ldots,N-1$, $\partial_{i}\operatorname*{Tr}(u)(\cdot,b^{\pm})=\operatorname*{Tr}(\partial_{i}u)(\cdot,b^{\pm})$. Thus, to characterize the trace space of $\dot{W}^{2,p}\left(  \Omega\right)  $, it suffices to study $\operatorname*{Tr}
(u)$ and the trace of the normal derivative $\operatorname*{Tr}(\partial_{N} u)$.

\begin{theorem}\label{theorem trace strip m=2}
Let $\Omega$ be as in \eqref{omega infinite strip} and let $1<p<\infty$. Then there exists a constant $c=c(N,p)>0$ such that for every $u\in\dot{W}^{2,p}\left(\Omega\right) $ and for every $i=1,\ldots,N$,
\begin{equation}\label{theorem trace strip m=2 eq 1}
\int_{\mathbb{R}^{N-1}}|\operatorname*{Tr}(\partial_{i}u)(x^{\prime} ,b^{+})-\operatorname*{Tr}(\partial_{i}u)(x^{\prime},b^{-})|^{p}dx
\leq (b^{+}-b^{-})^{p-1}\int_{\Omega}|\partial_{i,N}^{2}u(x)|^{p}dx,  
\end{equation}
\begin{equation}\label{theorem trace strip m=2 eq 2}
\begin{split}
\int_{\mathbb{R}^{N-1}}\left\vert \operatorname*{Tr}(u)(x^{\prime} ,b^{+})-\operatorname*{Tr}(u)(x^{\prime},b^{-})-(\operatorname*{Tr}(\partial_{N}u)(x^{\prime},b^{+})  + \operatorname*{Tr}(\partial_{N}u)(x^{\prime},b^{-}))\frac{b^{+}-b^{-}}{2}\right\vert ^{p}dx^{\prime}   \\
\leq(b^{+}-b^{-})^{2p-1}\int_{\Omega}|\partial_{N}^{2}u(x)|^{p}dx,  
\end{split}
\end{equation}
and 
\begin{equation}\label{theorem trace strip m=2 eq 3}
|\operatorname*{Tr}(\partial_{i}u)(\cdot,b^{\pm})|_{\W_{(b^{+}-b^{-})}^{1-1/p,p}(\mathbb{R}^{N-1})}^{p}\leq c\int_{\Omega}|\nabla^{2}u(x)|^{p}dx. 
\end{equation}
\end{theorem}

Theorems \ref{theorem trace strip m=2} can be extended to domains of the form \eqref{omega two graphs}: see Theorems \ref{theorem trace m=2} and \ref{theorem trace m>2} below for the cases $m=2$ and $m\geq2$, respectively.

The previous theorem is complemented by the following lifting result.

\begin{theorem}\label{theorem lifting strip m=2}
Let $\Omega$ be as in \eqref{omega infinite strip},  $a>0$, and $1<p<\infty$.  Suppose that  $f_{0}^{\pm}\in W_{\operatorname*{loc}}^{1,p}(\mathbb{R}^{N-1})$ and $f_{1}^{\pm}\in L_{\operatorname*{loc}}^{p}(\mathbb{R}^{N-1})$ are such that
\begin{gather*}
\int_{\mathbb{R}^{N-1}}\left\vert f_{0}^{+}(x^{\prime})-f_{0}^{-}(x^{\prime})-(f_{1}^{+}(x^{\prime})+f_{1}^{-}(x^{\prime}))\frac{b^{+}-b^{-}} {2}\right\vert ^{p}dx^{\prime}<\infty,\\
\int_{\mathbb{R}^{N-1}}(|\nabla_{\shortparallel}f_{0}^{+}(x^{\prime})-\nabla_{\shortparallel}f_{0}^{-}(x^{\prime})|^{p}+|f_{1}^{+}(x^{\prime})-f_{1}^{-}(x^{\prime})|^{p})\,dx^{\prime}<\infty,\\
|\nabla_{\shortparallel}f_{0}^{\pm}|_{\W_{(a)}^{1-1/p,p}(\mathbb{R}^{N-1})}<\infty,\quad|f_{1}^{\pm}|_{\W_{(a)}^{1-1/p,p}(\mathbb{R}^{N-1})}<\infty,
\end{gather*}
where $\nabla_{\shortparallel}:=(\partial_{1},\ldots,\partial_{N-1})$.  Then there exists $u\in\dot{W}^{2,p}(\Omega)$ such that $\operatorname*{Tr}(u)=f_{0}^{\pm}$, $\operatorname*{Tr}(\partial_{N}u)=f_{1}^{\pm}$ on $\Gamma^{\pm}$, and
\begin{align*}
\int_{\Omega}  &  |\nabla^{2}u(x)|^{p}dx\leq c(b^{+}-b^{-})^{-2p+1} \int_{\mathbb{R}^{N-1}}\left\vert f_{0}^{+}(x^{\prime})-f_{0}^{-}(x^{\prime})-(f_{1}^{+}(x^{\prime}) + f_{1}^{-}(x^{\prime})) \frac{b^{+}-b^{-}}{2}\right\vert ^{p}dx^{\prime}\\
&  \quad+c(b^{+}-b^{-})^{-p+1}\int_{\mathbb{R}^{N-1}}(|\nabla_{\shortparallel}f_{0}^{+}(x^{\prime})-\nabla_{\shortparallel}f_{0}^{-}(x^{\prime})|^{p} + |f_{1}^{+}(x^{\prime})-f_{1}^{-}(x^{\prime})|^{p})\,dx^{\prime}\\
&  \quad+c|\nabla_{\shortparallel}f_{0}^{-}|_{\W_{(a)}^{1-1/p,p} (\mathbb{R}^{N-1})}^{p}+c|\nabla_{\shortparallel}f_{0}^{+}|_{\W_{(a)}^{1-1/p,p}(\mathbb{R}^{N-1})}^{p}+c|f_{1}^{-}|_{\W_{(a)}^{1-1/p,p}(\mathbb{R}^{N-1})}^{p}+c|f_{1}^{+}|_{\W_{(a)}^{1-1/p,p}(\mathbb{R}^{N-1})}^{p},
\end{align*}
for some constant $c=c(a,N,p)>0$.  Moreover, the map $(f_0^-,f_0^+,f_1^-,f_1^+) \mapsto u$ is linear.
\end{theorem}

The extension of Theorem \ref{theorem lifting strip m=2} to domains of the form \eqref{omega two graphs} seems to be quite complicated. We refer to Remark \ref{remark no lifting} below for more details on the technical difficulties one encounters.  We have also been unable to prove Theorems \ref{theorem trace strip m=2} and \ref{theorem lifting strip m=2} in the case $p=1$. For bounded Lipschitz domains $\Omega\subset\mathbb{R}^{N}$ we have that the image of the trace
operator
\begin{equation*}
u\in W^{2,1}(\Omega)\mapsto\operatorname*{Tr}\nolimits_{2}(u)
:=\left( \operatorname*{Tr}(u),\operatorname*{Tr}\left(  \frac{\partial u}{\partial\nu}\right)  \right)
\end{equation*}
is the space $B^{1,1}_1(\partial\Omega)\times L^{1}(\partial\Omega)$ (see \cite{leoni2017book}). We are aware of only two complete proofs of this result: one in a classical paper of Uspenski\u{\i} \cite{uspenskii1961}, and
one in a recent paper of Mironescu and Russ \cite{mironescu-russ2015}. In particular, in the latter the authors use an equivalent norm for $B^{1,1}_1$, which relies on the Littlewood--Paley decomposition (see \cite{leoni2017book}). The adaptation of the proofs in \cite{uspenskii1961} and \cite{mironescu-russ2015} to $\dot{W}^{2,1}\left(\Omega\right)  $ seems quite challenging.

We conclude this introduction with an application of our results to the existence of solutions to the Dirichlet problem for the $p$-Laplacian in unbounded domains of the form \eqref{omega infinite strip}.   In Section \ref{section applications} we prove analogous results when $\Omega$ is of the form \eqref{omega two graphs} and study the Neumann problem and more general second-order quasilinear elliptic equations.

\begin{theorem}\label{theorem dirichlet laplacian}
Let $\Omega$ be as in \eqref{omega infinite strip}, let $a>0$, let $1<p<\infty$, and let $g\in L^{p^{\prime}}(\Omega;\mathbb{R}^{N})$ and $f^{\pm}\in L_{\operatorname*{loc}}^{1}(\mathbb{R}^{N-1})$. Then the Dirichlet problem
\begin{equation}\label{dirichlet p-laplacian}
\left\{
\begin{array}
[c]{ll}
-\operatorname{div}(|\nabla u|^{p-2}\nabla u)=\operatorname{div}g & \text{in
}\Omega,\\
u=f^{\pm} & \text{on }\Gamma^{\pm}
\end{array}
\right.  
\end{equation}
admits a solution in $\dot{W}^{1,p}(\Omega)$ if and only if $f^{\pm}$ satisfy \eqref{difference f plus of minus} and \eqref{seminorm strip}.  In either case, there exists a constant $c=c(a,N,p)>0$ such that 
\begin{align}\label{dirichlet bounds}
\begin{split}
\int_{\Omega}|\nabla u(x)|^{p}dx  &  \leq c\int_{\Omega}|g(x)|^{p^{\prime}}dx + c\int_{\mathbb{R}^{N-1}}|f^{+}(x^{\prime})-f^{-}(x^{\prime})|^{p}\, dx^{\prime}\\
&  \quad+c|f^{-}|_{\W_{(a)}^{1-1/p,p}(\mathbb{R}^{N-1})}^{p} 
+ c|f^{+}|_{\W_{(a)}^{1-1/p,p}(\mathbb{R}^{N-1})}^{p}.
\end{split}
\end{align}
\end{theorem}

\subsection{Plan of paper}

In Section \ref{section prelims} we record notational conventions used throughout the paper as well as prove some useful preliminary results.  In Section \ref{section spaces} we define the general screened homogeneous fractional Sobolev spaces and study their basic properties.  These spaces play a key role in our trace results, but they are of independent interest and have intriguing properties.  In Section \ref{section traces strip} we develop the trace and lifting results when $\Omega$ is of the form \eqref{omega infinite strip}.  This section contains the proofs of Theorems \ref{theorem trace strip}, \ref{theorem lifting strip}, \ref{theorem trace strip p=1}, \ref{theorem lifting strip p=1}, \ref{theorem trace strip m=2}, and \ref{theorem lifting strip m=2}. In Section \ref{section traces strip-like} we study the trace and lifting results when $\Omega$ is of the form \eqref{omega two graphs}.  Section \ref{section applications} contains a number of applications of our results to quasilinear partial differential equations, including the proof of Theorem \ref{theorem dirichlet laplacian}.

\section{Preliminaries}\label{section prelims}

In this section we collect a number of preliminary results that will be of use throughout the paper.  We begin with some notational conventions, then we develop some facts about mollifiers, and then we develop an elementary calculus result that plays a key role in our trace spaces.  We conclude with a discussion of seminormed spaces.

\subsection{Notational conventions}
We now record our notational conventions.

\emph{Integers and multi-indices:}  We write $\mathbb{N} = \{1,2,\dotsc\}$ for the set of positive integers and $\mathbb{N}_0 = \mathbb{N} \cup \{0\}$ for the non-negative integers.  For a multi-index $\alpha \in \mathbb{N}_0^N$ we write $|\alpha| = \alpha_1 + \cdots + \alpha_N$.  We often write $\alpha \in \mathbb{N}_0^N$ as $\alpha = (\alpha',\alpha_N) \in \mathbb{N}_0^{N-1} \times \mathbb{N}_0$ to distinguish the horizontal part, $\alpha'$, from the vertical part $\alpha_N$.

\emph{Measures:} In what follows $\mathcal{L}^{N}$ stands for the Lebesgue measure in $\mathbb{R}^{N}$ and $\mathcal{H}^{k}$ is the $k-$dimensional Hausdorff measure.  We denote by $B(x,r)$ and $B[x,r]$ the open and closed balls in $\mathbb{R}^{N}$ centered at $x$ and radius $r$, respectively.  We write $\alpha_{N}$ for the Lebesgue measure of the $N$-dimensional unit ball $B(0,1)$ and $\beta_{N}$ for the $\mathcal{H}^{N-1}$ measure of the unit sphere $\mathbb{S}^{N-1}:=\partial B(0,1)$. 

\emph{Derivatives:} Given $k\in\mathbb{N}$ we denote by $\nabla^{k}$ the vector of all partial derivatives $\partial^\alpha$ for multi-indices $\alpha\in\mathbb{N}_{0}^{N}$ with $|\alpha|=k$. We also define
\begin{equation}
\nabla^{0}u:=u. \label{gradient zero}
\end{equation}

\emph{Horizontal and vertical variables:} For $x=(x_{1},\ldots,x_{N})\in\mathbb{R}^{N}$, $N\ge 2$, we write $x=(x^{\prime},x_{N})\in\mathbb{R}^{N-1}\times\mathbb{R}$, where $x^{\prime}:=(x_{1},\ldots,x_{N-1})$. We denote by $B^{\prime}(x^{\prime},r)$ the $(N-1)$-dimensional open ball centered at $x^{\prime} \in \mathbb{R}^{N-1}$ and radius $r>0$, and we write
\begin{equation*}
\nabla_{\shortparallel}:=(\partial_{1},\ldots,\partial_{N-1})
\end{equation*}
for the horizontal gradient.  Similarly, given $k\in\mathbb{N}$ we denote by $\nabla_{\shortparallel}^{k}$ the vector of all partial derivatives $\partial^\alpha$ for multi-indices $\alpha=(\alpha^{\prime},0)\in\mathbb{N}
_{0}^{N-1}\times\mathbb{N}_{0}$ with $|\alpha|=k$. We define 
\begin{equation}\label{gradient prime zero}
\nabla_{\shortparallel}^{0}u:=u. 
\end{equation}

\emph{Mollifiers:} Given a function $\phi:\mathbb{R}^{N}\rightarrow\mathbb{R}$ and
$\varepsilon>0$ we define
\begin{equation}\label{mollifier}
\phi_{\varepsilon}(x):=\frac{1}{\varepsilon^{N}}\phi\left(  \frac{x}{\varepsilon}\right). 
\end{equation}
Note that if $\phi$ is integrable, then a change of variables shows that
\begin{equation*}
\int_{\mathbb{R}^{N}}\phi_{\varepsilon}(x)\,dx =\int_{\mathbb{R}^{N}}\phi(y)\,dy.
\end{equation*}

\emph{Lipschitz functions:} Given a function $\eta:\mathbb{R}^N \rightarrow\mathbb{R}$ we set
\begin{equation*}
|\eta|_{0,1}:=\sup\left\{  \frac{|\eta(x)-\eta(y)}{|x-y|}\,|\,x,y\in\mathbb{R}^{N},\,x \neq y \right\}  
\end{equation*}
to be the Lipschitz seminorm of $\eta$.  We say $\eta$ is Lipschitz if and only if $|\eta|_{0,1} < \infty$.

\subsection{Special mollifiers}

In our analysis we will make heavy use of mollifiers.  Here we construct a special family of mollifiers and develop some auxiliary results. We begin by constructing mollifiers satisfying special moment conditions. In what follows we assume that $N\ge 2$.

\begin{proposition}\label{proposition mollifier}
Let $k,m\in\mathbb{N}$ be given. There exists a
function $\varphi\in C_{c}^{m}(\mathbb{R}^{N-1})$ such that $\operatorname*{supp}\varphi\subseteq B^{\prime}(0,1)$, 
\begin{equation}
\int_{\mathbb{R}^{N-1}}\varphi(x^{\prime})\,dx^{\prime}=1,\text{ and }%
\int_{\mathbb{R}^{N-1}}(x^{\prime})^{\alpha}\varphi(x^{\prime})\,dx^{\prime
}=0\text{ for }1\leq|\alpha|\leq k. \label{unit_mass}%
\end{equation}

\end{proposition}

\begin{proof}  
Assume that $\varphi:\mathbb{R}^{N-1}\rightarrow\mathbb{R}$ is given by
\begin{equation*}
\varphi(x)=%
\begin{cases}
\phi(|x^{\prime}|) & \text{for }|x^{\prime}|\leq1\\
0 & \text{for }|x^{\prime}|>1,
\end{cases}
\end{equation*}
for a function $\phi\in C^{m}([0,1])$ to be chosen. Then, using spherical
coordinates, we may compute
\begin{align*}
\int_{B^{\prime}(0,1)}  &  (x^{\prime})^{\alpha}\varphi(x^{\prime
})\,dx^{\prime}=\int_{0}^{1}r^{N-2+|\alpha|}\int_{\partial B^{\prime}%
(0,1)}(z^{\prime})^{\alpha}\varphi(rz^{\prime})\,d\mathcal{H}^{N-2}(z^{\prime
})dr\\
&  =\left(  \int_{\partial B^{\prime}(0,1)}(z^{\prime})^{\alpha}%
d\mathcal{H}^{N-2}(z^{\prime})\right)  \int_{0}^{1}r^{N-2+|\alpha|}%
\phi(r)\,dr=:Q_{\alpha}\int_{0}^{1}r^{N-2+|\alpha|}\phi(r)\,dr,
\end{align*}
where here we note that if $N=2$ then $\mathcal{H}^{N-2} = \mathcal{H}^{0}$ denotes counting measure and $\partial B'(0,1) = \{-1,1\}$. Note that $Q_{\alpha}=0$ except when $\alpha=2\beta$, so we may reduce to
studying the case when $|\alpha|$ is even. Without loss of generality we can
assume that $k=2l$.

Now we make the ansatz that $\phi$ is given by
\begin{equation*}
\phi(r)=(1-r^{2})^{m+1}\psi(r^{2})
\end{equation*}
for $\psi\in C^{m}([0,1])$. This forces $\phi^{(j)}(1)=0$ for $0\leq j\leq m$,
and so our function $\varphi$ has all derivatives up to order $m$ vanish at
$\partial B^{\prime}(0,1)$, which guarantees that $\varphi\in C_{c}%
^{m}(\mathbb{R}^{N-1})$.

It remains to enforce the moment conditions. In terms of $\phi$, the first of
these is
\begin{equation*}
\int_{0}^{1}r^{N-2}\phi(r)\,dr=\frac{1}{\alpha_{N-1}},
\end{equation*}
and the second is
\begin{equation*}
0=\int_{0}^{1}r^{N-2+2j}\phi(r)\,dr\text{ for }j=1,\dotsc,l.
\end{equation*}
Plugging in our ansatz for $\phi$ and making a change of variable $s = r^{2}$
then shows that our moment conditions for $\psi$ reduce to
\begin{equation*}
\frac{1}{\alpha_{N-1}}=\frac{1}{2}\int_{0}^{1}\psi(s)s^{(N-3)/2}(1-s)^{m+1}ds
\end{equation*}
and
\begin{equation*}
0=\frac{1}{2}\int_{0}^{1}\psi(s)s^{(N-3)/2+j}(1-s)^{m+1}ds\text{ for }1\leq
j\leq l.
\end{equation*}

Let $X=\operatorname*{span}(1,s,s^{2},\dotsc,s^{l})$ be the polynomials
(restricted to $[0,1]$) of degree at most $l$. Define $Y=\operatorname*{span}%
(s,s^{2},\dotsc,s^{l})$. We endow $X$ with the inner product
\begin{equation*}
\langle\zeta,\eta\rangle:=\frac{1}{2}\int_{0}^{1}\zeta(s)\eta(s)s^{(N-3)/2}%
(1-s)^{m+1}ds,
\end{equation*}
which is an inner-product since the weight function $s^{(N-3)/2}(1-s)^{m+1}$
is nonnegative and integrable (since $N \ge2$) on $[0,1]$ and vanishes only at
the endpoints.

We then rewrite our moment conditions as
\begin{equation*}
\langle1,\psi\rangle=\frac{1}{\alpha_{N-1}}\text{ and }\langle s^{j}%
,\psi\rangle=0\text{ for }j=1,\dotsc,l.
\end{equation*}
In order to enforce the latter condition we assume that $\psi\in Y^{\bot}$,
which is possible since $\dim(Y^{\bot})=1$. We can then find $\psi\in Y^{\bot
}$ satisfying $\langle1,\psi\rangle=1/\alpha_{N-1}$ if and only if $1^{\bot
}\neq0$, where $1^{\bot}\in Y^{\bot}$ is the orthogonal projection of $1$ onto
$Y^{\bot}$. Now, $1^{\bot}=0$ is equivalent to $1\in Y$, which is impossible
since if this were to hold then $Y=X$. Thus $1^{\bot}\neq0$, and we are
guaranteed the existence of $\psi\in X\subset C^{m}([0,1])$ satisfying all the
moment conditions.
\end{proof}

Next we identify a special structure for derivatives of mollifiers.

\begin{proposition}\label{proposition derivatives mollifiers}
Let $\varphi\in C_{c}^{m}(\mathbb{R}^{N-1})$. Then for every multi-index $\alpha\in\mathbb{N}_{0}^{N}$
with $1\leq|\alpha|\leq m$, there exists a function $\psi^{\alpha}\in
C_{c}^{m-|\alpha|}(\mathbb{R}^{N-1})$, with%
\begin{equation}\label{proposition derivatives mollifiers 1}
\int_{\mathbb{R}^{N-1}}\psi^{\alpha}(y^{\prime})\,dy^{\prime}=0,
\end{equation}
such that for every $x^{\prime},y^{\prime}\in\mathbb{R}^{N-1}$ and $x_{N}>0$,%
\begin{equation}
\frac{\partial^{\alpha}}{\partial x^{\alpha}}\left(  \frac{1}{x_{N}^{N-1}%
}\varphi\left(  \frac{x^{\prime}-y^{\prime}}{x_{N}}\right)  \right)  =\frac
{1}{x_{N}^{|\alpha|+N-1}}\psi^{\alpha}\left(  \frac{x^{\prime}-y^{\prime}%
}{x_{N}}\right)  . \label{psi i}%
\end{equation}
\end{proposition}

\begin{proof}
We proceed by finite induction on $|\alpha| \in \{1,\dotsc,m\}$.  Throughout the proof we will use the notation \eqref{mollifier} with $\varepsilon=x_{N}$. 

We begin with the base case, $|\alpha|=1$, which means that $\alpha=e_{i}$ for some $i=1,\ldots,N$. Then for $i=1,\ldots,N-1$,
\begin{align*}
\frac{\partial}{\partial x_{i}}\left(  \varphi_{x_{N}}\left(  x^{\prime
}-y^{\prime}\right)  \right)   &  =\frac{1}{x_{N}^{N}}\partial_{i}%
\varphi\left(  \frac{x^{\prime}-y^{\prime}}{x_{N}}\right)  =\frac{1}{x_{N}%
}\psi_{x_{N}}^{e_{i}}(x^{\prime}-y^{\prime})
\end{align*}
for $\psi^{e_{i}}(y^{\prime}):=\partial_{i}\varphi(y^{\prime})$, while for $i=N$ we have 
\begin{align*}
\frac{\partial}{\partial x_{N}}\left(  \varphi_{x_{N}}\left(  x^{\prime
}-y^{\prime}\right)  \right)   &  =\frac{1}{x_{N}^{N}}\left[  -(N-1)\varphi
\left(  \frac{x^{\prime}-y^{\prime}}{x_{N}}\right)  -\nabla_{\shortparallel
}\varphi\left(  \frac{x^{\prime}-y^{\prime}}{x_{N}}\right)  \cdot
\frac{x^{\prime}-y^{\prime}}{x_{N}}\right] \\
&  =\frac{1}{x_{N}}\psi_{x_{N}}^{e_{N}}(x^{\prime}-y^{\prime})
\end{align*}
for $\psi^{e_{N}}(y^{\prime}):=-(N-1)\varphi(y^{\prime})-\nabla_{\shortparallel
}\varphi\left(  y^{\prime}\right)  \cdot y^{\prime}$.  
The change of variables $z^{\prime}:=\frac{x^{\prime}-y^{\prime}}{x_{N}}$ reveals that
\begin{equation*}
\int_{\mathbb{R}^{N-1}}\varphi_{x_{N}}(x^{\prime}-y^{\prime})\,dy^{\prime
}=\int_{\mathbb{R}^{N-1}}\varphi(z^{\prime})\,dz^{\prime},
\end{equation*}
so for every $i=1,\ldots,N$ we have that
\begin{multline}\label{derivative zero}
\frac{1}{x_N}  \int_{\mathbb{R}^{N-1}} \psi^{e_i}(y') dy' 
=  \frac{1}{x_N}  \int_{\mathbb{R}^{N-1}} \psi_{x_N}^{e_i}(y') dy'  
=   \int_{\mathbb{R}^{N-1}}\frac{\partial}{\partial x_{i}}(\varphi_{x_{N}}\left(
x^{\prime}-y^{\prime}\right)  )\,dy^{\prime} \\
=\frac{\partial}{\partial x_{i}}\int_{\mathbb{R}^{N-1}}\varphi_{x_{N}}\left(  x^{\prime}-y^{\prime}\right)\,dy^{\prime}
=0. 
\end{multline}
Consequently, \eqref{proposition derivatives mollifiers 1} holds for $|\alpha|=1$, which completes the proof of the base case.

Next assume that the result is true for all $\beta\in\mathbb{N}_{0}^{N}$ with
$|\beta|=k$, $1\leq k\leq m-1$, and let $\alpha\in\mathbb{N}_{0}^{N}$ with
$|\alpha|=k+1$. Write $\alpha=\beta+e_{i}$, where $|\beta|=k$ and
$i\in\{1,\ldots,N\}$. Then
\begin{align*}
\frac{\partial^{\alpha}}{\partial x^{\alpha}}\left(  \varphi_{x_{N}}\left(
x^{\prime}-y^{\prime}\right)  \right)   &  =\frac{\partial}{\partial x_{i}%
}\frac{\partial^{\beta}}{\partial x^{\beta}}\left(  \varphi_{x_{N}}\left(
x^{\prime}-y^{\prime}\right)  \right) \\
&  =\frac{\partial}{\partial x_{i}}\left(  \frac{1}{x_{N}^{|\beta|+N-1}}%
\psi^{\beta}\left(  \frac{x^{\prime}-y^{\prime}}{x_{N}}\right)  \right)  ,
\end{align*}
where in the last equality we have used the induction hypothesis. We now
distinguish two cases. If $i\in\{1,\ldots,N-1\}$, then
\begin{equation*}
\frac{\partial^{\alpha}}{\partial x^{\alpha}}\left(  \varphi_{x_{N}}\left(
x^{\prime}-y^{\prime}\right)  \right)  =\frac{1}{x_{N}^{|\beta|+N}}%
\partial_{i}\psi^{\beta}\left(  \frac{x^{\prime}-y^{\prime}}{x_{N}}\right)
=\frac{1}{x_{N}^{|\alpha|+N-1}}\psi^{\alpha}\left(  \frac{x^{\prime}%
-y^{\prime}}{x_{N}}\right)  ,
\end{equation*}
where $\psi^{\alpha}(y^{\prime}):=\partial_{i}\psi^{\beta}(y^{\prime})$. Note
that in this case we trivially have the identity $\int_{\mathbb{R}^{N-1}}\psi^{\alpha}(y^{\prime})\,dy^{\prime}=0$. On the other hand, if $i=N$, then
\begin{align*}
\frac{\partial^{\alpha}}{\partial x^{\alpha}}\left(  \varphi_{x_{N}}\left(
x^{\prime}-y^{\prime}\right)  \right)   &  =\frac{1}{x_{N}^{|\beta|+N}}\left[
-(|\beta|+N-1)\psi^{\beta}\left(  \frac{x^{\prime}-y^{\prime}}{x_{N}}\right)
-\nabla_{\shortparallel}\psi^{\beta}\left(  \frac{x^{\prime}-y^{\prime}}%
{x_{N}}\right)  \cdot\frac{x^{\prime}-y^{\prime}}{x_{N}}\right] \\
&  =\frac{1}{x_{N}^{|\alpha|}}\psi_{x_{N}}^{\alpha}(x^{\prime}-y^{\prime}),
\end{align*}
where $\psi^{\alpha}(y^{\prime}):=-(|\beta|+N-1)\psi^{\beta}(y^{\prime
})-\nabla_{\shortparallel}\psi^{\beta}\left(  y^{\prime}\right)  \cdot
y^{\prime}$.  Writing 
\begin{equation*}
\nabla_{\shortparallel}\psi^{\beta}\left(  y^{\prime
}\right)  \cdot y^{\prime}=\operatorname{div}_{y^{\prime}}(y^{\prime}
\psi^{\beta})-(N-1)\psi^{\beta} 
\end{equation*}
and using the induction hypothesis then shows that 
\begin{equation*}
\int_{\mathbb{R}^{N-1}}\psi^{\alpha}(y^{\prime})\,dy^{\prime}=\int%
_{\mathbb{R}^{N-1}}-|\beta|\psi^{\beta}(y^{\prime})\,dy^{\prime}=0.
\end{equation*}
Thus, the result holds for $\alpha$, and so by finite induction this completes the proof.
\end{proof}

Finally, we prove an estimate for a special type of convolution.

\begin{proposition}\label{proposition structure}
Let $0<a<b$ and let $\phi\in L^{\infty}(\mathbb{R}^{N-1})$ be such that $\operatorname*{supp}\phi\subseteq B^{\prime}(0,ab^{-1})$ and
\begin{equation}\label{phi average zero}
\int_{\mathbb{R}^{N-1}}\phi(y^{\prime})\,dy^{\prime}=0.
\end{equation}
Let $1 < p < \infty$ and $f\in L_{\operatorname*{loc}}^{p}(\mathbb{R}^{N-1})$, and consider the function $v: \mathbb{R}^{N-1}\times(0,b) \to \mathbb{R}$ defined by
\begin{equation*}
v(x):=\frac{1}{x_{N}^{N}}\int_{\mathbb{R}^{N-1}}f(y^{\prime})\phi\left(
\frac{x^{\prime}-y^{\prime}}{x_{N}}\right)  \,dy^{\prime}.
\end{equation*}
Then there exists a constant $c=c(a,b,N,p,\Vert \phi \Vert_{L^\infty})>0$ such that
\begin{equation*}
\int_{\mathbb{R}^{N-1}\times(0,b)}|v(x)|^{p}dx\leq c\int_{\mathbb{R}^{N-1}%
}\int_{B^{\prime}(x^{\prime},a)}\frac{|f(y^{\prime})-f(x^{\prime})|^{p}%
}{|x^{\prime}-y^{\prime}|^{N+p-2}}\,dy^{\prime}dx^{\prime}.
\end{equation*}

\end{proposition}

\begin{proof}
In view of \eqref{phi average zero}, we can rewrite
\begin{equation*}
v(x)=\frac{1}{x_{N}^{N}}\int_{\mathbb{R}^{N-1}}(f(y^{\prime})-f(x^{\prime
}))\phi\left(  \frac{x^{\prime}-y^{\prime}}{x_{N}}\right)  \,dy^{\prime}.
\end{equation*}
Since $\operatorname*{supp}\phi\subseteq B^{\prime}(0,ab^{-1})$, we have
\begin{equation*}
|v(x)|\leq\frac{\Vert \phi \Vert_{L^\infty}}{x_{N}^{N}}\int_{B^{\prime}(x^{\prime},x_{N}ab^{-1}%
)}|f(y^{\prime})-f(x^{\prime})|\,dy^{\prime}.
\end{equation*}
Raising both sides to the power $p$ and using H\"{o}lder's inequality, we
obtain%
\begin{align*}
|v(x)|^{p}  &  \leq\frac{cx_{N}^{(N-1)(p-1)}}{x_{N}^{Np}}\int_{B^{\prime
}(x^{\prime},x_{N}ab^{-1})}|f(y^{\prime})-f(x^{\prime})|^{p}dy^{\prime}\\
&  =\frac{c}{x_{N}^{N+p-1}}\int_{B^{\prime}(x^{\prime},x_{N}ab^{-1}%
)}|f(y^{\prime})-f(x^{\prime})|^{p}dy^{\prime}.
\end{align*}
Integrating both sides over $\mathbb{R}^{N-1}\times(0,b)$ and using Tonelli's
theorem shows that
\begin{align*}
\int_{\mathbb{R}^{N-1}\times(0,b)}|v(x)|^{p}dx  &  \leq c\int_{\mathbb{R}%
^{N-1}}\int_{0}^{b}\frac{c}{x_{N}^{N+p-1}}\int_{B^{\prime}(x^{\prime}%
,x_{N}ab^{-1})}|f(y^{\prime})-f(x^{\prime})|^{p}dy^{\prime}dx_{N}dx^{\prime}\\
&  =c\int_{\mathbb{R}^{N-1}}\int_{B^{\prime}(x^{\prime},a)}|f(y^{\prime
})-f(x^{\prime})|^{p}\int_{ba^{-1}|x^{\prime}-y^{\prime}|}^{b}\frac{1}%
{x_{N}^{N+p-1}}dx_{N}dy^{\prime}dx^{\prime}\\
&  \leq c\int_{\mathbb{R}^{N-1}}\int_{B^{\prime}(x^{\prime},a)}\frac
{|f(y^{\prime})-f(x^{\prime})|^{p}}{|x^{\prime}-y^{\prime}|^{N+p-2}%
}\,dy^{\prime}dx^{\prime},
\end{align*}
which is the desired bound.
\end{proof}

\subsection{High-order integration by parts on intervals}

We now record an elementary calculus result that will play a key role in our trace analysis.

\begin{proposition}\label{proposition by parts}
Given a function $f\in C^{m}([0,b])$, for $b >0$ and $m\in\mathbb{N}$, we have that
\begin{equation}
\sum_{k=0}^{m-1}\frac{(-1)^{k}}{k!}(f^{(k)}(b)+(-1)^{k+1}f^{(k)}(0))\left(
\frac{b}{2}\right)  ^{k}=\frac{1}{(m-1)!}\int_{0}^{b}f^{(m)}(t)\left(
\frac{b}{2}-t\right)  ^{m-1}dt \label{formula by parts}%
\end{equation}
and%
\begin{equation}
f(t)=\sum_{k=0}^{m-1}\frac{(-1)^{k}}{k!}f^{(k)}(b)(b-t)^{k}+\frac{1}%
{(m-1)!}\int_{t}^{b}f^{(m)}(\tau)\left(  t-\tau\right)  ^{m-1}d\tau.
\label{formula taylor inverted}%
\end{equation}
In particular, for $m=2$,%
\begin{gather}
f(b)-f(0)-(f^{\prime}(b)+f^{\prime}(0))\frac{b}{2}=\int_{0}^{b}f^{\prime
\prime}(t)\left(  \frac{b}{2}-t\right)  \,dt,\label{formula by parts m=2}\\
f(t)=f(b)-f^{\prime}(b)(b-t)+\frac{1}{2}\int_{t}^{b}f^{\prime\prime}%
(\tau)\left(  t-\tau\right)  \,d\tau. \label{formula taylor inverted m=2}%
\end{gather}

\end{proposition}

\begin{proof}
To prove the first identity, we simply apply the fundamental theorem of calculus and then integrate by parts $m$-times:
\begin{align*}
f(b)-f(0)  &  =\int_{0}^{b}f^{\prime}(t)\,dt=-\int_{0}^{b}f^{\prime}(t)\left(
\frac{b}{2}-t\right)  ^{\prime}\,dt\\
&  =(f^{\prime}(b)+f^{\prime}(0))\frac{b}{2}+\int_{0}^{b}f^{\prime\prime
}(t)\left(  \frac{b}{2}-t\right)  \,dt\\
&  =(f^{\prime}(b)+f^{\prime}(0))\frac{b}{2}-(f^{\prime\prime}(b)-f^{\prime
\prime}(0))\left(  \frac{b}{2}\right)  ^{2}+\frac{1}{2}\int_{0}^{b}%
f^{\prime\prime\prime}(t)\left(  \frac{b}{2}-t\right)  ^{2}dt\\
&  =\sum_{k=1}^{m-1}\frac{(-1)^{k-1}}{k!}(f^{(k)}(b)+(-1)^{k+1}f^{(k)}%
(0))\left(  \frac{b}{2}\right)  ^{k}+\frac{1}{(m-1)!}\int_{0}^{b}%
f^{(m)}(t)\left(  \frac{b}{2}-t\right)  ^{m-1}dt.
\end{align*}
To obtain the second estimate, we apply Taylor's formula to the function
$g(s):=f(b-s)$ to obtain
\begin{equation*}
g(s)=\sum_{k=0}^{m-1}\frac{1}{k!}g^{(k)}(s)s^{k}+\frac{1}{(m-1)!}\int_{0}%
^{s}g^{(m)}(r)\left(  s-r\right)  ^{m-1}dr.
\end{equation*}
Taking $s=b-t$ gives
\begin{equation*}
f(t)=\sum_{k=0}^{m-1}\frac{(-1)^{k}}{k!}f^{(k)}(b)(b-t)^{k}+\frac{1}%
{(m-1)!}\int_{0}^{b-t}(-1)^{m}f^{(m)}(b-r)\left(  b-t-r\right)  ^{m-1}dr.
\end{equation*}
We then make the change of variables $b-r=\tau$ to conclude.
\end{proof}

\subsection{Seminormed spaces}\label{subsection seminormed spaces}

Here we recall a few useful facts about seminormed spaces (we refer to Taylor's book \cite{taylor_functional} for a more thorough discussion). Given a real vector space $X$, a \emph{seminorm} is a function $q:X\rightarrow
\lbrack0,\infty)$ with the properties that $q(x+y)\leq q(x)+q(y)$ for all $x,y\in X$ and $q(tx)=|t|q(x)$ for every $t\in\mathbb{R}$ and $x\in X$. Note that $q(0)=0$, but in general the subspace $K:=\{x\in X\,|\,q(x)=0\}$ may contain more than the zero vector. The family of \emph{balls} $\{x\in X\,|\,q(x)<r\}$, for $r>0$, is a local base for a topology that turns $X$ into a locally convex topological vector space.  Note that, following \cite{taylor_functional}, we do not require locally convex spaces to be Hausdorff.

A sequence $\{x_{n}\}_{n \in \mathbb{N}}$ in $X$ is a \emph{Cauchy sequence} if for every $\varepsilon>0$, there exists $n_{\varepsilon}\in\mathbb{N}$ such that 
\begin{equation*}
q(x_{n}-x_{m})<\varepsilon
\end{equation*}
for all $n,m\geq n_{\varepsilon}$. We say that $X$ is \emph{sequentially complete }if every Cauchy sequence converges in $X$, that is, if for every Cauchy sequence $\{x_{n}\}_{n \in \mathbb{N}}$ there exists $x\in X$ such that every $\varepsilon>0$, there exists $n_{\varepsilon}\in\mathbb{N}$ such that
\begin{equation*}
q(x_{n}-x)<\varepsilon
\end{equation*}
for all $n\geq n_{\varepsilon}$. Since $X$ is not Hausdorff in general, the limit $x$ will not be unique. Indeed, if $x_{n}\rightarrow x$ and $z\in K$, then $x_{n}\rightarrow x+z$ as well.

One can prove (see for instance Theorem 3.8-C in \cite{taylor_functional}) that a linear functional $T:X\rightarrow\mathbb{R}$ is
continuous if and only if there exists a constant $c>0$ such that
\begin{equation}\label{continuity linear functional}
|T(x)|\leq cq(x) \text{ for all } x \in X.
\end{equation}
We write 
\begin{equation*}
 X^\prime = \{ T : X \to \mathbb{R} \,|\,\, T \text{ is continuous}\}
\end{equation*}
for the linear space of continuous linear functionals.  For a linear map $T: X \to \mathbb{R}$ we define
\begin{equation}\label{seminorm T}
\Vert T\Vert_{X^{\prime}}:=\sup\{|T(x)|\,|\,\,q(x)\leq1\} \in [0,\infty]. 
\end{equation}
The following lemma records some essential properties of this map.

\begin{lemma}\label{T norm lemma}
The following hold.
\begin{enumerate}
 \item A linear map $T: X \to \mathbb{R}$ is continuous if and only if $\Vert T\Vert_{X^{\prime}}<\infty$.  Moreover, if $T \in X^\prime$, then 
\begin{equation}\label{T norm estimate}
 |T(x)| \le q(x) \Vert T \Vert_{X^\prime} \text{ for all }x \in X.
\end{equation}

 \item $\Vert \cdot \Vert_{X^{\prime}}$ defines a norm on $X^\prime$.
\end{enumerate}
\end{lemma}
\begin{proof}
We begin with the proof of the first item.  The ``only if'' part is obvious thanks to \eqref{continuity linear functional}.  Suppose, then, that $\Vert T\Vert_{X^{\prime}} < \infty$.   By the positive homogeneity of $q$, for every $x\in X \backslash K$ we have that $y=x/q(x)$ satisfies $q(y)=1$, and so the linearity of $T$ implies that 
\begin{equation*}
|T(x)|  = q(x)|T(y)|  \le q(x) \Vert T \Vert_{X^\prime}.
\end{equation*} 
It remains to consider the case $x \in K$.  In this case $0 = q(s x) \le 1$ for every $s >0$, and hence $s |T(x)| = |T(sx)| \le \Vert T\Vert_{X^\prime}$ for every $s>0$.  Sending $s \to \infty$ then shows that $T(x) =0$, and thus that $|T(x)|  =0= q(x) \Vert T \Vert_{X^\prime}$.  We have now shown that \eqref{T norm estimate} holds, which in particular guarantees that $T$ is continuous.   This completes the proof of the first item.

To prove the second item, we first note that $X' \ni T \mapsto \Vert T \Vert_{X^\prime}$ clearly defines a seminorm, so it suffices to prove that $\Vert T \Vert_{X^\prime} =0$ implies that $T =0$.  If $\Vert T\Vert_{X^{\prime}}=0$, then \eqref{T norm estimate} implies that $T(x) =0$ for all $x \in X$, and so $T =0$.  This proves the second item.
\end{proof}

It is often convenient to use a seminormed space $X$ to produce a normed space; to do so, we take the quotient over the subspace $K$. More precisely, we define an equivalence relation $\sim$ on $X$ by setting, for $x,y \in X$,  $x \sim y$ if $x-y\in K$.  Then the quotient vector space $X/K$ is defined to be the set
\begin{equation*}
 X/K = \{ [x] \,|\,\, x \in X\},
\end{equation*}
where $[x]:=\{y\in X\,|\,\,x\sim y\}$, endowed with the vector space structure $\alpha [x] + \beta [y] := [\alpha x + \beta y]$.  The quotient space $X/K$ is a normed space when endowed with the norm
\begin{equation}\label{norm quotient}
\Vert\lbrack x]\Vert_{X/K}:=q(x).
\end{equation}
In view of \eqref{norm quotient}, it follows that a sequence $\{[x_{n}]\}_{n\in \mathbb{N}}$
in $X/K$ is a Cauchy sequence if and only if $\{x_{n}\}_{n \in \mathbb{N}}$ is a Cauchy
sequence in $X$. Thus, the completeness of $X/K$ is equivalent to the
sequential completeness of $X$.

Observe that if $T:X\rightarrow\mathbb{R}$ is continuous and if $x\sim y$, then $x=y+z$ for some $z\in K$.  The estimate \eqref{continuity linear functional} guarantees that $T(z)=0$, and so $T(x)=T(y)$. Thus, the functional $T_{1}:X/K\rightarrow
\mathbb{R}$ given by
\begin{equation*}
T_{1}([x]):=T(x)
\end{equation*}
is well-defined, linear and continuous, and by \eqref{seminorm T} and \eqref{norm quotient},
\begin{align*}
\Vert T_{1}\Vert_{(X/K)^{\prime}}  &  =\sup\{|T_{1}([x])|\,|\,\,\Vert\lbrack
x]\Vert_{X/K}\leq1\}\\
&  =\sup\{|T(x)|\,|\,\,q(x)\leq1\}=\Vert T\Vert_{X^{\prime}}<\infty.
\end{align*}
Conversely, if $T_{1}:X/K\rightarrow\mathbb{R}$ is linear and continuous, then%
\begin{equation*}
|T_{1}([x])|\leq\Vert T_{1}\Vert_{(X/K)^{\prime}}\Vert\lbrack x]\Vert
_{X/K}=\Vert T_{1}\Vert_{(X/K)^{\prime}}q(x)
\end{equation*}
for all $x\in X$. Thus, if we define the functional $T: X \to \mathbb{R}$ via $T(x):=T_{1}([x])$, then the previous inequality implies that
\begin{equation*}
\Vert T\Vert_{X^{\prime}}=\sup\{|T_{1}([x])|\,|\,\,q(x)\leq1\}\leq\Vert
T_{1}\Vert_{(X/K)^{\prime}},
\end{equation*}
and so Lemma \ref{T norm lemma} guarantees that $T$ is continuous.
This shows that
\begin{equation*}
\Vert T_{1}\Vert_{(X/K)^{\prime}}=\Vert T\Vert_{X^{\prime}},
\end{equation*}
that is, that there is an isometric isomorphism between the topological dual
$X^{\prime}$ of $X$ and the topological dual $(X/K)^{\prime}$ of $X/K$.

In particular, if $X/K$ is reflexive and $\{x_{n}\}_{n\in \mathbb{N}}$ is a bounded sequence in $X$, i.e. there exists $c>0$ such that  $q(x_{n})\leq c$ for all $n\in\mathbb{N}$, then \eqref{norm quotient} tells us that $\Vert\lbrack x_{n}]\Vert_{X/K}\leq c$ for all $n \in \mathbb{N}$. It follows by Kakutani's theorem (see, for instance, Theorem 3.17 of \cite{brezis_functional}) that there exists a subsequence $\{[x_{n_{k}}]\}_{k\in \mathbb{N}}$  and $[x]\in X/K$ such that $[x_{n_{k}}]\rightharpoonup
\lbrack x]$.  That is,
\begin{equation*}
T_{1}([x_{n_{k}}])\rightarrow T_{1}([x])
\end{equation*}
for every $T_{1}\in(X/K)^{\prime}$. In view of the previous discussion, this
is equivalent to saying that
\begin{equation*}
T(x_{n_{k}})\rightarrow T(x)
\end{equation*}
for every $T\in X^{\prime}$.  By slight abuse of notation, we will write this as $x_{n_{k}%
}\rightharpoonup x$.

\section{Screened homogeneous fractional Sobolev spaces}\label{section spaces}

In this section we introduce and study the properties of a class of screened homogeneous fractional Sobolev spaces.  These spaces play an essential role in the main trace theory results of the paper, but they are interesting in their own right.

\subsection{Homogeneous Sobolev spaces}

Before introducing the screened spaces we recall the usual homogeneous Sobolev spaces, starting with the spaces of integer order.  Given an open set $ U  \subseteq\mathbb{R}^{N}$, $m\in\mathbb{N}$, and $1\leq p\leq\infty$, the homogeneous Sobolev space $\dot{W}^{m,p}( U )$ is defined as the space of all real-valued functions $u\in L_{\operatorname*{loc}}^{1}( U )$ whose distributional derivatives of order $m$ belong to $L^{p}( U )$. The space $\dot{W}^{m,p}( U )$ is endowed with the seminorm
\begin{equation}\label{seminorm sobolev}
u\mapsto\Vert\nabla^{m}u\Vert_{L^{p}( U )}, 
\end{equation}
where we recall that $\nabla^{m}$ the vector of all partial derivatives $\partial^\alpha$ for multi-indices $\alpha \in\mathbb{N}_{0}^{N}$ with $|\alpha|=m$. In view of the discussion in Section \ref{subsection seminormed spaces}, the seminorm \eqref{seminorm sobolev} turns $\dot{W}^{m,p}( U )$ into a locally convex topological vector space.  Also, a linear functional $T:\dot{W}^{m,p}( U )\rightarrow\mathbb{R}$ is continuous if and only if there
exists a constant $c>0$ such that
\begin{equation}\label{continuity T homogeneous sobolev}
|T(u)|\leq c\Vert\nabla^{m}u\Vert_{L^{p}( U )}
\end{equation}
for all $u\in\dot{W}^{m,p}( U )$.  Moreover, $\Vert\nabla^{m}u\Vert_{L^{p}( U )}=0$ if and only if $u$ is a polynomial of degree less than or equal to $m-1$ in each connected component of $ U $.

Let $\mathcal{P}_{m-1}$ denote the set of functions on $ U $ that agree on each connected component of $ U $ with a polynomial from $\mathbb{R}^{N}$ into $\mathbb{R}$ of degree less than or equal to $m-1$.  Then 
the quotient space $\dot{W}^{m,p}( U )/\mathcal{P}_{m-1}$ is a Banach space with the norm
\begin{equation*}
\Vert\lbrack u]\Vert_{\dot{W}^{m,p}( U )/\mathcal{P}_{m-1}}:=\Vert
\nabla^{m}u\Vert_{L^{p}( U )}.
\end{equation*}
When $m=1$, the family $\mathcal{P}_{0}$ is just the family of functions that are constant in each connected component of $ U $,  and by slight abuse of notation we write $\dot{W}^{1,p}( U )/\mathbb{R}$ in place of $\dot{W}^{1,p}( U )/\mathcal{P}_{0}$.

The homogeneous spaces can be extended to non-integer order $0 < s < 1$ by use of the seminorm
\begin{equation}\label{homogeneous seminorm}
 |u|_{\dot{W}^{s,p}( U )} := \left(\int_ U  \int_ U  \frac{|u(y)-u(x)|^p}{|y-x|^{sp+N}} dy dx  \right)^{1/p}.
\end{equation}
For $1 \le p < \infty$ and $0 < s < 1$ the seminormed \emph{homogeneous fractional Sobolev space}  $\dot{W}^{s,p}( U )$ consists of those functions $u \in L^1_{\operatorname*{loc}}( U )$ for which $|u|_{\dot{W}^{s,p}( U )} < \infty$.  As above, we can use this space to generate a Banach space $\dot{W}^{s,p}( U ) / \mathbb{R}$ by quotienting out by the functions that are constant in each connected component.  The inhomogeneous fractional Sobolev space $W^{s,p}( U )$ is given by $W^{s,p}( U ) = L^p( U ) \cap \dot{W}^{s,p}( U )$ with norm $\Vert f \Vert_{W^{s,p}( U )} = \Vert f \Vert_{L^{p}( U )} + \vert f \vert_{W^{s,p}( U )}$, which makes the space Banach.  We refer to \cite{adams-fournier2003book}, \cite{bergh-lofstrom-1976book}, \cite{besov-ilin-nikolskii1979book}, \cite{besov-ilin-nikolskii1978book},  \cite{burenkov1998book}, \cite{dinezza-palatucci-valdinoci2012}, \cite{grisvard2011book},  \cite{leoni2017book}, \cite{mazya2011book}, \cite{necas2012book}, \cite{peetre1976book},  \cite{triebel1995book} for an exhaustive survey of what is known about these spaces.

In order to prove our trace results we will need to employ a version of the fundamental theorem of calculus.  We now record a result that will allow us to employ such a tool in the spaces $\dot{W}^{m,p}( U )$.

\begin{theorem}\label{theorem AC}
Let $ U  \subseteq\mathbb{R}^{N}$ be an open set, $m\in\mathbb{N}$, and $1 \leq p<\infty$. If $u \in \dot{W}^{m,p}( U )$, then $u$ has a representative $\overline{u}$ that is absolutely continuous together with all its derivatives of order up to $m-1$ on $\mathcal{L}^{N-1}$-a.e. line segments of $ U $ that are parallel to the coordinate axes and whose classical partial derivatives of order $m$ belong to $L^{p}( U )$.  Moreover the classical partial derivatives of order $m$ of $\overline{u}$ agree $\mathcal{L}^{N}$-a.e. with the weak derivatives of $u$.
\end{theorem}
\begin{proof}

The proof is a modification of Theorem 11.45 from the book \cite{leoni2017book}, and as such we will refer frequently to other results from the book.  Throughout the proof we use the following notational convention: given $x_{i}^{\prime}\in\mathbb{R}^{N-1}$ and $x_i \in \mathbb{R}$ we write $(x_i',x_i) \in \mathbb{R}^N$ for the vector obtained by placing $x_i$ in the $i^{th}$ component and the components of $x_i'$ in the remaining $N-1$ components.  Then for a set $E\subseteq\mathbb{R}^{N}$, we write
\begin{equation*}
E_{x_{i}^{\prime}}:=\{x_{i}\in\mathbb{R}:\,(x_{i}^{\prime},x_{i})\in E\}.
\end{equation*}
Moreover, if $v: U \rightarrow\mathbb{R}$ is Lebesgue integrable, with a slight abuse of notation, if $ U _{x_{i}^{\prime}}$ is empty, we set
\begin{equation*}
\int_{ U _{x_{i}^{\prime}}}v(x_{i}^{\prime},x_{i})\,dx_{i}:=0, 
\end{equation*}
so that by Fubini's theorem
\begin{equation}\label{s funny notation}
\int_{ U }v(x)\,dx = \int_{\mathbb{R}^{N-1}} \Bigl(\int_{ U _{x_{i}^{\prime}}}v(x_{i}^{\prime},x_{i}) \,dx_{i}\Bigr)dx_{i}^{\prime}.
\end{equation}

Assume now that $u\in\dot{W}^{m,p}( U )$. Consider a sequence of standard mollifiers $\{\varphi_{\varepsilon}\}_{\varepsilon>0}$ and for every $\varepsilon>0$ define $u_{\varepsilon}:=u\ast\varphi
_{\varepsilon}$ in $ U _{\varepsilon}:=\{x\in U :\,\operatorname*{dist}(x,\partial U )>\varepsilon\}$. By Lemma 11.25 in \cite{leoni2017book} we have that
\begin{equation*}
\lim_{\varepsilon\rightarrow0^{+}}\int_{ U _{\varepsilon}}\Vert\nabla
^{m}u_{\varepsilon}(x)-\nabla^{m}u(x)\Vert^{p}\,dx=0.
\end{equation*}
It follows by Fubini's theorem and \eqref{s funny notation} that for all
$i=1,\ldots,N$,
\begin{equation*}
\lim_{\varepsilon\rightarrow0^{+}}\int_{\mathbb{R}^{N-1}}\Bigl(\int_{( U _{\varepsilon})_{x_{i}^{\prime}}} \Vert\nabla^{m}u_{\varepsilon}(x_{i}^{\prime},x_{i})-\nabla^{m}u(x_{i}^{\prime},x_{i}) \Vert^{p} \,dx_{i}\Bigr)\,dx_{i}^{\prime}=0,
\end{equation*}
where $( U _{\varepsilon})_{x_{i}^{\prime}}:=\{x_{i}\in\mathbb{R} :\ (x_{i}^{\prime},x_{i}) \in  U _{\varepsilon}\}$, and so we may find a subsequence $\{\varepsilon_{n}\}_{n \in \mathbb{N}}$ such that for all $i=1,\ldots,N$ and for $\mathcal{L}^{N-1}$ a.e. $x_{i}^{\prime}\in\mathbb{R}^{N-1}$,
\begin{equation}\label{s good slice}
\lim_{n\rightarrow\infty}\int_{( U _{\varepsilon_{n}})_{x_{i}^{\prime}}} \Vert\nabla^{m}u_{\varepsilon_{n}}(x_{i}^{\prime},x_{i})-\nabla^{m} u(x_{i}^{\prime},x_{i})\Vert^{p}\,dx_{i}=0.
\end{equation}
We set $u_{n}:=u_{\varepsilon_{n}}$ and define the sets
\begin{align*}
E  & :=\{x\in U :\,\lim\limits_{n\rightarrow\infty}u_{n}(x)\text{ exists in }\mathbb{R}\},\\
E_{i,l}  & :=\{x\in U :\,\lim\limits_{n\rightarrow\infty}\partial_{i}^{l}u_{n}(x)\text{ exists in }\mathbb{R}\}
\end{align*}
for $i=1,\ldots,N$ and $l=1,\ldots,m-1$. We then define $\bar{u}, v_{i,l} :  U  \to \mathbb{R}$ via
\begin{align*}
\overline{u}(x)  & :=\left\{
\begin{array}
[c]{ll}%
\lim\limits_{n\rightarrow\infty}u_{n}(x) & \text{if }x\in E,\\
0 & \text{otherwise,}%
\end{array}
\right.  \\
v_{i,l}(x)  & :=\left\{
\begin{array}
[c]{ll}%
\lim\limits_{n\rightarrow\infty}\partial_{i}^{l}u_{n}(x) & \text{if }x\in
E_{i,l},\\
0 & \text{otherwise.}%
\end{array}
\right.
\end{align*}
By standard properties of mollifiers (see for instance Theorem C.16 in \cite{leoni2017book}) we have that $\{u_{n}\}_{n \in \mathbb{N}}$ converges pointwise to $u$ at every Lebesgue point of $u$, so the set $E$ contains every Lebesgue point of $u$. It follows from Corollary B.119 in \cite{leoni2017book} that $\mathcal{L}^{N}( U \setminus E)=0$, and so $\overline{u}$ is a representative of $u$.  Similarly, $\{\partial_{i}^{l} u_{n}\}_{n \in \mathbb{N}}$ converges pointwise to $\partial_{i}^{l}u$ at every Lebesgue
point of $\partial_{i}^{l}u$ and $\mathcal{L}^{N}( U \setminus E_{i,l})=0$ for every $n=1,\ldots,m-1$.

It remains to prove that $\overline{u}$ has the desired properties.  To this end we define
\begin{equation*}
F:=E\cap\bigcap_{i=1}^{N}\bigcap_{l=1}^{m-1}E_{i,l}.
\end{equation*}
By Fubini's theorem and \eqref{s funny notation} for every $i=1,\ldots,N$ we
have that%
\begin{equation*}
\int_{\mathbb{R}^{N-1}}\biggl(\int_{ U _{x_{i}^{\prime}}}\Vert\nabla
^{m}u(x_{i}^{\prime},x_{i})\Vert^{p}dx_{i}\biggr)\,dx_{i}^{\prime}<\infty
\end{equation*}
and
\begin{equation*}
\int_{\mathbb{R}^{N-1}}\mathcal{L}^{1}(\{x_{i}\in U _{x_{i}^{\prime} } :\,(x_{i}^{\prime},x_{i})\notin F\})\,dx_{i}^{\prime}=0,
\end{equation*}
and so we may find a set $N_{i}\subset\mathbb{R}^{N-1}$, with $\mathcal{L}^{N-1}(N_{i})=0$, such that for all $x_{i}^{\prime}\in\mathbb{R}^{N-1}\setminus N_{i}$ for which $ U _{x_{i}^{\prime}}$ is nonempty we have
that
\begin{equation*}
\int_{ U _{x_{i}^{\prime}}}\Vert\nabla^{m}u(x_{i}^{\prime},x_{i}) \Vert^{p}dx_{i}<\infty
\end{equation*}
and \eqref{s good slice} holds for all $i=1,\ldots,N$ and $(x_{i}^{\prime}, x_{i})\in F$ for $\mathcal{L}^{1}$ a.e. $x_{i}\in U _{x_{i}^{\prime}}$. 

Fix any such $x_{i}^{\prime}$ and let $I\subseteq U _{x_{i}^{\prime}}$ be a maximal interval. Fix $t_{0}\in I$ such that $(x_{i}^{\prime},t_{0})\in F$ and let $t\in I$. For all $n$ large, the interval of endpoints $t$ and $t_{0}$ is contained in $( U _{\varepsilon_{n}})_{x_{i}^{\prime}}$ and so, since $u_{n} \in C^{\infty}( U _{\varepsilon_{n}})$, by Taylor's formula,
\begin{equation*}
u_{n}(x_{i}^{\prime},t)   =u_{n}(x_{i}^{\prime},t_{0})
+\sum_{k=1}^{m-1} \frac{1}{k!}\partial_{i}^{k}u_{n}(x_{i}^{\prime},t_{0})(t-t_{0})^{k}
+\frac{1}{(m-1)!} \int_{t_{0}}^{t}(t-s)^{m-1}\partial_{i}^{m}u_{n}(x_{i}^{\prime},s)\,ds, 
\end{equation*}
and
\begin{align*}
\partial_{i}^{j}u_{n}(x_{i}^{\prime},t)   =\partial_{i}^{j}u_{n}(x_{i}^{\prime},t_{0})
&+\sum_{k=1}^{m-j-1}\frac{1}{k!}\partial_{i}^{k+j} u_{n}(x_{i}^{\prime},t_{0})(t-t_{0})^{k} \\
& +\frac{1}{(m-j-1)!}\int_{t_{0}}^{t}(t-s)^{m-j-1}\partial_{i}^{m}u_{n}(x_{i}^{\prime},s)\,ds,
\end{align*}
for all $j=1,\ldots,m-1$. Since $(x_{i}^{\prime},t_{0})\in F$, we have $u_{n}(x_{i}^{\prime},t_{0})\rightarrow\overline{u}(x_{i}^{\prime},t_{0})\in\mathbb{R}$ and $\partial_{i}^{k}u_{n}(x_{i}^{\prime},t_{0})\rightarrow v_{i,k}(x_{i}^{\prime},t_{0})\in\mathbb{R}$ for all $k=1,\ldots,m-1$. On the
other hand, by \eqref{s good slice},
\begin{equation*}
\lim_{n\rightarrow\infty}\int_{t_{0}}^{t}|(t-s)^{m-1}(\partial_{i}^{m} u_{n}(x_{i}^{\prime},s) - \partial_{i}^{m}u(x_{i}^{\prime},s))|\,ds=0.
\end{equation*}
Hence as $n\rightarrow\infty$,
\begin{equation*}
u_{n}(x_{i}^{\prime},t)   \rightarrow\overline{u}(x_{i}^{\prime},t_{0})
+\sum_{k=1}^{m-1}\frac{1}{k!}v_{i,k}(x_{i}^{\prime},t_{0})(t-t_{0})^{k}
+\frac{1}{(m-1)!}\int_{t_{0}}^{t}(t-s)^{m-1}\partial_{i}^{m}u(x_{i}^{\prime},s)\,ds, 
\end{equation*}
and
\begin{equation*}
\partial_{i}^{j}u_{n}(x_{i}^{\prime},t)   \rightarrow v_{i,j}(x_{i}^{\prime
},t_{0})+\sum_{k=1}^{m-j-1}\frac{1}{k!}v_{i,k+j}(x_{i}^{\prime},t_{0}%
)(t-t_{0})^{k}+\frac{1}{(m-j-1)!}\int_{t_{0}}^{t}(t-s)^{m-j-1}\partial_{i}%
^{m}u(x_{i}^{\prime},s)\,ds
\end{equation*}
for all $j=1,\ldots,m-1$. Note that by the definition of $E$ and $\overline{u}$, this implies, in particular, that $(x_{i}^{\prime},t)\in E\cap \bigcap_{l=1}^{m-1}E_{i,l}$ and
\begin{equation*}
\overline{u}(x_{i}^{\prime},t)   
=\overline{u}(x_{i}^{\prime},t_{0})+\sum_{k=1}^{m-1}\frac{1}{k!}v_{i,k}(x_{i}^{\prime},t_{0})(t-t_{0})^{k}
+\frac{1}{(m-1)!}\int_{t_{0}}^{t}(t-s)^{m-1}\partial_{i}^{m}u(x_{i}^{\prime},s)\,ds, 
\end{equation*}
and
\begin{equation*}
v_{i,j}(x_{i}^{\prime},t)  
=v_{i,j}(x_{i}^{\prime},t_{0})+\sum_{k=1}^{m-j-1}\frac{1}{k!}v_{i,k+j}(x_{i}^{\prime},t_{0})(t-t_{0})^{k}
+\frac{1}{(m-j-1)!}\int_{t_{0}}^{t}(t-s)^{m-j-1}\partial_{i}^{m}u(x_{i}^{\prime},s)\,ds
\end{equation*}
for \emph{all} $t\in I$ and for all $j=1,\ldots,m-1$. Hence, by Theorem 3.16 in \cite{leoni2017book}, the function $\overline{u}(x_{i}^{\prime},\cdot)$ is of class $C^{m-1}$ and its derivative of order $m-1$ is absolutely continuous in $I$ with $\partial_{i}^{j}\overline{u}(x_{i}^{\prime},t)=v_{i,j}(x_{i}^{\prime},t)$ for all $t\in I$ and $\partial_{i}^{m} \overline{u}(x_{i}^{\prime},t)=\partial_{i}^{m}u(x_{i}^{\prime},t)$ for
$\mathcal{L}^{1}$ a.e. $t\in I$.
\end{proof}

\subsection{Screened homogeneous fractional Sobolev spaces: definitions and basic properties}

In this subsection we define some new fractional Sobolev spaces and study their functional properties.  We are primarily interested in these spaces due to their role in our trace results, but they are of independent interest due to their intriguing properties.  As such, we have chosen to define them on general open sets in $\mathbb{R}^N$ rather than on the set appropriate for the trace theory, $\mathbb{R}^{N-1}$.

We begin by defining the spaces.  Let $U \subseteq \mathbb{R}^N$ be open.   We say that a function $\sigma : U \to (0,\infty]$ is a \emph{screening function} on $U$ if $\sigma$ is lower semi-continuous, in which case for each $x \in U$ we define the measurable set
\begin{equation*}
 H(x) = B(0,\sigma(x)) \cap (-x + U),
\end{equation*}
where $-x + U = \{z \in \mathbb{R}^N \;\vert\; z = -x +y \text{ for some }y \in U\}$.  Here we understand that $B(x,\infty) = \mathbb{R}^N$.  Given a screening function $\sigma : U \to (0,\infty]$, $1\leq p<\infty$, and $0<s<1$, we define the \emph{screened homogeneous fractional Sobolev space} $\W_{(\sigma)}^{s,p}(U)$ to be the space of all functions $f\in L_{\operatorname*{loc}}^{1}(U)$ such that
\begin{equation*}
|f|_{\W_{(\sigma)}^{s,p}(U)}:=\left(  \int_{U}\int_{H(x)}\frac{|f(x+h)-f(x)|^{p}}{|h|^{sp+N}}\,dh dx \right)^{1/p}<\infty.
\end{equation*}
The quantity $|\cdot|_{\W_{(\sigma)}^{s,p}(U)}$ is only a seminorm since any constant function will have seminorm zero. Note that we can rewrite 
\begin{equation*}
|f|_{\W_{(\sigma)}^{s,p}(U)} :=\left(  \int_{U}\int_{B(x,\sigma(x)) \cap U}\frac{|f(y)-f(x)|^{p}}{|y-x|^{sp+N}}\,dy dx \right)^{1/p}.
\end{equation*}
Comparing this to \eqref{homogeneous seminorm} makes evident three important features.  First, it justifies our choice of the moniker \emph{screening function}:  $\sigma$ screens the difference quotient from values of $y$ far away from $x$.  Second, it shows that we recover the standard fractional Sobolev space for any screening function satisfying $\sigma \ge \text{diam}(U)$ on $U$.  In particular,   $\W_{(\infty)}^{s,p}(U) = \dot{W}^{s,p}(U)$.  Third, this form of the seminorm establishes that $|f|_{\W_{(\sigma)}^{s,p}(U)} \le |f|_{\dot{W}^{s,p}(U)}$ and hence that 
\begin{equation*}
\dot{W}^{s,p}(U) \subseteq \W_{(\sigma)}^{s,p}(U).
\end{equation*}
We will show later in Theorems \ref{theorem strict inclusion} and \ref{theorem strict inclusion vanishing} that in general this inclusion may be strict.  

We now begin our study of the screened spaces by proving that $|f|_{\W_{(\sigma)}^{s,p}(U)}=0$ if and only if $f$ is constant in each connected component.

\begin{proposition}\label{proposition seminorm}
Suppose that $U \subseteq \mathbb{R}^N$ is open.  Let $\sigma:U\rightarrow(0,\infty]$ be a screening function, $1\leq p<\infty$, and $0<s<1$. Given $f\in L_{\operatorname*{loc}}^{p}(U)$, we have that $|f|_{\W_{(\sigma)}^{s,p}(U)}=0$ if and only if $f$ is equivalent to a constant in each connected component of $U$.
\end{proposition}
\begin{proof}
Obviously, if $f$ equals a constant almost everywhere in $U$, then $|f|_{\W_{(\sigma)}^{s,p}(U)}=0$.  Suppose, then, that $|f|_{\W_{(\sigma)}^{s,p}(U)}=0$.  We will prove that $f$ is equivalent to a constant.  Assume initially that $U$ is connected.

Since $|f|_{\W_{(\sigma)}^{s,p}(U)}=0$, there exists a Lebesgue measurable set $E\subset U$ with $\mathcal{L}^{N}(E)=0$ such that if $x\in U\setminus E$,
then
\begin{equation*}
\int_{H(x)}\frac{|f(x+h)-f(x)|^{p}}{|h|^{sp+N}}\,dh=0.
\end{equation*}
In turn, $f(x+h)-f(x)=0$ for $\mathcal{L}^{N}$ a.e. $h\in H(x)$, which implies that $f$ is equivalent to a constant in the set $B(x,\sigma(x)) \cap  U$ whenever $x \in  U \setminus E$. 

Fix $x_0 \in  U \setminus E$ and let $c$ denote the constant that $f$ is equivalent to in the set $B(x_0,\sigma(x_0)) \cap  U$.  Define the set 
\begin{equation*}
 S = \{x \in  U \;\vert\; \text{there exists } \delta >0 \text{ s.t. } f = c \text{ a.e. in } B(x,\delta) \subseteq  U \}.
\end{equation*}
By construction we have that $x_0 \in S$.  We claim that, in fact, $S =  U$.  Once this is established, it follows that $f = c$ a.e. in $ U$, which completes the proof when $ U$ is connected.

Suppose that $x \in S$.  Then there exists $\delta >0$ such that $f =c$ a.e. in $B(x,\delta) \subseteq  U$.  If $y \in B(x,\delta/2)$, then $B(y,\delta/2) \subseteq B(x,\delta)$, and so $f = c$ a.e. in $B(y,\delta/2) \subset  U$.  Thus $B(x,\delta/2) \subseteq S$, and we conclude that $S$ is open.

Suppose now that $\{x_n\}_{n \in \mathbb{N}}$ is a sequence in $S$ such that $x_n \to x \in  U$ as $n \to \infty$.  Set $R \in (0,\infty)$ according to 
\begin{equation}\label{R semicont def}
R = 
\begin{cases}
1 &\text{if } \sigma(x) = \infty \\
\sigma(x)/2 & \text{if } \sigma(x) <\infty.
\end{cases}
\end{equation}
Note that $x \in \sigma^{-1}((R,\infty])$ by construction, but $\sigma$ is lower semi-continuous, so the set $\sigma^{-1}((R,\infty])$ is open.  Consequently, we may choose $0 < r < R$ such that $B(x,r) \subseteq \sigma^{-1}((R,\infty]) \subseteq  U$.  Since $E$ is a null set, there must exist $y \in B(x,r) \backslash E$.  Then, by construction, $|x-y| < r < R < \sigma(y)$, and hence $x \in B(y,\sigma(y)) \cap  U$.  From the above analysis, we know that $f$ is equivalent to a constant $b$ in $B(y,\sigma(y)) \cap  U$, and hence $f$ is equivalent to $b$ in $B(x,\delta)$ for some $\delta >0$.  However, for $n \in \mathbb{N}$ sufficiently large we have that $x_n \in B(x,\delta)$, and since $x_n \in S$ we conclude that $b=c$.  Hence $x \in S$, and we deduce that $S$ is relatively closed in $ U$.

We have now shown that $S \subseteq  U$ is nonempty, open, and relatively closed.  Since $ U$ is connected, we conclude that $S =  U$, which completes the proof of the claim.  Now, in the general case when $ U$ is not connected, we let $\{ U_\alpha\}_{\alpha \in A}$ denote the connected components of $ U$.  Since $|f|_{\W_{(\sigma)}^{s,p}( U)}=0$, we also know that $|f|_{\W_{(\sigma)}^{s,p}( U_\alpha)}=0$ for each $\alpha \in A$.  The above analysis then shows that $f$ is equivalent to a constant in each $ U_\alpha$, which completes the proof in the general case.

\end{proof}

We next turn our attention to proving a Poincar\'{e}-type inequality for the screened spaces.  We refer to  \cite{bourgain-brezis-mironescu2002} and \cite{ponce2004} for similar inequalities in the non-screened case.  We begin with a lemma.

\begin{lemma}\label{screening radius lemma}
Let $ U \subseteq \mathbb{R}^N$ be open and let $\sigma :  U \to (0,\infty]$ be a screening function.  For each $x \in  U$ there exists $0 < r_x < \sigma(x)$ such that $B(x,2r_x) \subseteq  U$ and  $2r_x < \sigma(y)$ for all $y \in B(x,r_x)$.
\end{lemma}
\begin{proof}
Let $R \in (0,\infty)$ be given by \eqref{R semicont def}.  Then $x \in \sigma^{-1}((R,\infty])$, and since $\sigma$ is lower semi-continuous we can pick $s >0$ such that $B(x,s) \subseteq \sigma^{-1}((R,\infty])$.  Set $r_x =\frac{1}{2} \min\{s,R\} \in (0,\sigma(x))$.  Clearly $B(x,2r_x) \subseteq  U$.  For $y \in B(x,r_x) \subseteq B(x,s) \subseteq \sigma^{-1}((R,\infty])$ we then have that $2r_x \le R < \sigma(y)$, as desired.
\end{proof}

Next we prove a Poincar\'{e} inequality for small balls. 

\begin{proposition}[Poincar\'{e} inequality on small balls]\label{proposition poincare}
Suppose that $ U \subseteq \mathbb{R}^N$ is open, $1\leq p<\infty$, and $0<s<1$.    Then there exists a constant $c = c(N,s,p) >0$ such that if $f\in L_{\operatorname*{loc}}^{1}( U)$, $B(x,r) \subseteq  U$, and $E \subseteq B(x,r)$ is a Lebesgue measurable set such that $\mathcal{L}^N(E) >0$, then 
\begin{equation*}
\int_{B(x,r)}|f(y)-f_E|^{p}dy 
\leq c\frac{r^{sp+N}}{\mathcal{L}^N(E)}\int_{B(x,r)}\int_{E}\frac{|f(y)-f(z)|^{p}}{|y-z|^{sp+N}}\,dzdy,
\end{equation*}
where
\begin{equation*}
f_E:=\frac{1}{\mathcal{L}^{N}(E)}\int_{E}f(z)\,dz.
\end{equation*}
In particular, given a screening function $\sigma: U \rightarrow(0,\infty]$ and $x \in  U$, if $0 < r \le r_x$ for $r_x$ given by Lemma \ref{screening radius lemma}, then
\begin{equation*}
\int_{B(x,r)}|f(y)-f_E|^{p}dy \leq 
c\frac{r^{sp+N}}{\mathcal{L}^N(E)}\int_{B(x,r)}\int_{H(y)}\frac{|f(y+h)-f(y)|^{p}}{|h|^{sp+N}}\,dhdy.
\end{equation*}

\end{proposition}

\begin{proof}
For every $y\in B(x,r)$ we have that
\begin{equation*}
f(y)-f_E    =\frac{1}{\mathcal{L}^{N}(E)}\int_{E}(f(y)-f(z))\,dz   =\frac{1}{\mathcal{L}^{N}(E)}\int_{E}|y-z|^{(sp+N)/p}\frac{f(y)-f(z)}{|y-z|^{(sp+N)/p}})\,dz.
\end{equation*}
Hence, by H\"{o}lder's inequality,
\begin{align*}
|f(y)-f_E|^{p}  &  \leq \frac{1 }{(\mathcal{L}^N(E))^p}
\left(\int_{E}|y-z|^{(sp+N)/(p-1)} dz\right)^{p-1} 
\int_{E}\frac{|f(y)-f(z)|^{p}}{|y-z|^{sp+N}}\,dz\\
&  \leq \frac{c r^{sp+N}}{\mathcal{L}^N(E)}\int_{E}\frac{|f(y)-f(z)|^{p}}{|y-z|^{sp+N}}\,dz.
\end{align*}
Integrating both sides in $y$ over $B(x,r)$ then gives the bound
\begin{equation*}
\int_{B(x,r)}|f(y)-f_E|^{p}dy
\leq \frac{c r^{sp+N}}{\mathcal{L}^N(E)} \int_{B(x,r)}\int_{E}\frac{|f(y)-f(z)|^{p}}{|y-z|^{sp+N}}\,dzdy,
\end{equation*}
which proves the first inequality.

To prove the second inequality we observe that if $0 < r \le r_x$, then Lemma \ref{screening radius lemma} implies that $2r \le 2 r_x < \sigma(y)$ for all $y \in B(x,r)$.  Then for each $y \in B(x,r)$ we have the inclusions $E \subseteq B(x,r) \subseteq B(y,2r) \cap  U \subseteq B(y,\sigma(y)) \cap  U$, and we may thus estimate
\begin{align*}
\frac{c r^{sp+N}}{\mathcal{L}^N(E)} \int_{B(x,r)}\int_{E}\frac{|f(y)-f(z)|^{p}}{|y-z|^{sp+N}}\,dzdy & \le   
\frac{c r^{sp+N}}{\mathcal{L}^N(E)} \int_{B(x,r)}\int_{B(y,\sigma(y)) \cap  U} \frac{|f(y)-f(z)|^{p}}{|y-z|^{sp+N}
}\,dzdy\\
&  \le \frac{c r^{sp+N}}{\mathcal{L}^N(E)} \int_{B(x,r)}\int_{H(y)} \frac{|f(y)-f(y+h)|^{p}}{|h|^{sp+N}
}\,dhdy.
\end{align*}
The second inequality then follows from this and the first inequality.
\end{proof}

With the small ball Poincar\'{e} inequality in hand, we can now prove a more general version.  Although this estimate is interesting in its own right, it will play a key role in proving completeness of $\W_{(\sigma)}^{s,p}( U)/ \mathbb{R}$.

\begin{theorem}[Poincar\'{e} inequality on connected unions of small balls]\label{full_poincare}
Suppose that $ U \subseteq \mathbb{R}^N$ is open and connected, $1\leq p<\infty$, and $0<s<1$. Let  $\sigma :  U \to (0,\infty]$ be a screening function. Further suppose that there exist $x_n \in  U$ and $r_n >0$ for $n=1,\dotsc,m$ such that $r_n = r_{x_n} >0$ is given by Lemma \ref{screening radius lemma} and 
\begin{equation*}
  U = \bigcup_{n=1}^m B(x_n,r_n).
\end{equation*}
Then for every Lebesgue measurable set $E \subseteq  U$ such that $\mathcal{L}^N(E) >0$ there exists a constant $c = c( U,E,p,s) >0$ such that 
\begin{equation}\label{full_poincare_0}
\Vert  f - f_E\Vert_{L^{p}( U)} \le c |f|_{\W_{(\sigma)}^{s,p}( U)}  
\end{equation}
for all $f \in L^1_{\operatorname*{loc}}( U)$, where 
\begin{equation*}
f_E = \frac{1}{\mathcal{L}^N(E)} \int_U f(x) dx. 
\end{equation*}
\end{theorem}
\begin{proof}
First note that the triangle and H\"{o}lder inequalities provide the estimates
\begin{equation*}
\begin{split}
\Vert  f - f_E\Vert_{L^{p}( U)} & \le \Vert  f - f_{ U} \Vert_{L^{p}( U)} + |f_U - f_{E}| (\mathcal{L}^N( U))^{1/p} \\
&\le \Vert  f - f_{ U} \Vert_{L^{p}( U)} + \frac{(\mathcal{L}^{N}( U))^{1/p}}{\mathcal{L}^{N}(E)} \Vert f-f_{ U}\Vert_{L^{1}(E)} \\
&\le  \Vert  f - f_{ U}\Vert_{L^{p}( U)}  + \frac{(\mathcal{L}^N( U))^{1/p}}{(\mathcal{L}^N(E))^{1/p}} \Vert f- f_U \Vert_{L^p(E)}  \\
&\le \left(1 +  \frac{(\mathcal{L}^N( U))^{1/p}}{(\mathcal{L}^N(E))^{1/p}} \right) \Vert  f - f_{ U}\Vert_{L^{p}( U)}.
\end{split}
\end{equation*}
Thus, in order to prove the estimate \eqref{full_poincare_0}, it suffices to prove that
\begin{equation}\label{full_poincare_1}
\Vert  f - f_U \Vert_{L^{p}( U)} \le c |f|_{\W_{(\sigma)}^{s,p}( U)}.  
\end{equation}

For $n =1,\dotsc,m$ set $B_n = B(x_n,r_n/2) \subset B(x_n,r_n) \subseteq  U$.  The triangle and H\"{o}lder inequalities again allow us to estimate
\begin{equation*}
\begin{split}
\Vert  f - f_U\Vert_{L^{p}( U)} & \le \Vert  f - f_{B_1} \Vert_{L^{p}( U)} + |f_U - f_{B_1}| (\mathcal{L}^N( U))^{1/p} \\
&\le 2 \Vert  f - f_{B_1}\Vert_{L^{p}( U)} \le 2 \sum_{n=1}^m \Vert  f - f_{B_1}\Vert_{L^{p}(B(x_n,r_n))}.
\end{split}
\end{equation*}
From this we readily deduce that in order to prove \eqref{full_poincare_1} it suffices to prove that for $n=1,\dotsc,m$ there exists a constant $k_n = k_n( U,p,s) >0$ such that 
\begin{equation}\label{full_poincare_2}
\Vert  f - f_{B_1}\Vert_{L^{p}(B(x_n,r_n))} \le k_n |f|_{\W_{(\sigma)}^{s,p}( U)}. 
\end{equation}
Note that if $n=1$ then this inequality is true by virtue of Proposition \ref{proposition poincare}, so we may further reduce to proving \eqref{full_poincare_2} when $n \neq 1$.

Fix $2 \le n \le m$.  Since $ U$ is open and connected, it is polygonally connected, and so we can choose a polygonal path in $ U$ starting at the center of $B_1$ and ending at the center of $B_n$.  Since the path is connected and compact, we can choose open balls $B_k^\ast = B(z_k,s_k) \subseteq  U$ for $k=1,\dotsc,\ell_n$ such that $B_1^\ast = B_1$, $B_{\ell_n}^\ast = B_n$, $s_k = r_{z_k}/2$ for $r_{z_k} >0$ given by Lemma \ref{screening radius lemma},    and $V_k := B_k^\ast \cap B_{k+1}^\ast \neq \varnothing$ for each $1 \le k \le \ell_n -1$.  We then use the triangle inequality to estimate 
\begin{equation}\label{full_poincare_3}
\Vert  f - f_{B_1}\Vert_{L^{p}(B(x_n,r_n))} \le \Vert  f - f_{B_n}\Vert_{L^{p}(B(x_n,r_n))} + \sum_{k=1}^{\ell_n-1} |  f_{B_k^\ast} - f_{B_{k+1}^\ast } | (\mathcal{L}^N(B(x_n,r_n)))^{1/p}.
\end{equation}
For the first term on the right we can use Proposition \ref{proposition poincare} to bound 
\begin{equation}\label{full_poincare_4}
 \Vert  f - f_{B_n}\Vert_{L^{p}(B(x_n,r_n))} \le c r_n^{sp} |f|_{\W_{(\sigma)}^{s,p}( U)}.
\end{equation}
The remaining terms require more work.

Fix $1 \le k \le \ell_n-1$ and note that, by construction, the set $V_k = B_k^\ast \cap B_{k+1}^\ast$ is nonempty and thus has positive measure.  Then the triangle and H\"{o}lder inequalities once again allow us to estimate
\begin{equation*}
\begin{split}
|  f_{B_k^\ast} - f_{B_{k+1}^\ast } | &(\mathcal{L}^N(B(x_n,r_n)))^{1/p}  \le |  f_{B_k^\ast} - f_{V_k } |(\mathcal{L}^N(B(x_n,r_n)))^{1/p}  + |  f_{V_k} - f_{B_{k+1}^\ast } | (\mathcal{L}^N(B(x_n,r_n)))^{1/p}\\
&\le \Vert f - f_{V_k} \Vert_{L^p(B_k^\ast)} \frac{(\mathcal{L}^N(B(x_n,r_n)))^{1/p}}{  (\mathcal{L}^N(B_k^\ast))^{1/p}} + \Vert f - f_{V_k} \Vert_{L^p(B_{k+1}^\ast)} \frac{(\mathcal{L}^N(B(x_n,r_n)))^{1/p}}{(\mathcal{L}^N(B_{k+1}^\ast))^{1/p}}. 
\end{split}
\end{equation*}
The radii of the balls $\{B_k^\ast\}$ are such that we can apply Proposition \ref{proposition poincare}.  Doing so and chaining the estimate together with the last inequality then provides us with the bound
\begin{equation}\label{full_poincare_5}
\sum_{k=1}^{\ell_n-1} |  f_{B_k^\ast} - f_{B_{k+1}^\ast } | (\mathcal{L}^N(B(x_n,r_n)))^{1/p} \le c_n |f|_{\W_{(\sigma)}^{s,p}( U)}
\end{equation}
for a constant $c_n = c_n( U,p,s) >0$.  Combining \eqref{full_poincare_3}, \eqref{full_poincare_4}, and \eqref{full_poincare_5} then proves that \eqref{full_poincare_2} holds for all $2 \le n \le m$, which in turn completes the proof of \eqref{full_poincare_1}.
 
\end{proof}

Our next result is a technical lemma that will allow us to effectively use Theorem \ref{full_poincare}.

\begin{lemma}\label{lemma exhaustion}
Let $ U \subseteq \mathbb{R}^N$ be open and connected and $\sigma :  U \to (0,\infty]$ be a screening function.  Then there exists $\{V_n\}_{n =1}^\infty$ with the following properties.
\begin{enumerate}
 \item For each $n \in \mathbb{N}$ we have that $V_n \subset  U$ is nonempty, open, and connected.
 \item For each $n \in \mathbb{N}$ there exists $\ell_n \in \mathbb{N}$, $x_{1,n}, \dotsc x_{\ell_n,n} \in  U$, and $r_{1,n}, \dotsc, r_{\ell_n,n} >0$ such that $r_{k,n} = r_{x_{k,n}}$ is given by Lemma \ref{screening radius lemma} and
\begin{equation*}
 V_n = \bigcup_{k=1}^{\ell_n} B(x_{k,n}, r_{k,n}).
\end{equation*}
 \item For each $n \in \mathbb{N}$ we have that $\bar{V}_n \subset V_{n+1}$.
 \item We have that 
\begin{equation*}
  U = \bigcup_{n=1}^\infty V_n.
\end{equation*}
\end{enumerate}
\end{lemma}
\begin{proof}
For each $x \in  U$ let $r_x >0$ be as given by Lemma \ref{screening radius lemma}, which in particular means that $B[x,r_x] \subset B(x,2r_x) \subseteq  U$.  Then $\{B(x,r_x)\}_{x \in  U}$ is an open cover of $ U$, and so by Lindel\"{o}f's theorem  we can choose a countably infinite subcover $\{B_k\}_{k=1}^\infty$, where $B_k = B(x_k,r_{k})$.

We now claim that if $\varnothing \neq I \subset \mathbb{N}$ is a finite set and $V = \bigcup_{k \in I} B_k$, then there exists a finite set $J \subset \mathbb{N}$ such that $I \subset J$ and the set $W = \bigcup_{k \in J} B_k$ is connected and satisfies $\bar{V} \subset W$.  To prove the claim first note that $\bar{V} \subset  U$ is compact, so we can choose a finite set $I \subset K \subset \mathbb{N}$ such that $\bar{V} \subset \bigcup_{k \in K} B_k$.  Note in particular that $K$ must contain at least two elements since $I$ is nonempty, and we may then write $K = \{k_1,\dotsc,k_m\}$ for $m \ge 2$.  Since $ U$ is open and connected, for $j=2,\dotsc, m$ there exists a polygonal path between the centers of $B_{k_1}$ and $B_{k_j}$.  This path is connected and compact, so we can choose a finite set $J_j \subset \mathbb{N}$ such that $k_1,k_j \in J_j$ and the set
\begin{equation*}
W_j =  \bigcup_{k \in J_j} B_k 
\end{equation*}
is connected.  Set $J = \bigcup_{j=2}^m J_j \subset \mathbb{N}$ and note that $I \subset K \subseteq J$.  We have that $B_{k_1} \subseteq W_j$ for $j=2,\dotsc,m$, so the set 
\begin{equation*}
 W = \bigcup_{j=2}^m W_j = \bigcup_{k \in J} B_k
\end{equation*}
is connected.  By construction we have that $\bar{V} \subset W$, so the claim is proved.

Now we define the sequence $\{V_n\}_{n=1}^\infty$ inductively, starting with $V_1 = B_1$.   Suppose now that $V_n = \bigcup_{k \in J_n} B_k$ is given for some $\varnothing \neq J_n \subset \mathbb{N}$, and that $V_n$ is connected.  We set $I_{n+1} = \{1,\dotsc,\max(J_n)\} \subset \mathbb{N}$ and use the above claim to produce $J_{n+1} \subset \mathbb{N}$ from $I_{n+1}$.  Set $V_{n+1} = \bigcup_{k\in J_{n+1}} B_k$.  The claim guarantees that  $I_{n+1} \subset J_{n+1}$, that 
\begin{equation}
 V_n = \bigcup_{k \in J_n} B_k \subseteq \bigcup_{k \in I_{n+1}} B_k \subset \overline{\bigcup_{k \in I_{n+1}} B_k } \subset \bigcup_{k \in J_{n+1}} B_k = V_{n+1},
\end{equation}
and that $V_{n+1}$ is connected.  This inductively defines the sequence $\{V_n\}_{n=1}^\infty$, and it is clear that the sequence satisfies all of the stated properties by construction.

\end{proof}

In view of Proposition \ref{proposition seminorm}, to turn $\W_{(\sigma)}^{s,p}( U)$ into a normed space we consider the following equivalence relation: given $f,g\in L_{\operatorname*{loc}}^{1}( U)$, we say that $f\sim g$ if $f-g$ is equivalent to a constant in each connected component of $ U$.    In the next proposition we show that the resulting quotient space $\W_{(\sigma)}^{s,p}( U)/\mathbb{R}$ is a Banach space.

\begin{theorem}\label{theorem screened complete}
Suppose that $ U \subseteq \mathbb{R}^N$ is open, $1\leq p<\infty$, and $0<s<1$. Let  $\sigma :  U \to (0,\infty]$ be a screening function.  Then the space $\W_{(\sigma)}^{s,p}( U)/\mathbb{R}$ is a Banach space with the norm
\begin{equation*}
\Vert\lbrack f]\Vert_{\W_{(\sigma)}^{s,p}( U)/\mathbb{R}} 
:=|f|_{\W_{(\sigma)}^{s,p}( U)}.
\end{equation*}
Moreover, if $1<p<\infty$, then $\W_{(\sigma)}^{s,p}( U)/\mathbb{R}$ is reflexive.
\end{theorem}

\begin{proof}
In view of Proposition \ref{proposition seminorm}, the homogeneity of $|\cdot|_{\W_{(\sigma)}^{s,p}( U)}$, and Minkowski's inequality we have that $\Vert\cdot\Vert_{\W_{(\sigma)}^{s,p}( U)/\mathbb{R}}$ is a norm.  According to the discussion in Section \ref{subsection seminormed spaces}, in order to prove that $\W_{(\sigma)}^{s,p}( U)/\mathbb{R}$ is a Banach space, it is enough to show that $\W_{(\sigma)}^{s,p}( U)$ is sequentially complete.  Suppose then that $\{f_{n}\}_{n=1}^\infty$ is a Cauchy sequence  in $\W_{(\sigma)}^{s,p}( U)$.  

Throughout the rest of the proof, which we divide into several steps, we will employ the notation 
\begin{equation*}
 (f)_E = \frac{1}{\mathcal{L}^N(E)} \int_E f(x) dx
\end{equation*}
whenever $E \subseteq  U$ is a measurable set of positive measure.

\textbf{Step 1 -- Consequences of Poincar\'e's inequality:}  Let $\{ U_\alpha\}_{\alpha \in A}$ denote the connected components of $ U$ and note that $A$ is countable.  For each $\alpha \in A$ let $\{V_k^\alpha\}_{k=1}^\infty$ be the sequence of open subsets of $ U_\alpha$ given by Lemma \ref{lemma exhaustion}.  The lemma allows us to apply Theorem \ref{full_poincare} on each set $V_k^\alpha$ with $E = V_1^\alpha \subseteq V_k^\alpha$ in order to see that 
\begin{equation*}
 \Vert (f_n - (f_n)_{V_1^\alpha}) - (f_m - (f_m)_{V_1^\alpha} )  \Vert_{L^p(V_k^\alpha)} \le c_{k,\alpha}  |f_n - f_m|_{\W_{(\sigma)}^{s,p}(V_k^\alpha)} \le c_{k,\alpha}  |f_n - f_m|_{\W_{(\sigma)}^{s,p}( U)}.
\end{equation*}
Consequently, $\{f_n - (f_n)_{V_1^\alpha}\}_{n=1}^\infty$ is Cauchy in $L^p(V_k^\alpha)$ and hence convergent to some $g_k^\alpha \in L^p(V_k^\alpha)$.  We now aim to determine how $g_j^\alpha$ and $g_k^\alpha$ are related when $j \neq k$.

Let $\alpha \in A$ and $1 \le j < k < \infty$.  Lemma \ref{lemma exhaustion} shows that $V_j^\alpha \subset V_k^\alpha$ and so both $g_j^\alpha$ and $g_k^\alpha$ are defined on $V_j^\alpha$.  On $V_j^\alpha$ we have that 
\begin{equation*}
0 =  (f_n - (f_n)_{V_1^\alpha}) - (f_n - (f_n)_{V_1^\alpha}) \to g_j^\alpha - g_k^\alpha \text{ in }L^p(V_j^\alpha) \text{ as } n \to \infty, 
\end{equation*}
and hence 
\begin{equation}\label{complete_2}
g_j^\alpha = g_k^\alpha \text{ almost everywhere in } V_j^\alpha.  
\end{equation}

Lemma \ref{lemma exhaustion} guarantees that $ U_\alpha =  \bigcup_{k=1}^\infty V_k^\alpha$, and so once we have the functions $\{g_k^\alpha\}_{k=1}^\infty$ in hand, we may define the function $f :  U \to \mathbb{R}$ via 
\begin{equation}\label{complete_3}
 f(x) = g_k^\alpha(x) \text{ whenever } x \in V_k^\alpha \text{ for some } \alpha \in A \text{ and }k \in \mathbb{N}.
\end{equation}
This is well-defined by virtue of \eqref{complete_2}.  It's clear that $f$ is measurable, and since any compact subset of $ U$ is contained in $\bigcup_{\alpha \in B} V_k^\alpha$ for $k$ large enough and $B \subseteq A$ some finite set, we actually have that $f \in L^p_{\operatorname*{loc}}( U)$.

\textbf{Step 2  -- Passing to the limit:}  We define the set  
\begin{equation*}
 \Gamma = \{ (x,h) \in  U \times \mathbb{R}^N \;\vert\; h \in H(x) \} \subseteq \mathbb{R}^{2N}.
\end{equation*}
The function $\Phi :  U \times \mathbb{R}^N \to \mathbb{R} \times \mathbb{R}^{N}$ given by $\Phi(x,h) = (\sigma(x)-|h|,x+h)$ is clearly measurable, and $\Gamma = \Phi^{-1}((0,\infty) \times  U)$, so $\Gamma$ is $\mathcal{L}^{2N}$-measurable.

Then we define the measure $\mu:\mathcal{B}(\Gamma)\rightarrow [0,\infty]$ via
\begin{equation}\label{b measure}
\mu(E):=\int_{ U}\int_{H(x)}\chi_{E}(x,h)\,\frac{dh}{|h|^{sp+N}}dx 
\end{equation}
and consider the space $L^{p}(\Gamma;\mu)$.  Given $F \in\W_{(\sigma)}^{s,p}( U)$, we define $v_F \in L^p(\Gamma;\mu)$ via $v_{F}(x,h)=f(x+h)-f(x)$ and note that
\begin{equation}\label{b isomorphism}
|F|_{\W_{(\sigma)}^{s,p}( U)} = \Vert v_{F}\Vert_{L^{p}(\Gamma;\mu)}. 
\end{equation}
Consequently, for $n,m \in \mathbb{N}$ we have the identity
\begin{equation*}
|f_{n}-f_{m}|_{\W_{(\sigma)}^{s,p}( U) }
=\Vert v_{f_{n}}-v_{f_{m}}\Vert_{L^{p}(\Gamma;\mu)}, 
\end{equation*}
and so $\{v_{f_{n}}\}$ is a Cauchy sequence in $L^{p}(\Gamma;\mu)$.  Hence, $v_{f_{n}}\rightarrow v$ in $L^{p}(\Gamma;\mu)$ as $n \to \infty$.

\textbf{Step 3  -- Identifying the limit:}  We will now show that $v = v_f$ for the function $f:  U \to \mathbb{R}$ defined by \eqref{complete_3}.  Once this is established we may use the equalities
\begin{equation*}
 |f-f_{n}|_{\W_{(\sigma)}^{s,p}( U) }=
 \Vert v_{f}-v_{f_{n}}\Vert_{L^{p}(\Gamma;\mu)} =  \Vert v-v_{f_{n}}\Vert_{L^{p}(\Gamma;\mu)} 
\end{equation*}
to conclude that $f_n \to f$ in $\W_{(\sigma)}^{s,p}( U)$ as $n \to \infty$, thereby proving the sequential completeness of $\W_{(\sigma)}^{s,p}( U)$.

We know from Step 1 that for each $k \in \mathbb{N}$ and $\alpha \in A$ we have that $f_n - (f_n)_{V_1^\alpha} \to g_k^\alpha$ in $L^p(V_k^\alpha)$ as $n \to \infty$.  As such, we may iteratively extract subsequences and then choose a diagonal subsequence to produce a single subsequence $\{f_{n_m}\}_{m=1}^\infty$  and a Lebesgue measurable set $D \subset  U$ with $\mathcal{L}^N(D) =0$ such that for all $k \in \mathbb{N}$ and $\alpha \in A$ we have that 
\begin{equation}\label{complete_4}
f_{n_m} - (f_{n_m})_{V_1^\alpha} \to g_k^\alpha \text{ pointwise on } V_k^\alpha \backslash D \text{ as }m \to \infty. 
\end{equation}
Next we use Fubini's theorem to find a Lebesgue measurable set $E \subset  U$ with $\mathcal{L}^{N}(E)=0$ and a further subsequence, which up to relabeling we can still refer to as $\{f_{n_{m}}\}_{m=1}^\infty$, such that if $x\in U\setminus E$, then
\begin{equation}\label{complete_5}
\lim_{m\rightarrow\infty}\int_{H(x)}\frac{|f_{n_{m}}(x+h)-f_{n_{k}}(x)-v(x,h)|^{p}}{|h|^{sp+N}}\,dh=0.
\end{equation}

Fix $x \in  U \backslash (D \cup E)$.  According to \eqref{complete_5}, we can pick a measurable set $G_x \subset H(x)$ with $\mathcal{L}^N(G_x) =0$ and a further subsequence $\{f_{n_{m_\ell}}\}_{\ell=1}^\infty$ (both the set and the subsequence depend on the point $x$) such that if $h \in H(x) \backslash G_x$, then
\begin{equation}\label{complete_6}
f_{n_{m_\ell}}(x+h) - f_{n_{m_\ell}}(x) - v(x,h) \to 0 \text{ as } \ell \to \infty.
\end{equation}

Now let $D_x = -x +D$ and note that for $h \in H(x) \backslash D_x$ we have that $x+h \in [B(x,\sigma(x)) \cap  U ]\backslash D \subseteq  U \backslash D$.  By the translation invariance of Lebesgue measure we have that  $\mathcal{L}^N(D_x) = \mathcal{L}^N(D) =0$.  In particular, this means that $\mathcal{L}^N(D_x \cup G_x) = 0$ as well.  

Fix $h \in H(x) \backslash (D_x \cup G_x)$ and pick $j,k \in \mathbb{N}$ and $\alpha \in A$ such that $x \in V_k^\alpha$ and $x+h \in V_j^\alpha$.  We may then write 
\begin{equation*}
f_{n_{m_\ell}}(x+h) - f_{n_{m_\ell}}(x) = [f_{n_{m_\ell}}(x+h) - (f_{n_{m_\ell}})_{V_1^\alpha}  ] - [f_{n_{m_\ell}}(x) - (f_{n_{m_\ell}})_{V_1^\alpha}].
\end{equation*}
According to \eqref{complete_4} and \eqref{complete_6} we may then send $\ell \to \infty$ and use the definition of $f$ from \eqref{complete_3} in order to deduce that
\begin{equation*}
v(x,h)   = g_j^\alpha(x+h) - g_k^\alpha(x)  =  f(x+h) - f(x).
\end{equation*}
Thus $v = v_f$ for $\mathcal{L}^{2N}$ a.e. $(x,h) \in \Gamma$, which completes the proof of sequential completeness.

\textbf{Step 4 -- Reflexivity:}  Now assume that $1<p<\infty$.  Since the space $L^{p}(\Gamma;\mu)$ is reflexive, and the mapping $\W_{(\sigma)}^{s,p} ( U)/\mathbb{R} \ni [f]\mapsto v_{f} \in L^p(\Gamma;\mu)$ is an isometric isomorphism by \eqref{norm quotient} and \eqref{b isomorphism}, we can identify $\W_{(\sigma)}^{s,p} ( U)/\mathbb{R}$ with a closed subspace of $L^{p}(\Gamma;\mu)$.  It then suffices to observe that closed subspaces of reflexive spaces are reflexive.
\end{proof}

Now we show that the screened spaces possess the same interpolation properties as the usual Sobolev spaces.

\begin{proposition}\label{proposition screened interpolation}
Suppose that $ U \subseteq \mathbb{R}^N$ is open, $1\leq p<\infty$,  $0< s_1 < s_2 <1$, and $\sigma :  U \to (0,\infty]$ is a screening function.  If $f \in \W_{(\sigma)}^{s_1,p}( U) \cap \W_{(\sigma)}^{s_2,p}( U)$ and $s = \theta s_1 + (1-\theta) s_2$ for some $\theta \in (0,1)$, then $f \in \W_{(\sigma)}^{s,p}( U)$ and 
\begin{equation}\label{proposition screened interpolation bound}
 |f|_{\W_{(\sigma)}^{s,p}( U)} \le \left( |f|_{\W_{(\sigma)}^{s_1,p}( U)} \right)^\theta \left( |f|_{\W_{(\sigma)}^{s_1,p}( U)} \right)^{1-\theta}.
\end{equation}
\end{proposition}
\begin{proof}
For each $x \in  U$ we use H\"{o}lder's inequality to bound 
\begin{align*}
 \int_{H(x)} \frac{|f(x+h)-f(x)|^p}{|h|^{sp+N}} dh  &= \int_{H(x)} \left( \frac{|f(x+h)-f(x)|^p}{|h|^{N+s_1 p}}\right)^{\theta} \left( \frac{|f(x+h)-f(x)|^p}{|h|^{N+s_2 p}}\right)^{1-\theta}  dh \\
& \le \left(\int_{H(x)}   \frac{|f(x+h)-f(x)|^p}{|h|^{N+s_1 p} dh }\right)^{\theta} \left(\int_{H(x)}  \frac{|f(x+h)-f(x)|^p}{|h|^{N+s_2 p}}dh \right)^{1-\theta}.   
\end{align*}
Then \eqref{proposition screened interpolation bound} follows by integrating over $x \in  U$ and  applying H\"older's inequality again.
\end{proof}

Next, we show that spaces defined by different screening functions nest when the screening functions are ordered.

\begin{proposition}\label{proposition screened nesting}
Suppose that $ U \subseteq \mathbb{R}^N$ is open, $1\leq p<\infty$, and  $0< s <1$.  Suppose that $\sigma_1,\sigma_2 :  U \to (0,\infty]$ are two screening functions such that $\sigma_1 \le \sigma_2$ on $ U$.  Then for each $f \in L^1_{\operatorname*{loc}}( U)$ we have that 
\begin{equation}\label{equivs 1}
|f|_{\W_{(\sigma_1)}^{s,p}( U) } \le  |f|_{\W_{(\sigma_2)}^{s,p}( U) }.
\end{equation}
In particular, we have the subspace inclusion $\W_{(\sigma_2)}^{s,p}( U) \subseteq \W_{(\sigma_1)}^{s,p}( U)$.
\end{proposition}
\begin{proof}
If $\sigma_1 \le \sigma_2$ on $ U$, then we have the obvious inequality
\begin{equation*}
 \int_{ U} \int_{H_{\sigma_1}(x)} \frac{|f(x+h)-f(x)|^p }{|h|^{sp+N}}dh dx \le  \int_{ U} \int_{H_{\sigma_2}(x)} \frac{|f(x+h)-f(x)|^p }{|h|^{sp+N}}dh dx 
\end{equation*}
for all $f \in L^1_{\operatorname*{loc}}( U)$, where $H_{\sigma_i(x)} = (-x +  U) \cap B(0,\sigma_i(x))$.  The result follows immediately from this.
\end{proof}

Inclusion in $L^p( U)$ is not a requirement for inclusion in $\W_{(\sigma)}^{s,p}( U)$.  Our next result examines what happens when we require both inclusions in the case when the screening function is bounded below.

\begin{proposition}\label{proposition screened Lp}
Suppose that $ U \subseteq \mathbb{R}^N$ is open, $1\leq p<\infty$, and  $0< s <1$.  Suppose that $\sigma :  U \to (0,\infty]$ is a screening function such that $0 < \sigma_- = \inf_{ U} \sigma$.   Then there exists a constant $c = c(N,s,p, \sigma_-)>0$ such that for each $f \in L^1_{\operatorname*{loc}}( U)$ we have that 
\begin{equation}\label{proposition screened Lp equivs}
 \Vert f \Vert_{L^p( U)}  +  |f|_{\W_{(\sigma)}^{s,p}( U) } \le \Vert f \Vert_{L^p( U)}  +  |f|_{\dot{W}^{s,p}( U) } \le c \left( \Vert f \Vert_{L^p( U)}  +  |f|_{\W_{(\sigma)}^{s,p}( U) } \right).
\end{equation} 
In particular, we have the algebraic and topological equalities
\begin{equation*}
 W^{s,p}( U) = \dot{W}^{s,p}( U) \cap L^p( U) = \W_{(\sigma)}^{s,p}( U) \cap L^p( U).
\end{equation*}
\end{proposition}
\begin{proof}
The first bound in \eqref{proposition screened Lp equivs} follows immediately from Proposition \ref{proposition screened nesting}, so it suffices to only prove the second.  To this end we first write  
\begin{align*}
|f|_{\dot{W}^{s,p}( U) }^p & = \int_U \int_{B(x,\sigma(x)) \cap  U} \frac{|f(y)-f(x)|^p}{|y-x|^{sp+N}}dy dx  + \int_U \int_{ U \backslash B(x,\sigma(x))} \frac{|f(y)-f(x)|^p}{|y-x|^{sp+N}}dy dx \\
& = |f|_{\W_{(\sigma)}^{s,p}( U) }^p  
+ \int_U \int_{ U \backslash B(x,\sigma(x))} \frac{|f(y)-f(x)|^p}{|y-x|^{sp+N}}dy dx.
\end{align*}
Next we estimate 
\begin{align*}
 & \int_U \int_{ U \backslash B(x,\sigma(x))} \frac{|f(y)-f(x)|^p}{|y-x|^{sp+N}}dy dx \\
 & \le 2^{p-1}\int_U \int_{ U \backslash B(x,\sigma(x))} \frac{|f(y)|^p}{|y-x|^{sp+N}}dy dx + 2^{p-1} \int_U \int_{ U \backslash B(x,\sigma(x))} \frac{|f(x)|^p}{|y-x|^{sp+N}}dy dx =: 2^{p-1}(I + II).
\end{align*}
We handle the first term with Tonelli's theorem, a change of variables, and spherical coordinates:
\begin{align*}
 I & = \int_{ U} \int_{ U} \chi_{\{ |y-x| \ge \sigma(x) \}}(x) \frac{|f(y)|^p}{|y-x|^{sp+N}}dx dy 
 \le \int_{ U} \int_{ U} \chi_{\{ |y-x| \ge \sigma_- \}}(x) \frac{|f(y)|^p}{|y-x|^{sp+N}}dx dy \\
 & \le \int_{ U}   \int_{B(0,\sigma_-)^c}  \frac{|f(y)|^p}{|h|^{sp+N}} dh  dy
 = \Vert f \Vert_{L^p( U)}^p  \beta_N \int_{\sigma_-}^\infty \frac{dr}{r^{1+sp}} = \frac{\beta_N }{sp \sigma_-^{sp}} \Vert f \Vert_{L^p( U)}^p.
\end{align*}
We may similarly bound the second term:  
\begin{equation*}
II \le \int_U \int_{B(0,\sigma_-)^c} \frac{|f(x)|^p}{|h|^{sp+N}} dh dx = \frac{\beta_N }{sp \sigma_-^{sp}} \Vert f \Vert_{L^p( U)}^p.
\end{equation*}
The second bound in  \eqref{proposition screened Lp equivs} then follows directly from these estimates.

\end{proof}

Our next result establishes that the screened homogeneous spaces are strictly bigger than the standard homogeneous spaces when the screening function is unity and the set $ U$ contains a ray in a set of directions of positive $\mathcal{H}^{N-1}$ measure.

\begin{theorem}\label{theorem strict inclusion}
Let $1 \le p < \infty$. Let $\Gamma \subseteq \mathbb{S}^{N-1}$ be such that $\mathcal{H}^{N-1}(\Gamma) >0$, $\rho \ge 0$, and define the infinite cone-like set
\begin{equation*}
K_{\Gamma,\rho} = \{x \in \mathbb{R}^N  \;\vert\; \rho < |x|  \text{ and } x/|x| \in \Gamma \}. 
\end{equation*}
Suppose that $ U \subseteq \mathbb{R}^N$ is open and that $K_{\Gamma,\rho} \subseteq  U$.  Then there exists a function 
\begin{equation*}
u \in \bigcap_{0 < s < 1} \W_{(1)}^{s,p}( U) \backslash  \bigcup_{0 < s < 1} \dot{W}^{s,p}( U). 
\end{equation*}
In particular, if $0 < s < 1$, then 
\begin{equation*}
\dot{W}^{s,p}( U) \subsetneqq \W_{(1)}^{s,p}( U).
\end{equation*}
\end{theorem}
\begin{proof}
 
Define $f \in C^\infty([0,\infty))$ via 
\begin{equation*}
 f(r) = \int_0^r \frac{dt}{(2+t)^{N/p} (\log (2+t))^{2/p}},
\end{equation*}
and note that $f$ is Lipschitz and satisfies
\begin{equation*}
 |f|_{0,1} =  \frac{1}{2^{N/p} (\log 2)^{2/p}}.
\end{equation*}
We then let $u \in C^{0,1}( U)$ be given by $u(x) = f(|x|)$.  Let $0 < s < 1$.  We will prove that $u \in \W_{(1)}^{s,p}( U) \backslash \dot{W}^{s,p}( U)$, from which the conclusions readily follow.

We begin by proving that $u \in \W_{(1)}^{s,p}( U)$.  Write
\begin{align*}
|u|_{\W_{(1)}^{s,p}( U)}^p & = \int_{B(0,1)\cap  U} \int_{B(0,1)} \frac{ |f(|x+h|) - f(|x|)|^p }{|h|^{sp+N}} dh dx  \\
& \quad +   \int_{B(0,1)^c \cap  U} \int_{B(0,1)} \frac{ |f(|x+h|) - f(|x|)|^p }{|h|^{sp+N}}  dh dx  =:  I + II.
\end{align*}

To estimate the term $I$ we use the fact that $f$ is Lipschitz together with Tonelli's theorem and a change to spherical coordinates: 
\begin{equation*}
I  \le \int_{B(0,1)} \int_{B(0,1)} \frac{|f|_{0,1}^p |h|^p  }{|h|^{sp+N}} dh dx  
 = \alpha_N \beta_N |f|_{0,1}^p \int_0^1 \frac{r^{N-1}}{r^{N-(1-s)p}}  dr    < \infty.
\end{equation*}
To handle the term $II$ we observe that $f'$ is positive and decreasing on $[0,\infty)$, so for $a,b \in [0,\infty)$ we have that
\begin{equation*}
 |f(a) - f(b)| \le f'( \min\{a,b\}  ) |a-b|.
\end{equation*}
On the other hand, for $|x| \ge 1$ and $|h| <1$ we have that 
\begin{equation*}
 |x|-1 = \min\{|x|,|x|-1\} \le \min\{|x|,|x+h| \},
\end{equation*}
so 
\begin{equation*}
 f'(\min\{|x|,|x+h| \}) \le f'(|x|-1) = \frac{1}{(1+|x|)^{N/p} (\log (1+|x|))^{2/p}}.
\end{equation*}
Combining these and using Tonelli's theorem and spherical coordinates then shows that 
\begin{align*}
II & \le \left(\int_{B(0,1)^c}  \frac{dx}{(1+|x|)^{N} (\log (1+|x|))^{2}} \right) \left(\int_{B(0,1)}  \frac{|h|^p }{|h|^{sp+N}} dh\right) \\
&= \beta_N^2 \left(\int_1^\infty \frac{r^{N-1}}{(1+r)^{N} (\log (1+r))^{2}} dr \right) \left( \int_0^1  \frac{r^{N-1}}{r^{N-(1-s)p}}  dr  \right) \\
&\le \beta_N^2 \left(\int_2^\infty \frac{1}{r (\log (r))^{2}} dr \right) \left( \int_0^1  \frac{1}{r^{1-(1-s)p}}  dr  \right) < \infty.
\end{align*}
Assembling the above estimates, we see that 
\begin{equation*}
 |u|_{\W_{(1)}^{s,p}( U)}^p = I + II < \infty,
\end{equation*}
and thus $u \in \W_{(1)}^{s,p}( U)$.

We now turn to the proof that $u \notin \dot{W}^{s,p}( U)$.   Let $0 < \varepsilon < 1-s$ be such that $N/p + \varepsilon \neq 1$ and choose $R_\varepsilon  >\max\{2,\rho\}$ such that if $t \ge R_\varepsilon$ then $(\log(2+t))^{2/p} \le (2+t)^{\varepsilon}$.  Then for $|x| \ge R_\varepsilon$ and $|h| \ge 3|x|$ we may estimate 
\begin{align*}
|f(|x+h|)-f(|x|)|^p &= \left( \int_{|x|}^{|x+h|} \frac{dt}{(2+t)^{N/p} (\log (2+t))^{2/p}}  \right)^p \\
&\ge  \left(  \int_{|x|}^{2|x|} \frac{dt}{(2+t)^{N/p+\varepsilon} }  \right)^p 
\ge  \left(  \int_{|x|}^{2|x|} \frac{dt}{(2t)^{N/p+\varepsilon} }  \right)^p \\
& = \frac{1}{2^{N+p\varepsilon}}   |x|^{p-N-p\varepsilon} \left( \frac{2^{1-N/p-\varepsilon} -1}{1-N/p -\varepsilon}\right)^p =: c(N,p,\varepsilon) |x|^{p-N-p\varepsilon},
\end{align*}
where $c(N,p,\varepsilon) >0$.  Hence 
\begin{align*}
  |u|_{\dot{W}^{s,p}( U)}^p
& \ge \int_{B(0,R_\varepsilon)^c \cap K_{\Gamma,\rho}}   \int_{B(0,3|x|)^c \cap K_{\Gamma,\rho}} \frac{|f(|x+h|)-f(|x|)|^p}{|h|^{sp+N}}dh dx \\  
& \ge c(N,p,\varepsilon) \int_{B(0,R_\varepsilon)^c \cap K_{\Gamma,\rho}} \int_{B(0,3|x|)^c \cap K_{\Gamma,\rho}} \frac{|x|^{p-N-p\varepsilon}}{|h|^{sp+N}}dh dx \\
& = c(N,p,\varepsilon) (\mathcal{H}^{N-1}(\Gamma))^2 \int_{R_\varepsilon}^\infty r^{N-1} r^{p-N-p\varepsilon} \int_{3r}^\infty \frac{t^{N-1}}{t^{sp+N}}dt dr \\
& = \frac{c(N,p,\varepsilon)}{sp 3^{sp}} (\mathcal{H}^{N-1}(\Gamma))^2 \int_{R_\varepsilon}^\infty  r^{p(1-s-\varepsilon)-1}  dr = \infty
\end{align*}
since $p(1-s - \varepsilon) >0$, and we conclude that $u \notin \dot{W}^{s,p}( U)$.

\end{proof}

\begin{remark}
The function $u$ constructed in Theorem \ref{theorem strict inclusion} clearly satisfies $\lim\limits_{|x| \to \infty} u(x) = \infty$ when $p < N$ and when $2 \le N =p$.   In these cases we deduce that for $ U \subseteq \mathbb{R}^N$ as in the theorem,  
\begin{equation*}
u \in \W_{(1)}^{s,p}( U) \backslash \bigcup_{1 \le q \le \infty} L^q( U). 
\end{equation*}
\end{remark}

Theorem \ref{theorem strict inclusion} dealt with unbounded sets and a unit screening function.  Next we consider the case in which the screening function vanishes on part of the boundary of a bounded cylinder-like set.  We show that again the screened space is strictly bigger than the standard homogeneous fractional Sobolev space.

\begin{theorem}\label{theorem strict inclusion vanishing}
Let $N \ge 2$, $1 \le p< \infty$, $0 < s < 1$, and 
\begin{equation*}
\frac{1}{1-s} \left(2- \frac{1}{p} \right) < r < \infty.
\end{equation*}
Let $a, b \in \mathbb{R}$ with $a < b$ and $\varnothing \neq V \subset \mathbb{R}^{N-1}$ be bounded and open.  Define $ U = V \times (a,b) \subset \mathbb{R}^N$, and let $\sigma:  U \to (0,1/2)$ be given by 
\begin{equation*}
 \sigma(x) = \frac{1}{2} \left( \frac{x_N - a}{b-a}\right)^r.
\end{equation*}
Then 
\begin{equation*}
\dot{W}^{s,p}( U) \subsetneqq  \W_{(\sigma)}^{s,p}( U). 
\end{equation*}
\end{theorem}
\begin{proof}
Thanks to translation invariance we may assume, without loss of generality, that $0 \in V$ and $a =0$, in which case $\sigma(x) = \frac{1}{2}(x_N/b)^r$.  We  choose $0 < R_1 < R_2 < \infty$ such that $B'(0,R_1) \subseteq V \subseteq B'(0,R_2)$.  Define the function $u: B'(0,R_2) \times (0,b) \to (0,\infty)$ via $u(x) = 1/x_N$.  We will prove that $u \in \W_{(\sigma)}^{s,p}(B'(0,R_2) \times (0,b))$ but that $u \notin \dot{W}^{s,p}(B'(0,R_1) \times (0,b))$, from which the result clearly follows.  

We begin with the proof that $u \in \W_{(\sigma)}^{s,p}(B'(0,R_2) \times (0,b))$.
We can estimate
\begin{align*}
|u|_{\W_{(\sigma)}^{s,p}(B'(0,R_2) \times (0,b))}^p & = \int_{B'(0,R_2) \times (0,b)} \int_{H(x)} \frac{|h_N|^p}{x_N^p |x_N+h_N|^p |h|^{sp+N}} dh dx \\ 
& \le \int_0^{b} \int_{B'(0,R_2)}   \int_{B(0,\sigma(x))} \frac{|h_N|^p}{x_N^p |x_N+h_N|^p |h|^{sp+N}} dh  dx' dx_N. 
\end{align*}
Note that $r>1$, so  for $x_N \in (0,b)$ and $-\frac{1}{2} (x_N/b)^r = - \sigma(x) < h_N$ we have that 
\begin{equation*}
0 < \frac{1}{2} x_N < x_N \left(1 - \frac{1}{2} \left(\frac{x_N}{b}\right)^{r-1}\right) = x_N - \frac{1}{2} \left(\frac{x_N}{b}\right)^r  < x_N +h_N,
\end{equation*}
which means that
\begin{equation*}
 \frac{1}{|x_N+h_N|^p} < \frac{2^p}{x_N^p}.
\end{equation*}
Then we use this, Tonelli's theorem, and spherical coordinates to bound
\begin{align*}
&  \int_0^{b} \int_{B'(0,R_2)} \int_{B(0,\sigma(x))} \frac{|h_N|^p}{x_N^p |x_N+h_N|^p |h|^{sp+N}} dh  dx' dx_N \\
& \le  \int_0^{b} \int_{B'(0,R_2)}  \frac{2^p}{x_N^{2p}}  \int_{B(0,\sigma(x))} |h|^{-N + (1-s)p} dh  dx' dx_N \\
& =  \frac{2^{sp} \beta_N \alpha_{N-1} R_2^{N-1} }{(1-s)p b^{(1-s)pr} }     \int_0^{b} \frac{x_N^{(1-s)rp}}{x_N^{2p}}     dx_N \\
&  =  \frac{2^{sp} \beta_N \alpha_{N-1} R_2^{N-1} }{(1-s)p b^{(1-s)pr} }     \int_0^{b} x_N^{(1-s)rp -2p}   dx_N < \infty
\end{align*}
since 
\begin{equation*}
 (1-s) rp - 2p > -1.
\end{equation*}
Thus $u \in \W_{(\sigma)}^{s,p}(B'(0,R_2) \times (0,b))$.

We now turn to the proof that $u \notin \dot{W}^{s,p}(B'(0,R_1) \times (0,b))$.  Note that $h \in -x + B'(0,R_1) \times (0,b)$ if and only if $-x_N < h_N < b-x_N$ and $|x' + h'| < R_1$.  Then 
\begin{align*}
& |u|_{\dot{W}^{s,p}(B'(0,R_1) \times (0,b))}^p 
 = \int_{B'(0,R_1) \times (0,b)}  \int_{-x_N}^{b-x_N} \int_{B'(-x',R_1)} \frac{|h_N|^p}{x_N^p |x_N+h_N|^p (|h'|^2 + h_N^2)^{(sp+N)/2}}dh' dh_N dx  \\
& \ge \int_0^{b/2} \int_{B'(0,R_1/2)}  \int_{-x_N}^{-x_N/2} \int_{B'(-x',R_1)} \frac{|h_N|^p}{x_N^p |x_N+h_N|^p (|h'|^2 + h_N^2)^{(sp+N)/2}}dh' dh_N dx' dx_N. 
\end{align*}
Now note that if $0 < x_N < b/2$, $|x'| < R_1/2$, $-x_N < h_N < -x_N/2$, and 
\begin{equation*}
 \frac{R_1}{4b} x_N < |h'| < \frac{R_1}{2b} x_N,
\end{equation*}
then we have that 
\begin{equation*}
 |x'+h'| \le |x'| + |h'| < \frac{R_1}{2} + \frac{R_1}{2b} b =R_1
\end{equation*}
and
\begin{equation*}
 \frac{R_1}{4b} |h_N| = \frac{R_1}{4b}(-h_N) < \frac{R_1}{4b} x_N < |h'| < \frac{R_1}{2b} x_N < \frac{R_1}{2b} (-h_N)  =  \frac{R_1}{b} |h_N|,
\end{equation*}
which in particular imply that there exists a constant $c = c(R_1,b,p) >0$ such that 
\begin{equation*}
 \frac{1}{|x_N + h_N|^p  (|h'|^2 + h_N^2)^{(sp+N)/2}} > \frac{c}{x_N^p |h'|^{sp+N}}.
\end{equation*}
Thus, if we write $A(x_N) = B'(0,(2b)^{-1}R_1 x_N) \backslash B'[0,(4b)^{-1}R_1 x_N]$, then we may use Tonelli's theorem and spherical coordinates to see that
\begin{align*}
 |u|_{\dot{W}^{s,p}(B'(0,R_1) \times (0,b))}^p 
& \ge c \int_0^{b/2} \frac{1}{x_N^{2p}}\int_{B'(0,R_1/2)}  \int_{-x_N}^{-x_N/2} \int_{A(x_N)} \frac{|h_N|^p}{|h'|^{sp+N}}dh' dh_N dx' dx_N \\
& = \frac{c \beta_N \alpha_{N-1} R_1^{N-1}}{2^{N-1}} \int_0^{b/2} \frac{1}{x_N^{2p}} \left( \int_{x_N/2}^{x_N} h_N^p dh_N \right) \left( \int_{(4b)^{-1}R_1 x_N}^{(2b)^{-1}R_1 x_N} \frac{r^{N-2}}{r^{sp+N}}dr  \right)    dx_N \\
& = C\int_0^{b/2} \frac{1}{x_N^{2p}}  x_N^{p+1} x_N^{-1 - sp} dx_N  = C \int_0^{b/2} \frac{1}{x_N^{(1+s)p}}  dx_N  = \infty,
\end{align*}
where $C=C(R_1,b,p,s,N)>0$. Hence $u \notin \dot{W}^{s,p}(B'(0,R_1) \times (0,b))$.

\end{proof}

\begin{remark}
We have phrased Theorem \ref{theorem strict inclusion vanishing} in terms of cylindrical sets in dimension $N \ge 2$, but the argument used in the proof may be readily adapted to prove the one-dimensional result
\begin{equation*}
\dot{W}^{s,p}((a,b)) \subsetneqq  \W_{(\sigma)}^{s,p}((a,b)) 
\end{equation*}
when $a,b \in \mathbb{R}$ with $a< b$ and 
\begin{equation}
 \sigma(x) = \frac{1}{2} \left( \frac{x-a}{b-a} \right)^r
\end{equation}
for $r >1$ as given in the theorem.
\end{remark}

\subsection{Screened homogeneous fractional Sobolev spaces: further properties in $\mathbb{R}^N$}

In this subsection we restrict our attention to the special case $ U = \mathbb{R}^N$ and prove a number of interesting results about the screened spaces.  One particular advantage of this case is that we have $H(x) = B(0,\sigma(x))$ for all $x \in \mathbb{R}^N$ whenever $\sigma : \mathbb{R}^N \to (0,\infty]$ is a screening function.

Our first result shows that any two screening functions that are bounded above and below give rise to the same screened spaces.  This should be compared to Proposition \ref{proposition screened nesting}.

\begin{theorem}\label{theorem screened equivalence}
Suppose that $\sigma_1,\sigma_2 : \mathbb{R}^N \to (0,\infty]$ are two screening functions such that 
\begin{equation}\label{theorem screened equivalence assumps}
 0 < \inf_{\mathbb{R}^N} \sigma_i \le \sup_{\mathbb{R}^N} \sigma_i < \infty \text{ for }i =1,2.
\end{equation}
Then for $0 < s < 1$ and $1 \le p < \infty$ there exist constants $c_i(s,p,N,\sigma_1,\sigma_2) >0$ for $i=0,1$ such that 
\begin{equation}\label{theorem screened equivalence bound}
 c_0   |f |_{\W_{(\sigma_2)}^{s,p}(\mathbb{R}^N)} \le |f |_{\W_{(\sigma_1)}^{s,p}(\mathbb{R}^N)} \le c_1 |f |_{\W_{(\sigma_2)}^{s,p}(\mathbb{R}^N)}
\end{equation}
for all $f \in L^1_{\operatorname*{loc}}(\mathbb{R}^N)$, and consequently we have the algebraic and topological equality 
\begin{equation*}
\W_{(\sigma_1)}^{s,p}(\mathbb{R}^N) = \W_{(\sigma_2)}^{s,p}(\mathbb{R}^N). 
\end{equation*}

\end{theorem}
\begin{proof}

We divide the proof into steps.

\textbf{Step 1 -- Doubled constants:}  We now prove the result when $\sigma_1 = r$ and $\sigma_2 = 2r$ for some constant $r \in (0,\infty)$.  Let $f \in L^1_{\operatorname*{loc}}(\mathbb{R}^N)$.  Then we may estimate
\begin{align*}
 \int_{\mathbb{R}^N} \int_{B(0,2r)\backslash B(0,r)} \frac{|f(x+h)-f(x)|^p }{|h|^{sp+N}}dh dx 
&= \frac{1}{2^{sp}}  \int_{\mathbb{R}^N} \int_{B(0,r)\backslash B(0,r/2)} \frac{|f(x+2h)-f(x)|^p }{|h|^{sp+N}}dh dx \\
& \le \frac{2^{p-1}}{2^{sp}}  \int_{\mathbb{R}^N} \int_{B(0,r)\backslash B(0,r/2)} \frac{|f(x+2h)-f(x+h)|^p }{|h|^{sp+N}}dh dx  \\
& \quad +  \frac{2^{p-1}}{2^{sp}}  \int_{\mathbb{R}^N} \int_{B(0,r)\backslash B(0,r/2)} \frac{|f(x+h)-f(x)|^p }{|h|^{sp+N}}dh dx \\
& \le  2^{p(1-s)}  \int_{\mathbb{R}^N} \int_{B(0,r)} \frac{|f(x+h)-f(x)|^p }{|h|^{sp+N}}dh dx,
\end{align*}
where in the first equality we have made the change of variables $h \mapsto 2h$ and in the last inequality we have changed $x \mapsto x+h$ in the first integral.  Thus we may decompose $B(0,2r) = B(0,r) \cup [B(0,2r) \backslash B(0,r)]$ in order to see that
\begin{equation*}
|f |_{\W_{(2r)}^{s,p}(\mathbb{R}^N)}^p \le (1+  2^{p(1-s)}) \int_{\mathbb{R}^N} \int_{B(0,r)} \frac{|f(x+h)-f(x)|^p }{|h|^{sp+N}}dh dx  =(1+  2^{p(1-s)}) |f |_{\W_{(r)}^{s,p}(\mathbb{R}^N)}^p.
\end{equation*}
This and \eqref{equivs 1} then prove \eqref{theorem screened equivalence bound} in this special case.

\textbf{Step 2 -- Pairs of constants:} We now prove the result when $\sigma_1 = r_1$ and $\sigma_2 = r_2$ for some constants $0 < r_1 < r_2 < \infty$.  Choose $k \in \mathbb{N}$ such that $r_2 < 2^k r_1$.  The analysis from  Step 1 provides us with a constant $c= c(s,p)>0$ such that $|f |_{\W_{(2r)}^{s,p}(\mathbb{R}^N)} \le c |f |_{\W_{(r)}^{s,p}(\mathbb{R}^N)}$ whenever $f \in L^1_{\operatorname*{loc}}(\mathbb{R}^n)$ and $r >0$.  Applying this iteratively then shows that 
\begin{equation*}
|f |_{\W_{(2^k r)}^{s,p}(\mathbb{R}^N)} \le c^k |f |_{\W_{(r)}^{s,p}(\mathbb{R}^N)}.
\end{equation*}
From this and \eqref{equivs 1} we then deduce that 
\begin{equation*}
 |f |_{\W_{(r_1)}^{s,p}(\mathbb{R}^N)} \le  |f |_{\W_{(r_2)}^{s,p}(\mathbb{R}^N)} \le  |f |_{\W_{(2^k r_1)}^{s,p}(\mathbb{R}^N)} \le c^k  |f |_{\W_{(r_1)}^{s,p}(\mathbb{R}^N)},
\end{equation*}
which is  \eqref{theorem screened equivalence bound} in this special case.

\textbf{Step 3 -- The general case:}  Now consider general $\sigma_1,\sigma_2$ satisfying \eqref{theorem screened equivalence assumps}.  Write 
\begin{equation*} 
 0 <  \sigma_{i,-} :=  \inf_{\mathbb{R}^N} \sigma_i \le \sup_{\mathbb{R}^N} \sigma_i  =: \sigma_{i,+}   < \infty \text{ for }i =1,2.
\end{equation*}
Then from Step 2 we can find constants $c_2,c_3 >0$ such that 
\begin{equation*}
 |f |_{\W_{(\sigma_{1,+})}^{s,p}(\mathbb{R}^N)} \le c_2  |f |_{\W_{(\sigma_{2,-})}^{s,p}(\mathbb{R}^N)}
\text{ and }
 |f |_{\W_{(\sigma_{2,+})}^{s,p}(\mathbb{R}^N)} \le c_3  |f |_{\W_{(\sigma_{1,-})}^{s,p}(\mathbb{R}^N)}
\end{equation*}
for all $f \in L^1_{\operatorname*{loc}}(\mathbb{R}^N)$.  Using this and \eqref{equivs 1} then shows that 
\begin{align*}
  |f |_{\W_{(\sigma_{1})}^{s,p}(\mathbb{R}^N)} \le  |f |_{\W_{(\sigma_{1,+})}^{s,p}(\mathbb{R}^N)} &\le c_2  |f |_{\W_{(\sigma_{2,-})}^{s,p}(\mathbb{R}^N)} \le c_2  |f |_{\W_{(\sigma_{2})}^{s,p}(\mathbb{R}^N)} \le c_2  |f |_{\W_{(\sigma_{2,+})}^{s,p}(\mathbb{R}^N)}  \\
&\le c_2 c_3 |f |_{\W_{(\sigma_{1,-})}^{s,p}(\mathbb{R}^N)} \le c_2 c_3  |f |_{\W_{(\sigma_{1})}^{s,p}(\mathbb{R}^N)},
\end{align*}
which proves  \eqref{theorem screened equivalence bound}.

\end{proof}

We can combine Theorems \ref{theorem strict inclusion} and \ref{theorem screened equivalence} to deduce the following interesting corollary.

\begin{corollary}
Let $1 \le p < \infty$, $0 < s < 1$, and suppose that $\sigma : \mathbb{R}^N \to (0,\infty)$ is a screening function that is bounded above.  Then 
\begin{equation*}
 \dot{W}^{s,p}(\mathbb{R}^N) \subsetneqq \W_{(\sigma)}^{s,p}(\mathbb{R}^N).
\end{equation*}
\end{corollary}
\begin{proof}
Write $\sigma_+ = \sup_{\mathbb{R}^N} \sigma  \in (0, \infty)$.  Then Theorems \ref{theorem strict inclusion} and \ref{theorem screened equivalence} may be combined with Proposition \ref{proposition screened nesting} to see that 
\begin{equation*}
 \dot{W}^{s,p}(\mathbb{R}^N) \subsetneqq \W_{(1)}^{s,p}(\mathbb{R}^N) = \W_{(\sigma_+)}^{s,p}(\mathbb{R}^N) \subseteq \W_{(\sigma)}^{s,p}(\mathbb{R}^N).
\end{equation*}

\end{proof}

We now turn our attention to the issue of the density of smooth functions in the screened spaces.

\begin{theorem}\label{theorem smooth density}
Let $1 \le p < \infty$ and $0 < s < 1$.  Suppose that $\sigma : \mathbb{R}^N \to (0,\infty)$ is a screening function such that $0 < \inf_{\mathbb{R}^N} \sigma \le \sup_{\mathbb{R}^N} \sigma < \infty$.  Then $C^\infty(\mathbb{R}^N) \cap \W_{(\sigma)}^{s,p}(\mathbb{R}^N)$ is dense in $\W_{(\sigma)}^{s,p}(\mathbb{R}^N)$ and  $C^\infty(\mathbb{R}^N) \cap \W_{(\sigma)}^{s,p}(\mathbb{R}^N) / \mathbb{R}$ is dense in $\W_{(\sigma)}^{s,p}(\mathbb{R}^N)/ \mathbb{R}$.
\end{theorem}
\begin{proof}
In light of Theorem \ref{theorem screened equivalence}, it suffices to prove the result under the assumption that $\sigma =1$.  For $h \in B(0,1)$ write $\Delta_h g = g(\cdot +h) - g$ whenever $g \in L^1_{\operatorname*{loc}}(\mathbb{R}^N)$.  Since $\sigma =1$ we may use Tonelli's theorem to write 
\begin{equation*}
  |g|_{\W_{(1)}^{s,p}(\mathbb{R}^N)} = \left( \int_{B(0,1)} \frac{1}{|h|^{sp+N}} \Vert \Delta_h g \Vert_{L^p(\mathbb{R}^N)}^p dh \right)^{1/p}.
\end{equation*}

Let $\varphi \in C^\infty_c(\mathbb{R}^N)$ be such that $0 \le \varphi$ and $\int_{\mathbb{R}^N} \varphi \,dx=1$.  For $\varepsilon >0$ write $\varphi_\varepsilon(x) = \varepsilon^{-N} \varphi(x/\varepsilon)$.  Fix $f \in \W_{(\sigma)}^{s,p}(\mathbb{R}^N)$.  For $\varepsilon >0$ let $f_\varepsilon = f \ast \varphi_\varepsilon \in C^\infty(\mathbb{R}^N)$ be the usual mollification of $f$.   Clearly $(\Delta_h f) \ast \varphi_\varepsilon = \Delta_h (f \ast \varphi_\varepsilon) = \Delta_h f_\varepsilon$.  Thus the usual properties of mollifiers show that 
\begin{equation}\label{smooth density 1}
 \Vert \Delta_h f_\varepsilon  \Vert_{L^p(\mathbb{R}^N)} \le  \Vert \Delta_h f \Vert_{L^p(\mathbb{R}^N)}  \text{ and } \lim_{\varepsilon \to 0}  \Vert \Delta_h f_\varepsilon - \Delta_h f \Vert_{L^p(\mathbb{R}^N)} =  0.
\end{equation}

For $0 < \varepsilon$ set $q_\varepsilon,q : B(0,1) \backslash \{0\} \to [0,\infty)$ via 
\begin{equation*}
q_\varepsilon(h) = \frac{1}{|h|^{sp+N}}  \Vert \Delta_h f_\varepsilon - \Delta_h f \Vert_{L^p(\mathbb{R}^N)}^p \text{ and } q(h) = \frac{1}{|h|^{sp+N}}  \Vert \Delta_h f \Vert_{L^p(\mathbb{R}^N)}^p.
\end{equation*}
Then \eqref{smooth density 1} shows that if $h \in B(0,1) \backslash \{0\}$ then  $q_\varepsilon(h) \le 2^p q(h)$ and $q_\varepsilon(h) \to 0$ as $\varepsilon \to 0$.  Since 
\begin{equation*}
\int_{B(0,1)} q(h) dh =   |f|_{\W_{(1)}^{s,p}(\mathbb{R}^N)}^p < \infty
\end{equation*}
we can then apply the dominated convergence theorem to see that 
\begin{equation*}
 0 = \lim_{\varepsilon \to 0} \int_{B(0,1)} q_\varepsilon(h) dx = \lim_{\varepsilon \to 0} |f_\varepsilon - f|_{\W_{(1)}^{s,p}(\mathbb{R}^N)}^p. 
\end{equation*}
The stated density results then follow directly from this.

\end{proof}

When the screening function is constant and $p=2$ we can use Fourier analysis techniques to arrive at another characterization of the seminorm on $\tilde{W}_{(\sigma)}^{s,2}(\mathbb{R}^N)$.  The proof is essentially the same of the one given by Strichartz (see Theorem 2.2 in \cite{strichartz2016}) in the case $N=2$ and $s =1/2$. We present it here for the convenience of the reader.

\begin{proposition}\label{proposition fourier}
Let $0 < s < 1$ and $\sigma : \mathbb{R}^N \to (0,\infty)$ be a screening function such that $0 < \inf_{\mathbb{R}^N} \sigma \le \sup_{\mathbb{R}^N} \sigma < \infty$.  Then there exists a constant $c=c(N,s,\sigma)>0$ such that if $f: \mathbb{R}^N \to \mathbb{R}$ is Schwartz class, then
\begin{equation*}
c^{-1}\int_{\mathbb{R}^{N}}\min\{|\xi|^{2s},|\xi|^{2}\}|\hat {f}(\xi)|^{2}d\xi   \leq 
|f|_{\W_{(\sigma)}^{s,2}(\mathbb{R}^N)}^2
  \leq c\int_{\mathbb{R}^{N}}\min\{|\xi|^{2s},|\xi|^{2}\}|\hat{f}(\xi)|^{2}d\xi,
\end{equation*}
where $\hat{f}$ is the Fourier transform of $f$.
\end{proposition}
\begin{proof}

In light of Theorem \ref{theorem screened equivalence}, it suffices to prove the result under the assumption that $\sigma =1$.  By Tonelli's theorem, Plancherel's theorem, and standard properties of the Fourier transform, we may compute
\begin{align*}
|f|_{\W_{(\sigma)}^{s,2}(\mathbb{R}^N)}^2 & = \int_{B(0,1)}    \frac{1}{|h|^{N+2s}}\int_{\mathbb{R}^{N}}|f(x+h)-f(x)|^{2}dxdh \\
&  =\int_{B(0,1)}\frac{1}{|h|^{N+2s}}\int_{\mathbb{R}^{N}}|\hat{f}(\xi)|^{2}|e^{2\pi ih\cdot\xi}-1|^{2} d\xi dh\\
&  =\int_{\mathbb{R}^{N}}|\hat{f}(\xi)|^{2}\int_{B(0,1)}\frac{|e^{2\pi ih\cdot\xi}-1|^{2}}{|h|^{N+2s}}dhd\xi.
\end{align*}
Define $m: \mathbb{R}^N \to [0,\infty)$ via 
\begin{equation*}
m(\xi)=\int_{B(0,1)}\frac{|e^{2\pi ih\cdot \xi}-1|^{2}}{|h|^{N+2s}}dh.
\end{equation*}
Note that if $\lambda>0$ and $|\xi|=1$, then by the change of variables $\zeta=\lambda h$ we get $d\zeta=\lambda^{N}dh$ and
\begin{equation*}
m(\lambda\xi)=\lambda^{2s} \int_{B(0,\lambda )}\frac{|e^{2\pi i \zeta\cdot\xi} -1|^{2}}{|\zeta|^{N+2s}} d\zeta 
= \lambda^{2s} \int_{B(0,\lambda )}\frac{|e^{2\pi i\zeta_{1}}-1|^{2}}{|\zeta|^{N+2s}}d\zeta,
\end{equation*}
where the second equality follows by invariance under rotation. It follows that for all $\lambda\geq 1/2$,
\begin{equation*}
c_{1}:=\int_{B(0,1/2)}\frac{|e^{2\pi i\zeta_{1}}-1|^{2}}{|\zeta|^{N+2s}}d\zeta \leq 
\frac{m(\lambda\xi)}{\lambda^{2s} } \leq\int_{\mathbb{R}^{N}}\frac{|e^{2\pi i\zeta_{1}}-1|^{2}}{|\zeta|^{N+1}}d\zeta=:c_{2}<\infty,
\end{equation*}
which shows that
\begin{equation}\label{m1}
c_{1}|\xi|^{2s} \leq m(\xi)\leq c_{2}|\xi|^{2s} 
\end{equation}
for all $\xi\in\mathbb{R}^{N}$ with $|\xi|\geq 1/2$.

On the other hand, standard properties of the complex exponential show that
\begin{equation*}
\pi|t|\leq|e^{2\pi it}-1|\leq 2\pi|t|
\end{equation*}
for all $|t|\leq 1/2$, and so
\begin{equation}\label{m2}
c_{3}\leq\frac{m(\xi)}{|\xi|^{2}}\leq 4 c_{3} 
\end{equation}
for all $\xi\in\mathbb{R}^{N}\setminus\{0\}$ with $|\xi|\leq 1/2$, where, again by invariance under rotation,
\begin{equation*}
c_{3}:=\frac{\pi^{2}}{|\xi|^{2}}\int_{B(0,1)}\frac{(h\cdot\xi)^{2}}{|h|^{N+2s}}dh
=\pi^{2} \int_{B(0,1)}\frac{h_{1}^{2}}{|h|^{N+2s}}dh<\infty.
\end{equation*}
Since $m(\xi)>0$ for all $\xi\in\mathbb{R}^{N}\setminus\{0\}$ and $m$ is continuous, it follows from \eqref{m1} and \eqref{m2} that there is a constant $c=c(N,s)>0$ such that
\begin{equation*}
c^{-1}\min\{|\xi|^{2s},|\xi|^{2}\}\leq m(\xi)\leq c\min\{|\xi|^{2s},|\xi|^{2}\}
\end{equation*}
for all $\xi\in\mathbb{R}^{N}$.  The stated estimate follows directly from this.
\end{proof}

\begin{remark}
For $0 < s < 1$ the Fourier version of the seminorm on $\dot{W}^{s,2}(\mathbb{R}^N)$ is 
\begin{equation*}
 \int_{\mathbb{R}^N} |\xi|^{2s} |\hat{f}(\xi)|^2 d\xi \asymp |f|_{\dot{W}^{s,2}(\mathbb{R}^N) }^2.
\end{equation*}
Comparing this with Proposition \ref{proposition fourier} shows that in the screened spaces control is lost for low frequencies.
\end{remark}

\subsection{Trace spaces}
\label{subsection trace spaces}

We have seen in Theorems \ref{theorem trace strip} and \ref{theorem lifting strip} that when $\Omega$ has the form \eqref{omega infinite strip} the trace space of $\dot{W}^{1,p}(\Omega)$ is given by pairs of functions in $\W_{(\sigma)}^{s,p}(\mathbb{R}^{N-1})$ whose difference is in $L^{p}(\mathbb{R}^{N-1})$, where $\sigma \equiv a$ for some $0<a\leq b$. A similar result holds when $\Omega$ is of the form \ref{omega two graphs} (see Theorems \ref{theorem trace m=1} and \ref{theorem lifting m=1}) Thus, we need to study these types of spaces.

\begin{definition}\label{definition space X}
Given a screening function $\sigma:\mathbb{R}^{N-1}\rightarrow(0,\infty)$, $1\leq p<\infty$, $0<s<1$, we define the space $\dot{X}_{(\sigma)}^{s,p}(\mathbb{R}^{N-1})$ as the space of all pairs of functions $(f^{-},f^{+})$ such that $f^{\pm}\in\W_{(\sigma)}^{s,p}(\mathbb{R}^{N-1})$ and
\begin{equation*}
\int_{\mathbb{R}^{N-1}}\frac{|f^{+}(x^{\prime})-f^{-}(x^{\prime})|^{p}}{(\sigma(x^{\prime}))^{p-1}}\,dx^{\prime}<\infty.
\end{equation*}
We set
\begin{equation*}
|(f^{-},f^{+})|_{\dot{X}_{(\sigma)}^{s,p}(\mathbb{R}^{N-1})}
:=\left(\int_{\mathbb{R}^{N-1}}\frac{|f^{+}(x^{\prime})-f^{-}(x^{\prime})|^{p}}
{(\sigma(x^{\prime}))^{p-1}}\,dx^{\prime}\right)  ^{1/p}
+|f^{-}|_{\W_{(\sigma)}^{s,p}(\mathbb{R}^{N-1})}
+|f^{+}|_{\W_{(\sigma)}^{s,p}(\mathbb{R}^{N-1})}.
\end{equation*}
Note that we do not allow the screening function to take the value $+\infty$ in this definition in order to force $1/\sigma^{p-1} >0$.
\end{definition}

Thanks to Proposition \ref{proposition seminorm} we know that if $|(f^{-},f^{+})|_{\dot{X}_{(\sigma)}^{s,p}(\mathbb{R}^{N-1})}=0$, then there exist $c^{\pm}\in\mathbb{R}$ such that $f^{\pm}(x^{\prime})=c^{\pm}$ for
$\mathcal{L}^{N-1}$ a.e. $x^{\prime}\in\mathbb{R}^{N-1}$, but in view of the first term in $|\cdot|_{\dot{X}_{(\sigma)}^{s,p}(\mathbb{R}^{N-1})}$, it must then hold that $c^{+}=c^{-}$.  Thus, $(f^{-},f^{+})\sim(g^{-},g^{+})$ if and only if there exists $c\in\mathbb{R}$ such that $f^{\pm}(x^{\prime})-g^{\pm
}(x^{\prime})=c$ for $\mathcal{L}^{N-1}$ a.e. $x^{\prime}\in\mathbb{R}^{N-1}$.  With a slight abuse of notation we write the quotient space $\dot{X}_{(\sigma)}^{s,p}(\mathbb{R}^{N-1})/K$ (see Section
\ref{subsection seminormed spaces}) as $\dot{X}_{(\sigma)}^{s,p}(\mathbb{R}^{N-1})/\mathbb{R}$.  Our next result establishes that $\dot{X}_{(\sigma)}^{s,p}(\mathbb{R}^{N-1})/\mathbb{R}$ is complete.

\begin{theorem}\label{theorem space X refelxive}
Let $\sigma:\mathbb{R}^{N-1}\rightarrow (0,\infty)$ be a screening function, let $1\leq p<\infty$, and let $0<s<1$. Then the space $\dot{X}_{(\sigma)}^{s,p}(\mathbb{R}^{N-1})/\mathbb{R}$ is a Banach space
with the norm
\begin{equation*}
\Vert\lbrack(f^{-},f^{+})]\Vert_{\dot{X}_{(\sigma)}^{s,p}(\mathbb{R}^{N-1})/\mathbb{R}}
:=|(f^{-},f^{+})|_{\dot{X}_{(\sigma)}^{s,p}(\mathbb{R}^{N-1})}.
\end{equation*}
Moreover, if $1<p<\infty$, then $\dot{X}_{(\sigma)}^{s,p}(\mathbb{R}^{N-1})/\mathbb{R}$ is reflexive.
\end{theorem}

\begin{proof}
We divide the proof into steps.

\textbf{Step 1 -- Limits of Cauchy sequences:}  According to the discussion in Section \ref{subsection seminormed spaces}, in order to prove that $\dot{X}_{(\sigma)}^{s,p}(\mathbb{R}^{N-1})/\mathbb{R}$ is a Banach space, it is enough to show that $\dot{X}_{(\sigma)}^{s,p}(\mathbb{R}^{N-1})$ is sequentially complete.  Let $\{(f_{n}^{-},f_{n}^{+})\}_{n=1}^\infty$ be a Cauchy sequence in $\dot{X}_{(\sigma)}^{s,p}(\mathbb{R}^{N-1})$.  We then know that  $\{f_{n}^{-}\}_{n=1}^\infty$ and $\{f_{n}^{+}\}_{n=1}^\infty$ are Cauchy sequences in $\W_{(\sigma)}^{s,p}(\mathbb{R}^{N-1})$.  In addition, we know that  $\{f_{n}^{+}-f_{n}^{-}\}_{n=1}^\infty$ is a Cauchy sequence in $L^{1}(\mathbb{R}^{N-1};\nu)$, where $\nu:\mathcal{B}(\mathbb{R}^{N-1})\rightarrow\lbrack0,\infty]$ is the measure given by
\begin{equation*}
\nu(E):=\int_{E}\frac{1}{(\sigma(x^{\prime}))^{p-1}}\,dx^{\prime}.
\end{equation*}
Using Theorem \ref{theorem screened complete} and the completeness of $L^p(\mathbb{R}^{N-1};\nu)$, we deduce that there exist a pair $f^{\pm} \in \W_{(\sigma)}^{s,p}(\mathbb{R}^{N-1})$ and a function $g\in L^{1}(\mathbb{R}^{N-1};\nu)$ such that $f_{n}^{\pm}\rightarrow f^{\pm}$ in $\W_{(\sigma)}^{s,p}(\mathbb{R}^{N-1})$ and $f_{n}^{+}-f_{n}^{-}\rightarrow g$ in $L^{1}(\mathbb{R}^{N-1};\nu)$ as $n \to \infty$.

\textbf{Step 2 -- Identifying the limits:}  We will use the same notation for averages as in the proof of Theorem \ref{theorem screened complete}.  Let $\{V_k\}_{k=1}^\infty$ be the sequence of open sets given by Lemma \ref{lemma exhaustion} for $\Omega = \mathbb{R}^{N-1}$.  The Lemma allows us to apply Theorem \ref{full_poincare} on each set $V_k$ in order to see that 
\begin{equation*}
 \Vert (f_n^\pm - (f_n^\pm)_{V_1}) - (f^\pm - (f^\pm)_{V_1} )  \Vert_{L^p(V_k)} \le c_{k}  |f_n^\pm - f^\pm|_{\W_{(\sigma)}^{s,p}(V_k)} \le c_{k}  |f_n^\pm - f^\pm|_{\W_{(\sigma)}^{s,p}(\mathbb{R}^{N-1})}.
\end{equation*}
Consequently, for each $k \in \mathbb{N}$ we have that $f_n^\pm - (f_n^\pm)_{V_1} \to f^\pm - (f^\pm)_{V_1}$ in $L^p(V_k)$ as $n \to \infty$.

Up to the extraction of a subsequence, which we continue to denote by $\{f_n^\pm\}_{n=1}^\infty$ for the sake of brevity, we may assume that there exists a measurable set $D \subseteq \mathbb{R}^{N-1}$ with $\mathcal{L}^{N-1}(D) =0$ such that 
\begin{equation*}
f_n^\pm - (f_n^\pm)_{V_1} \to f^\pm - (f^\pm)_{V_1} \text{ and } f_n^+ - f_n^- \to g \text{ pointwise on } \mathbb{R}^{N-1}\backslash D \text{ as }n \to \infty.
\end{equation*}
Choosing any $x \in \mathbb{R}^{N-1} \backslash D$, we may then compute 
\begin{equation*}
\begin{split}
 (f_n^-)_{V_1} - (f_n^+)_{V_1} & = [f_n^+(x) - (f_n^+)_{V_1}  ] -  [f_n^-(x) - (f_n^-)_{V_1}  ] - [f_n^+(x) - f_n^-(x)]\\
&\to  [f^+(x) - (f^+)_{V_1}  ] -  [f^-(x) - (f^-)_{V_1}  ] - g(x) \text{ as }n \to \infty.
\end{split}
\end{equation*}
From this we deduce that there exists a constant $C \in \mathbb{R}$ such that 
\begin{equation*}
 f^+ - f^- - g = C \text{ almost everywhere in } \mathbb{R}^{N-1}.
\end{equation*}
Now, the fact that $f_{n}^{-}\rightarrow f^{-}$ in $\W_{(\sigma)}^{s,p}(\mathbb{R}^{N-1})$ also implies that $f_{n}^{-}\rightarrow f^{-}+C$ in $\W_{(\sigma)}^{s,p}(\mathbb{R}^{N-1})$.  This allows us to make the replacement $f^- \mapsto f^- +C$ in order to see that
\begin{equation*}
 g = \lim_{n \to \infty} f_n^+ - \lim_{n \to \infty} f_n^-  \in L^1(\mathbb{R}^{N-1};\nu),
\end{equation*}
from which we then deduce that  $(f^{-},f^{+})\in \dot{X}_{(\sigma)}^{s,p}(\mathbb{R}^{N-1})$ and $\dot{X}_{(\sigma)}^{s,p}(\mathbb{R}^{N-1})$ is sequentially complete.

\textbf{Step 3 -- Reflexivity:} Assume that $1<p<\infty$. Write 
\begin{equation*}
 \Gamma = \{ (x,h) \in \mathbb{R}^{N-1} \times \mathbb{R}^{N-1} \;\vert\; |h| < \sigma(x)\} \subseteq \mathbb{R}^{2(N-1)}.
\end{equation*}
Consider the mapping
\begin{equation*}
\dot{X}_{(\sigma)}^{s,p}(\mathbb{R}^{N-1})/\mathbb{R}  \ni \lbrack(f^{-},f^{+})]  \mapsto 
(v_{f^{-}},v_{f^{+}},f^{+}-f^{-}) \in 
L^{p}(\Gamma;\mu)\times L^{p}(\Gamma;\mu)\times L^{p}(\mathbb{R}^{N-1};\nu),
\end{equation*}
where $\mu$ is the measure defined in \eqref{b measure} and $v_{f^{\pm}}(x^{\prime},h^{\prime})=f^{\pm}(x^{\prime}+h^{\prime})-f^{\pm}(x^{\prime})$.  This map is an isometric embedding (see \eqref{b isomorphism}), and by the previous two steps we can identify $\dot{X}_{(\sigma)}^{s,p}(\mathbb{R}^{N-1})/\mathbb{R}$ with its image, which is a closed subspace of the Cartesian product of three reflexive spaces.  Since the product of reflexive spaces is reflexive and closed subspaces of reflexive spaces are reflexive, it follows that $\dot{X}_{(\sigma)}^{s,p}(\mathbb{R}^{N-1})/\mathbb{R}$ is reflexive.
\end{proof}

In view of the previous proposition and of the discussion in Section \ref{subsection seminormed spaces}, when $1<p<\infty$ a bounded sequence in $\dot{X}_{(\eta)}^{s,p}(\mathbb{R}^{N-1})$ admits a subsequence that converges weakly in $\dot{X}_{(\eta)}^{s,p}(\mathbb{R}^{N-1})$.  In the next theorem we study the relationship between this weak limit and the weak limit with respect to $L_{\operatorname*{loc}}^{p}(\mathbb{R}^{N-1})$ convergence.

\begin{theorem}
Let $\sigma:\mathbb{R}^{N-1}\rightarrow(0,\infty)$ be a screening function, $1<p<\infty$, and $0<s<1$.  Let $(f_{n}^{-},f_{n}^{+})\in\dot{X}_{(\sigma)}^{s,p}(\mathbb{R}^{N-1})$ for $n\in\mathbb{N}$, and let $(f^{-},f^{+})\in\dot{X}_{(\sigma)}^{s,p}(\mathbb{R}^{N-1})$ be such that $(f_{n}^{-},f_{n}^{+}) \rightharpoonup (f^{-},f^{+})$ in $\dot{X}_{(\sigma)}^{s,p}(\mathbb{R}^{N-1})$ and $f_{n}^{\pm}\rightharpoonup g^{\pm}$ in $L_{\operatorname*{loc}}^{p}(\mathbb{R}^{N-1})$ as $n \to \infty$. Then there exist $c^{\pm}\in\mathbb{R}$ such that $f^{\pm}(x^{\prime})-g^{\pm}(x^{\prime})=c^{\pm}$ for $\mathcal{L}^{N-1}$ a.e. $x^{\prime}\in\mathbb{R}^{N-1}$.  Moreover, if
\begin{equation}\label{b2 a}
\int_{\mathbb{R}^{N-1}}\frac{1}{(\sigma(x^{\prime}))^{p-1}}\,dx^{\prime}=\infty,
\end{equation}
then $(f^{-},f^{+})\sim(g^{-},g^{+})$.
\end{theorem}

\begin{proof}
Let $\{V_k\}_{k=1}^\infty$ be the sequence of open subsets of $\mathbb{R}^{N-1}$ given by Lemma \ref{lemma exhaustion}.  Fix $k \in \mathbb{N}$ and suppose that $\varphi^{\pm}\in C_{c}^{\infty}(\mathbb{R}^{N-1})$ are such that $\operatorname*{supp}\varphi^{\pm} \subset V_k$ and $\int_{\mathbb{R}^{N-1}}\varphi^{\pm}(x^{\prime})\,dx^{\prime}=0$.  Using these, we define the linear functional $T :\dot{X}_{(\sigma)}^{s,p}(\mathbb{R}^{N-1}) \to \mathbb{R}$ via 
\begin{equation}
T(h^{-},h^{+}):=\int_{\mathbb{R}^{N-1}}h^{-}(x^{\prime})\varphi^{-}(x^{\prime})\,dx^{\prime}
+\int_{\mathbb{R}^{N-1}}h^{+}(x^{\prime})\varphi^{+}(x^{\prime})\,dx^{\prime}. 
\end{equation}
By  H\"{o}lder's inequality, the Poincar\'{e} inequality of Theorem \ref{full_poincare} (which is applicable due to Lemma \ref{lemma exhaustion}), and the fact that $\int_{\mathbb{R}^{N-1}}\varphi^{\pm}(x^{\prime})\,dx^{\prime}=0$, we have the estimate
\begin{align*}
|T(h^{-},h^{+})|  & \leq \int_{\mathbb{R}^{N-1}} |h^{-}(x^{\prime})-h_{V_1}^{-}| |\varphi^{-}(x^{\prime})|\, dx^{\prime}
+ \int_{\mathbb{R}^{N-1}} |h^{+}(x^{\prime})-h_{V_1}^{+}| |\varphi^{+}(x^{\prime})| \,dx^{\prime}\\
&  \leq  \Vert \varphi^{-} \Vert_{L^{p^{\prime}}(\mathbb{R}^{N-1})}
   \Vert h^{-}-h_{V_1}^{-} \Vert_{L^{p}(V_k}
   + \Vert\varphi^{+} \Vert_{L^{p^{\prime}}(\mathbb{R}^{N-1})}
   \Vert h^{+}-h_{V_1}^{+}\Vert_{L^{p}(V_k)}\\
&  \leq c|h^{-}|_{\W_{(\sigma)}^{s,p}(\mathbb{R}^{N-1})}
+c|h^{+}|_{\W_{(\sigma)}^{s,p}(\mathbb{R}^{N-1})} 
\leq c|(h^{-},h^{+})|_{\dot{X}_{(\sigma)}^{s,p}(\mathbb{R}^{N-1})}.
\end{align*}
This shows that $T\in(\dot{X}_{(\eta)}^{s,p}(\mathbb{R}^{N-1}))^{\prime}$, and it follows immediately that 
\begin{equation}\label{b3}
T(f_{n}^{-},f_{n}^{+})\rightarrow T(f^{-},f^{+}) \text{ as }n \to \infty.
\end{equation}

On the other hand, the functionals $T^\pm : L^p(V_k) \to \mathbb{R}$ defined by 
\begin{equation*}
T^{\pm}(h^{\pm}):=\int_{V_k}h^{\pm}(x^{\prime})\varphi^{\pm}(x^{\prime})\,dx^{\prime}
\end{equation*}
are continuous. Hence, by hypothesis $T^{\pm}(f_{n}^{\pm})\rightarrow T^{\pm}(g^{\pm})$ as $n \to \infty$. Comparing with  \eqref{b3}, we arrive at the identity
\begin{equation*}
\int_{V_k}(f^{-}-g^{-})(x^{\prime})\varphi^{-}(x^{\prime
})\,dx^{\prime}
+\int_{V_k}(f^{+}-g^{+})(x^{\prime})\varphi
^{+}(x^{\prime})\,dx^{\prime}=0.
\end{equation*}
Given the arbitrariness of $\varphi^{\pm}$, we deduce that there exist $c_k^{\pm}\in\mathbb{R}$ such that $f^{\pm}(x^{\prime})-g^{\pm}(x^{\prime})=c_k^{\pm}$ for $\mathcal{L}^{N-1}$ a.e. $x^{\prime}\in V_k$.  However, Lemma \ref{lemma exhaustion} guarantees that $V_k \subset V_{k+1}$ for $k \in \mathbb{N}$ and that $\mathbb{R}^{N-1} = \bigcup_{k=1}^\infty V_K$, so we deduce the existence of a single pair of constant $c^\pm \in \mathbb{R}$ such that $c_k^\pm = c^\pm$ for all $k \in \mathbb{N}$.  Hence, $f^\pm(x') - g^\pm(x') = c^\pm$ for $\mathcal{L}^{N-1}$ a.e. $x^{\prime}\in\mathbb{R}^{N-1}$.  This proves the first assertion.

Now assume that \eqref{b2 a} holds. Then since $f_{n}^{\pm}\rightharpoonup g^{\pm}$ in $L_{\operatorname*{loc}}^{p}(\mathbb{R}^{N-1})$ as $n \to \infty$,  standard lower semicontinuity results imply that for each $k \in \mathbb{N}$
\begin{align*}
\int_{V_k}\frac{|g^{+}(x^{\prime})-g^{-}(x^{\prime})|^{p}}{(\sigma(x^{\prime}))^{p-1}}\,dx^{\prime} 
& \leq \liminf_{n\rightarrow\infty} 
\int_{V_k} \frac{|f_n^{+}(x^{\prime})-f_n^{-}(x^{\prime})|^{p}}{(\sigma(x^{\prime}))^{p-1}}\,dx^{\prime}<\infty \\
& \leq \liminf_{n\rightarrow\infty} 
\int_{\mathbb{R}^{N-1}} \frac{|f_n^{+}(x^{\prime})-f_n^{-}(x^{\prime})|^{p}}{(\sigma(x^{\prime}))^{p-1}}\,dx^{\prime} := M^p <\infty.
\end{align*}
Thus
\begin{align*}
|c^{+}-c^{-}|\Vert\sigma^{-(p-1)/p}\Vert_{L^{p}(V_k)}  &  \le
\Vert((f^{+}-g^{+})-(f^{-}-g^{-}))\sigma^{-(p-1)/p}\Vert_{L^{p}(V_k)}\\
&  \leq\Vert(f^{+}-f^{-})\sigma^{-(p-1)/p}\Vert_{L^{p}(V_k)}
+ \Vert(g^{+}-g^{-})\sigma^{-(p-1)/p}\Vert_{L^{p}(V_k)} \\
&  \leq\Vert(f^{+}-f^{-})\sigma^{-(p-1)/p}\Vert_{L^{p}(\mathbb{R}^{N-1})}
+ M <\infty,
\end{align*}
which implies, upon sending $k \to \infty$, that $c^{+}=c^{-}$.
\end{proof}

\section{Traces in the strip case}\label{section traces strip}

In this section we consider the special case in which $\Omega$ is a horizontal strip of the form \eqref{omega infinite strip}.

\subsection{The case $m=1$, $p>1$}

In this subsection we prove Theorems \ref{theorem trace strip} and  \ref{theorem lifting strip} and then derive some extra information about homogeneous Sobolev spaces.

\begin{proof}[Proof of Theorem \ref{theorem trace strip}]

For the sake of simplicity, we take $b^{-}=0$ and we set $b:=b^{+}$.

Due to the absolute continuity results of Theorem \ref{theorem AC}, we may apply the fundamental theorem of calculus to see that
\begin{equation*}
u(x^{\prime},b)-u(x^{\prime},0)=\int_{0}^{b}\partial_{N}u(x^{\prime},x_{N})\,dx_{N}
\end{equation*}
for $\mathcal{L}^{N-1}$ a.e. $x^{\prime}\in\mathbb{R}^{N-1}$. It then follows by H\"{o}lder's inequality that
\begin{equation*}
|u(x^{\prime},b)-u(x^{\prime},0)|^{p}\leq b^{p-1}\int_{0}^{b}|\partial
_{N}u(x^{\prime},x_{N})|^{p}dx_{N}.
\end{equation*}
We then integrate in $x^{\prime}$ over $\mathbb{R}^{N-1}$ and using Tonelli's theorem to obtain the estimate
\begin{equation*}
\int_{\mathbb{R}^{N-1}}|u(x^{\prime},b)-u(x^{\prime},0)|^{p}dx^{\prime}\leq
b^{p-1}\int_{\Omega}|\partial_{N}u(x)|^{p}dx.
\end{equation*}
This proves \eqref{theorem trace strip est 1}.

We will prove \eqref{theorem trace strip est 2} only for $\Gamma^{-}$, as a similar argument establishes the corresponding bound on $\Gamma^+$.  We again use the fundamental theorem of calculus (due to Theorem \ref{theorem AC}) to see that for $\mathcal{L}^{N-1}$ a.e. $x^{\prime}\in\mathbb{R}^{N-1}$, $h^{\prime}\in\mathbb{R}^{N-1}$ with $|h^{\prime}|<b$, and $0<x_{N}<b$ we
have that
\begin{align*}
u(x^{\prime}+h^{\prime},0)-u(x^{\prime},0)  &  =u(x^{\prime}+h^{\prime}%
,x_{N})-u(x^{\prime},x_{N})-\int_{0}^{x_{N}}\partial_{N}u(x^{\prime}%
+h^{\prime},s)\,ds\\
&  \quad-\int_{0}^{x_{N}}\partial_{N}u(x^{\prime},s)\,ds=\int_{0}^{1}%
\nabla_{\shortparallel}u(x^{\prime}+sh^{\prime},x_{N})\cdot h^{\prime}ds\\
&  \quad-\int_{0}^{x_{N}}\partial_{N}u(x^{\prime}+h^{\prime},s)\,ds-\int%
_{0}^{x_{N}}\partial_{N}u(x^{\prime},s)\,ds.
\end{align*}
From this we may estimate
\begin{align*}
|u(x^{\prime}+h^{\prime},0)-u(x^{\prime},0)|  &  \leq|h^{\prime}|\int_{0}%
^{1}|\nabla_{\shortparallel}u(x^{\prime}+sh^{\prime},x_{N})|ds\\
&  \quad+\int_{0}^{x_{N}}|\partial_{N}u(x^{\prime}+h^{\prime},s)|\,ds+\int%
_{0}^{x_{N}}|\partial_{N}u(x^{\prime},s)|\,ds.
\end{align*}
We now take the $L^{p}$ norm in $x^{\prime}$ over $\mathbb{R}^{N-1}$ on both sides of this inequality and apply Minkowksi's inequality for integrals (see, for instance, Corollary B.83 in \cite{leoni2017book});  using the change of variables $y^{\prime}=x^{\prime}+sh^{\prime}$ and $y^{\prime}=x^{\prime}+h^{\prime}$ in the first integral and second integral on the resulting right-hand-side, we then obtain the bound
\begin{equation*}
\Vert u(\cdot+h^{\prime},0)-u(\cdot,0)\Vert_{L^{p}(\mathbb{R}^{N-1})}%
\leq|h^{\prime}|\Vert\nabla_{\shortparallel}u(\cdot,x_{N})\Vert_{L^{p}%
(\mathbb{R}^{N-1})}+2\int_{0}^{x_{N}}\Vert\partial_{N}u(\cdot,s)\Vert
_{L^{p}(\mathbb{R}^{N-1})}ds.
\end{equation*}
Next, we average in $x_{N}$ over the interval $[0,|h^{\prime}|]$ to see that
\begin{equation}
\Vert u(\cdot+h^{\prime},0)-u(\cdot,0)\Vert_{L^{p}(\mathbb{R}^{N-1})}\leq
3\int_{0}^{|h^{\prime}|}\Vert\nabla u(\cdot,x_{N})\Vert_{L^{p}(\mathbb{R}^{N-1})}dx_{N}. \label{1}
\end{equation}
In turn, this implies that
\begin{equation*}
\frac{\Vert u(\cdot+h^{\prime},0)-u(\cdot,0)\Vert_{L^{p}(\mathbb{R}^{N-1})}^{p}}{|h^{\prime}|^{p+N-2}}
\leq\frac{3^{p}}{|h^{\prime}|^{p+N-2}} \left(\int_{0}^{|h^{\prime}|}\Vert\nabla u(\cdot,x_{N})\Vert_{L^{p}(\mathbb{R}^{N-1})}dx_{N}\right)^{p}.
\end{equation*}
We then integrate in $h^{\prime}$ over $B^{\prime}(0,b)$ and use spherical coordinates to estimate
\begin{align*}
\int_{B^{\prime}(0,b)}  &  \frac{\Vert u(\cdot+h^{\prime},0)-u(\cdot
,0)\Vert_{L^{p}(\mathbb{R}^{N-1})}^{p}}{|h^{\prime}|^{p+N-2}}dh^{\prime}\\
&  \leq\int_{B^{\prime}(0,b)}\frac{3^{p}}{|h^{\prime}|^{p+N-2}}\left(
\int_{0}^{|h^{\prime}|}\Vert\nabla u(\cdot,x_{N})\Vert_{L^{p}(\mathbb{R}%
^{N-1})}dx_{N}\right)  ^{p}dh^{\prime}\\
&  =3^{p}\beta_{N-1}\int_{0}^{b}\frac{1}{r^{p}}\left(  \int_{0}^{r}\Vert\nabla
u(\cdot,x_{N})\Vert_{L^{p}(\mathbb{R}^{N-1})}dx_{N}\right)  ^{p}dr.
\end{align*}
Applying Hardy's inequality to the right-hand side (see, for instance, Theorem C.41 in \cite{leoni2017book}) and using Tonelli's theorem then yields the bound
\begin{equation*}
\int_{B^{\prime}(0,b)}\frac{\Vert u(\cdot+h^{\prime},0)-u(\cdot,0)
\Vert_{L^{p}(\mathbb{R}^{N-1})}^{p}}{|h^{\prime}|^{p+N-2}}dh^{\prime}
\leq \frac{3^{p}\beta_{N-1}p^{p}}{(p-1)^{p}}\int_{\Omega}|\nabla u(x)|^{p}dx,
\end{equation*}
which completes the proof of \eqref{theorem trace strip est 2} and thus of the first two items of the theorem.

It remains to prove the integration-by-parts formula of the third item.  For this it suffices to note that if $u \in \dot{W}^{1,p}(\Omega)$ and $\psi \in C^1_c(\mathbb{R}^N)$, then $u \in W^{1,p}(U)$ for any open set $U \subseteq \Omega$ with Lipschitz boundary such that $\operatorname{supp}(\psi) \cap \Omega \subset U$.  Thus the third item follows immediately from the usual trace theory in $W^{1,p}(U)$.
\end{proof}

Next we prove Theorem \ref{theorem lifting strip}.

\begin{proof}[Proof of Theorem \ref{theorem lifting strip}]
For the sake of simplicity, we again take $b^{-}=0$ and we set $b:=b^{+}$. Without loss of generality we may assume that $0<a<b$. Let $\varphi\in C_{c}^{\infty}(\mathbb{R}^{N-1})$ be a nonnegative function such that $\int_{\mathbb{R}^{N-1}}\varphi(y^{\prime})\,dy^{\prime}=1$ and $\operatorname*{supp}\varphi\subseteq B^{\prime}(0,ab^{-1})$. Define $u^-: \Omega \to \mathbb{R}$ via
\begin{equation*}
u^{-}(x) =(\varphi_{x_{N}}\ast f^{-})(x^{\prime})
=\int_{\mathbb{R}^{N-1}} f^{-}(y^{\prime})\varphi_{x_{N}}\left(  x^{\prime}-y^{\prime}\right)\,dy^{\prime},
\end{equation*}
where $\varphi_{x_{N}}$ is given in \eqref{mollifier}. Then for $i=1,\ldots,N$ we have that
\begin{equation*}
\partial_{i}u^{-}(x)=\int_{\mathbb{R}^{N-1}}f^{-}(y^{\prime})\frac{\partial}{\partial x_{i}}(\varphi_{x_{N}}\left(  x^{\prime}-y^{\prime}\right))\,dy^{\prime}.
\end{equation*}
In view of Proposition \ref{proposition derivatives mollifiers}, we are in a position to apply Proposition \ref{proposition structure} to conclude that
\begin{equation*}
\int_{\Omega}|\partial_{i}u^{-}(x)|^{p}dx
\leq c\int_{\mathbb{R}^{N-1}} \int_{B^{\prime}(x^{\prime},a)} \frac{|f^{-}(y^{\prime})-f^{-}(x^{\prime})|^{p}}{|x^{\prime}-y^{\prime}|^{N+p-2}}\,dy^{\prime}dx^{\prime}.
\end{equation*}
This shows that $u^{-}\in\dot{W}^{1,p}(\Omega)$. Moreover, since by standard properties of mollifiers $\varphi_{x_{N}}\ast f^{-}\rightarrow f^{-}$ in $L_{\operatorname*{loc}}^{p}(\mathbb{R}^{N-1})$ as $x_{N}\rightarrow0^{+}$, we have that $\operatorname*{Tr}(u^{-})=f^{-}$ on $\Gamma^{-}$.  Similarly, if we define $u^+: \Omega \to \mathbb{R}$ via 
\begin{equation*}
u^{+}(x)=(\varphi_{b-x_{N}}\ast f^{+})(x^{\prime}),
\end{equation*}
then a similar argument shows that $u^{+}\in\dot{W}^{1,p}(\Omega)$ with $\operatorname*{Tr}(u^{+})=f^{+}$ on $\Gamma^{+}$.

Now let $\theta\in C^{\infty}([0,b])$ be such that $\theta=1$ in a neighborhood of $0$ and $\theta=0$ in a neighborhood of $b$, and define $u : \Omega \to \mathbb{R}$ via
\begin{equation*}
u(x)=\theta(x_{N})u^{-}(x)+(1-\theta(x_{N}))u^{+}(x).
\end{equation*}
It is clear by construction that $\operatorname*{Tr}(u)=f^{+}$ on $\Gamma^+$ and $\operatorname*{Tr}(u)=f^{-}$ on $\Gamma^-$ and that the map $(f^+,f^-) \mapsto u$ is linear.  For $i=1,\ldots,N-1$ we compute
\begin{align*}
\partial_{i}u(x)  &  =\theta(x_{N})\partial_{i}u^{-}(x)+(1-\theta
(x_{N}))\partial_{i}u^{+}(x),\\
\partial_{N}u(x)  &  = \theta^{\prime}(x_{N})(u^{-}(x)-u^{+}(x))+\theta
(x_{N})\partial_{N}u^{-}(x)+(1-\theta(x_{N}))\partial_{N}u^{+}(x).
\end{align*}
Since $\nabla u^{\pm}\in L^{p}(\Omega)$ and $\theta$ is bounded, in order to prove that $\nabla u \in L^p(\Omega)$ it remains only to show that $\theta^{\prime}(u^{-}-u^{+}) \in L^{p}(\Omega)$.

By the fundamental theorem of calculus (which can be applied since $u^{-}$ and $u^{+}$ are absolutely continuous on $\mathcal{L}^{N-1}$ a.e. line parallel to $e_{N}$),
\begin{align*}
u^{-}(x)  &  =f^{-}(x^{\prime})+\int_{0}^{x_{N}}\partial_{N}u^{-}(x^{\prime
},s)\,ds,\\
u^{+}(x)  &  =f^{+}(x^{\prime})-\int_{x_{N}}^{b}\partial_{N}u^{+}(x^{\prime
},s)\,ds.
\end{align*}
Hence,
\begin{equation*}
|u^{+}(x)-u^{-}(x)|   \leq|f^{+}(x^{\prime})-f^{-}(x^{\prime})|+\int_{0}^{b}|\partial_{N}u^{-}(x^{\prime},s)|\,ds + \int_{0}^{b}|\partial_{N}u^{+}(x^{\prime},s)|\,ds.
\end{equation*}
We may then apply H\"{o}lder's inequality to see that
\begin{equation*}
|u^{+}(x)   -u^{-}(x)|^{p}\leq2^{p-1}|f^{+}(x^{\prime})-f^{-}(x^{\prime
})|^{p}
  +2^{p-1}b^{p-1}\int_{0}^{b}(|\partial_{N}u^{-}(x^{\prime},s)|^{p} + |\partial_{N}u^{+}(x^{\prime},s)|^{p})\,ds.
\end{equation*}
Since $|\theta^{\prime}(x_{N})|\leq cb^{-1}$, it follows that
\begin{equation*}
|\theta^{\prime}(x_{N})   (u^{+}(x)-u^{-}(x))|^{p}\leq cb^{-p} |f^{+}(x^{\prime})-f^{-}(x^{\prime})|^{p}
+cb^{-1}\int_{0}^{b}(|\partial_{N}u^{-}(x^{\prime},s)|^{p} + |\partial_{N}u^{+}(x^{\prime},s)|^{p})\,ds.
\end{equation*}
Integrating both sides over $\Omega$ and using Tonelli's theorem then shows that
\begin{align*}
\int_{\Omega}|\theta^{\prime}(x_{N})    (u^{+}(x)-u^{-}(x))|^{p}dx & \leq
cb^{-p+1}\int_{\mathbb{R}^{N-1}}|f^{+}(x^{\prime})-f^{-}(x^{\prime})|^{p}dx^{\prime} \\
&  +c\int_{\Omega}(|\partial_{N}u^{-}(x)|^{p}+|\partial_{N}u^{+}(x)|^{p})\,dx,
\end{align*}
which completes the proof.
\end{proof}

Our next result establishes the existence of functions in $\dot{W}^{1,p}(\Omega)$ that cannot be extended to $\dot{W}^{1,p}(\mathbb{R}^N)$.

\begin{proposition}\label{proposition no extension}
Let $\Omega$ be as in \eqref{omega infinite strip}. Then there exists a function $u\in\dot{W}^{1,p}(\Omega)$ such that $u$ cannot be extended to a function in $\dot{W}^{1,p}(\mathbb{R}^{N})$.  In particular, there does not exist a bounded extension operator $E: \dot{W}^{1,p}(\Omega) \to \dot{W}^{1,p}(\mathbb{R}^{N})$
\end{proposition}
\begin{proof}

Write $f \in \W_{(1)}^{1-1/p,p}(\mathbb{R}^{N-1}) \backslash \dot{W}^{1-1/p,p}(\mathbb{R}^{N-1})$ for the function constructed in Theorem \ref{theorem strict inclusion}.  According to Theorem \ref{theorem lifting strip} we can choose $u \in \dot{W}^{1,p}(\Omega)$ such that $\operatorname*{Tr} (u) = f$ on $\Gamma^+$ and on $\Gamma^-$ (i.e. we use $f^+ = f^- = f$ in the theorem).

Suppose now, by way of contradiction, that there exists $v \in \dot{W}^{1,p}(\mathbb{R}^N)$ such that $v =u$ a.e. on $\Omega$.  We may then use the trace theory for homogeneous Sobolev spaces on $\mathbb{R}^N$ (see Theorem 18.57 and Remark 18.60 in \cite{leoni2017book}) to deduce that the trace of $v$ onto $\Gamma^-$ belongs to $\dot{W}^{1-1/p,p}(\mathbb{R}^{N-1})$.  However, this trace must agree with $f$, and so we arrive at a contradiction.  Thus, there is no such $v$.

\end{proof}

\subsection{The case $m=1$, $p=1$}

In this subsection we turn our attention to the proofs of Theorems \ref{theorem trace strip p=1} and \ref{theorem lifting strip p=1}.

\begin{proof}[Proof of Theorem \ref{theorem trace strip p=1}]

As in the proof of Theorem \ref{theorem trace strip}, we set $b^{-}=0$ and  $b:=b^{+}$.  The inequality \eqref{theorem trace strip p=1 eq 1} follows as in the proof of Theorem \ref{theorem trace strip}. On the other hand, from \eqref{1} we get
\begin{align*}
\int_{\mathbb{R}^{N-1}}|u(x^{\prime}+h^{\prime},0)-u(x^{\prime}%
,0)|\,dx^{\prime}  &  \leq3\int_{0}^{|h^{\prime}|}\int_{\mathbb{R}^{N-1}%
}|\nabla u(x)|\,dx^{\prime}dx_{N}\\
&  \leq3\int_{0}^{\varepsilon}\int_{\mathbb{R}^{N-1}}|\nabla u(x)|\,dx^{\prime
}dx_{N}.
\end{align*}
Hence,
\begin{equation*}
\sup_{|h^{\prime}|\leq\varepsilon}\int_{\mathbb{R}^{N-1}}|u(x^{\prime
}+h^{\prime},0)-u(x^{\prime},0)|\,dx^{\prime}\leq3\int_{0}^{\varepsilon}%
\int_{\mathbb{R}^{N-1}}|\nabla u(x)|\,dx^{\prime}dx_{N}\rightarrow0
\end{equation*}
as $\varepsilon\rightarrow0^{+}$ by the Lebesgue dominated convergence
theorem, which can be applied since by Fubini's theorem,
\begin{equation*}
\int_{0}^{b}\left(  \int_{\mathbb{R}^{N-1}}|\nabla u(x^{\prime},x_{N}%
)|\,dx^{\prime}\right)  dx_{N}=\int_{\Omega}|\nabla u(x)|\,dx<\infty.
\end{equation*}
This proves the first two items of the theorem, and the third follows as in the proof of Theorem \ref{theorem trace strip}.
\end{proof}

We next prove the corresponding lifting result.

\begin{proof}[Proof of Theorem \ref{theorem lifting strip p=1}]

Once again we set $b^{-}=0$ and  $b:=b^{+}$ for simplicity. Let $\varepsilon_{0}>0$ be such that
\begin{equation*}
\Vert f^{\pm}\Vert_{1}:=\sup_{|h^{\prime}|\leq\varepsilon_{0}}\int%
_{\mathbb{R}^{N-1}}|f^{\pm}(x^{\prime}+h^{\prime})-f^{\pm}(x^{\prime
})|\,dx^{\prime}<\infty.
\end{equation*}
We construct a decreasing sequence $\varepsilon_{n}\rightarrow0^{+}$ such that
\begin{equation}
\sup_{|h^{\prime}|\leq\varepsilon_{n}}\int_{\mathbb{R}^{N-1}}|f^{-}(x^{\prime
}+h^{\prime})-f^{-}(x^{\prime})|\,dx^{\prime}\leq\frac{1}{2^{n}}\Vert
f^{-}\Vert_{1} \label{tr 1}%
\end{equation}
and define $f_{n}^{-}:=\varphi_{\varepsilon_{n}}\ast f^{-}$, where $\varphi\in C_{c}^{\infty}(\mathbb{R}^{N-`1})$ is a standard mollifier. 

We claim that $\partial_i f_{n}^{-} \in L^{1}(\mathbb{R}^{N-1})$.  Indeed, since $\int_{\mathbb{R}^{N-1}}\partial_i\varphi \left(  z^{\prime}\right)  \,dz^{\prime}=0$, we may compute
\begin{align*}
 \partial_i f_{n}^{-}(x^{\prime})  &  = \frac{1}{\varepsilon_{n}^{N}} \int_{\mathbb{R}^{N-1}}f^{-}(y^{\prime}) \partial_i \varphi \left(  \frac{x^{\prime}-y^{\prime} }{\varepsilon}\right)  \,dy^{\prime} \\
& = \frac{1}{\varepsilon_{n}}\int_{\mathbb{R}^{N-1}}f^{-}(x^{\prime} - \varepsilon_{n} z^{\prime}) \partial_i \varphi \left(z^{\prime}\right)  \,dz^{\prime}\\
&  =\frac{1}{\varepsilon_{n}}\int_{\mathbb{R}^{N-1}}[f^{-}(x^{\prime} - \varepsilon_{n} z^{\prime}) - f^{-}(x^{\prime})] \partial_i \varphi\left(  z^{\prime}\right)  \,dz^{\prime},
\end{align*}
and so Tonelli's theorem allows us to estimate
\begin{align*}
\int_{\mathbb{R}^{N-1}}\left\vert  \partial_i f_{n}^{-}(x^{\prime}) \right\vert \,dx^{\prime}  
&  \leq\frac{1}{\varepsilon_{n}} \int_{B^{\prime}(0,1)} \left\vert \partial_i \varphi \left(z^{\prime}\right)  \right\vert \int_{\mathbb{R}^{N-1}}|f^{-}(x^{\prime}-\varepsilon_{n} z^{\prime}) - f^{-}(x^{\prime})|\,dx^{\prime}dz^{\prime}\\
&  \leq\frac{c}{\varepsilon_{n}}\sup_{|h^{\prime}|\leq\varepsilon_{n}}\int_{\mathbb{R}^{N-1}}|f^{-}(x^{\prime}+h^{\prime})-f^{-}(x^{\prime})|\,dx^{\prime}<\infty.
\end{align*}

Next we construct a strictly decreasing sequence $\left\{  t_{n}\right\}_{n}$ in $\left(  0,1\right)$ such that $t_{n}\rightarrow0$ and
\begin{equation}\label{tr 2}
\left\vert t_{n+1}-t_{n}\right\vert \leq\frac{1}{2^{n}}\frac{\Vert f^{-}%
\Vert_{1}}{\left\Vert \nabla_{x^{\prime}}f_{n+1}^{-}\right\Vert _{L^{1}%
}+\left\Vert \nabla_{x^{\prime}}f_{n}^{-}\right\Vert _{L^{1}}+1},
\end{equation}
and we then set
\begin{equation*}
t_{0}:=\sum_{n=1}^{\infty}t_{n}-t_{n+1}. 
\end{equation*}
Using this sequence, we define the function $v^- : \mathbb{R}^{N-1} \times (0, t_0) \to \mathbb{R}$ via 
\begin{equation*}
v^{-}\left(  x\right) = \dfrac{t_{n}-x_{N}}{t_{n}-t_{n+1}}f_{n+1}^{-}\left(  x^{\prime}\right)
+\dfrac{x_{N}-t_{n+1}}{t_{n}-t_{n+1}}f_{n}^{-}\left(  x^{\prime}\right)  
\text{ if }t_{n+1}\leq x_{N}\leq t_{n}.
\end{equation*}
Then for $t_{n+1}<x_{N}<t_{n}$ we can bound
\begin{equation}\label{tr gagliardo estimates}
\begin{split}
\left\vert \partial_{i}v^{-}\left(  x\right)  \right\vert  &  \leq\left\vert
\partial_{i}f_{n+1}^{-}\left(  x^{\prime}\right)  \right\vert +\left\vert
\partial_{i}f_{n}^{-}\left(  x^{\prime}\right)  \right\vert \text{ for all
}i=1,\ldots,N-1,\\
\left\vert \partial_{N}v^{-}\left(  x\right)  \right\vert  &  \leq\left\vert
\dfrac{f_{n+1}^{-}\left(  x^{\prime}\right)  -f_{n}^{-}\left(  x^{\prime
}\right)  }{t_{n}-t_{n+1}}\right\vert.  
\end{split}
\end{equation}
Hence, for $i=1,\ldots,N-1$, we may use \eqref{tr 2} and \eqref{tr gagliardo estimates} to estimate
\begin{align*}
\int_{\mathbb{R}^{N-1}\times(0,t_{0})}\left\vert \partial_{i} v^{-}\right\vert
\,dx  &  =\sum_{n=1}^{\infty}\int_{t_{n+1}}^{t_{n}}\int_{\mathbb{R}^{N-1}%
}\left\vert \partial_{i} v^{-}\right\vert \,dx^{\prime}\,dx_{N}\\
&  \leq\sum_{n=1}^{\infty}\left\vert t_{n+1}-t_{n}\right\vert \int%
_{\mathbb{R}^{N-1}}\left(  \left\vert \partial_{i}f_{n+1}^{-}\left(
x^{\prime}\right)  \right\vert +\left\vert \partial_{i}f_{n}^{-}\left(
x^{\prime}\right)  \right\vert \right)  \,dx^{\prime}\leq\Vert f^{-}\Vert_{1}\sum_{n=1}^{\infty}\frac{1}{2^{n}}.
\end{align*}
On the other hand, from \eqref{tr gagliardo estimates} we know that
\begin{align*}
\int_{\mathbb{R}^{N-1}\times(0,t_{0})}\left\vert \partial_{N} v^{-} \right\vert
\,dx  &  =\sum_{n=1}^{\infty}\int_{t_{n+1}}^{t_{n}}\int_{\mathbb{R}^{N-1}}\left\vert \partial_{N}v^{-}\right\vert \,dx^{\prime}\,dx_{N}\\
&  \leq\sum_{n=1}^{\infty}\int_{\mathbb{R}^{N-1}}\left\vert f_{n+1}^{-}\left(
x^{\prime}\right)  -f_{n}^{-}\left(  x^{\prime}\right)  \right\vert
\,dx^{\prime}.
\end{align*}
To estimate this final term, we use the identity
\begin{equation*}
f_{n+1}^{-}\left(  x^{\prime}\right)  -f_{n}^{-}\left(  x^{\prime}\right)
=\int_{\mathbb{R}^{N-1}}[f^{-}(x^{\prime}-\varepsilon_{n+1}z^{\prime}) - f^{-}(x^{\prime}-\varepsilon_{n}z^{\prime})]\varphi\left(  z^{\prime}\right)  \,dz^{\prime}
\end{equation*}
in conjunction with Tonelli's theorem and \eqref{tr 1} to see that
\begin{align*}
\int_{\mathbb{R}^{N-1}}\left\vert f_{n+1}^{-}\left(  x^{\prime}\right)
-f_{n}^{-}\left(  x^{\prime}\right)  \right\vert \,dx^{\prime}  
&  \leq \int_{B^{\prime}(0,1)}\varphi\left(  z^{\prime}\right)  \int_{\mathbb{R}^{N-1}}|f^{-}(x^{\prime}-\varepsilon_{n+1}z^{\prime})-f^{-}(x^{\prime} - \varepsilon_{n}z^{\prime})|\,dx^{\prime}dz^{\prime} \\
&  \leq\int_{B^{\prime}(0,1)}\varphi\left(  z^{\prime}\right)  \int_{\mathbb{R}^{N-1}}|f^{-}(x^{\prime}-\varepsilon_{n+1}z^{\prime})-f(x^{\prime})|\,dx^{\prime}dz^{\prime} \\
&  \quad+\int_{B^{\prime}(0,1)}\varphi\left(  z^{\prime}\right) \int_{\mathbb{R}^{N-1}}|f^{-}(x^{\prime}-\varepsilon_{n}z^{\prime})-f(x^{\prime})|\,dx^{\prime}dz^{\prime}\leq \frac{2}{2^{n}}\Vert f^{-}
\Vert_{1}.
\end{align*}
Combining these, we deduce that
\begin{equation*}
 \int_{\mathbb{R}^{N-1}\times(0,t_{0})}\left\vert \partial_{N} v^{-} \right\vert
\,dx \leq  2 \Vert f^{-} \Vert_{1} \sum_{n=1}^\infty \frac{1}{2^n}.
\end{equation*}
Hence $v^- \in \dot{W}^{1,1}(\mathbb{R}^{N-1} \times (0,t_0))$.

The function $v^-$ is not defined in all of $\Omega$, but by scaling we can arrive at a function $u^- : \Omega \to \mathbb{R}$.  Indeed, we set $u^{-}(x) = v^{-}(x^{\prime},x_{N}t_{0}/b)$ for $x\in\Omega$.  Clearly  $u^{-}\in\dot
{W}^{1,1}\left(  \Omega\right)$ and the trace of $u^-$ on $\Gamma^-$ is $f^-$.  Similarly, we can construct a function $u^{+}\in\dot{W}^{1,1}\left(\Omega\right)$ the trace of which is $f^{+}$ on $\Gamma^+$.  With $u^-$ and $u^+$ in hand, we can now proceed as in the second part of the proof of Theorem \ref{theorem lifting strip} to construct the desired function $u$.
\end{proof}

\subsection{The case $m=2$, $p>1$}

In this subsection we study the traces of functions $\dot{W}^{2,p}\left(\Omega\right)$ when $p >1$.

\begin{proof}[Proof of Theorem \ref{theorem trace strip m=2}]

As in the proof of Theorem \ref{theorem trace strip}, we take $b^{-}=0$, set $b:=b^{+}$.  Applying Theorem \ref{theorem AC} and Proposition \ref{proposition by parts} (see in particular \eqref{formula by parts m=2}) to $u(x^{\prime},\cdot)$ for $\mathcal{L}^{N-1}$ a.e. $x' \in \mathbb{R}^{N-1}$,  we find that
\begin{equation*}
u(x^{\prime},b)-u(x^{\prime},0)-(\partial_{N}u(x^{\prime},b)+\partial_{N}u(x^{\prime},0))\frac{b}{2}
=\int_{0}^{b}\partial_{N}^{2}u(x^{\prime},x_{N})\left(  b-\frac{x_{N}}{2}\right)  \,dx_{N}.
\end{equation*}
Hence, by H\"{o}lder's inequality,
\begin{equation*}
\left\vert u(x^{\prime},b)-u(x^{\prime},0)-(\partial_{N}u(x^{\prime},b)+\partial_{N}u(x^{\prime},0))\frac{b}{2}\right\vert ^{p}\leq b^{2p-1} \int_{0}^{b}|\partial_{N}^{2}u(x^{\prime},x_{N})|^{p}dx_{N}.
\end{equation*}
Integrating both sides with respect to $x^{\prime} \in \mathbb{R}^{N-1}$ and using Tonelli's theorem then shows that
\begin{equation*}
\int_{\mathbb{R}^{N-1}}\left\vert u(x^{\prime},b)-u(x^{\prime},0)-(\partial_{N} u(x^{\prime},b)+\partial_{N}u(x^{\prime},0))\frac{b}{2}\right\vert^{p}dx^{\prime}\leq b^{2p-1}\int_{\Omega}|\partial_{N}^{2}u(x)|^{p}dx,
\end{equation*}
which proves \eqref{theorem trace strip m=2 eq 2}.

On the other hand, for $i=1,\ldots,N$,
\begin{equation*}
\partial_{i}u(x^{\prime},b)-\partial_{i}u(x^{\prime},0)=\int_{0}^{b} \partial_{i,N}^{2}u(x^{\prime},x_{N})\,dx_{N},
\end{equation*}
and so we may apply H\"{o}lder's inequality to see that
\begin{equation*}
|\partial_{i}u(x^{\prime},b)-\partial_{i}u(x^{\prime},0)|^{p}
\leq b^{p-1} \int_{0}^{b}|\partial_{i,N}^{2}u(x^{\prime},x_{N})|^{p}dx_{N},
\end{equation*}
which upon integration in $x^{\prime} \in \mathbb{R}^{N-1}$ yields the bound
\begin{equation*}
\int_{\mathbb{R}^{N-1}}|\partial_{i}u(x^{\prime},b)-\partial_{i} u(x^{\prime},0)|^{p}dx\leq b^{p-1}\int_{\Omega}|\partial_{i,N}^{2}u(x)|^{p}dx.
\end{equation*}
This is \eqref{theorem trace strip m=2 eq 1}.

Finally, since $\partial_{i}u\in\dot{W}^{1,p}\left(  \Omega\right)$, the inequality \eqref{theorem trace strip m=2 eq 3} follows by applying Theorem \ref{theorem trace strip} to $\partial_{i}u$.
\end{proof}

Next we prove the corresponding lifting result.

\begin{proof}[Proof of Theorem \ref{theorem lifting strip m=2}]
For simplicity we again take $b^{-}=0$ and set $b:=b^{+}$.  Without loss of generality we may assume that $0<a<b$.  Let $\varphi\in C_{c}^{\infty}(\mathbb{R}^{N-1})$ be a nonnegative even function such that $\int_{\mathbb{R}^{N-1}}\varphi(y^{\prime})\,dy^{\prime}=1$ and $\operatorname*{supp}\varphi\subseteq B^{\prime}(0,ab^{-1})$.  Define $u^- : \Omega \to \mathbb{R}$ via 
\begin{align*}
u^{-}(x)  &  =\int_{\mathbb{R}^{N-1}}[f_{0}^{-}(y^{\prime})+f_{1}%
^{-}(y^{\prime})x_{N}]\varphi_{x_{N}}\left(  x^{\prime}-y^{\prime}\right)
\,dy^{\prime}\\
&  =\int_{\mathbb{R}^{N-1}}[f_{0}^{-}(x^{\prime}-x_{N}z^{\prime})+f_{1}%
^{-}(x^{\prime}-x_{N}z^{\prime})x_{N}]\varphi\left(  z^{\prime}\right)
\,dz^{\prime},
\end{align*}
where the second equality follows by the change of variables $z^{\prime} =\frac{x^{\prime}-y^{\prime}}{x_{N}}$.  Standard properties of mollifiers guarantee that for $i =0,1$ we have that $\varphi_{x_{N}}\ast f_{i}^{-}\rightarrow f_{i}^{-}$ in $L_{\operatorname*{loc}}^{p}(\mathbb{R}^{N-1})$ as $x_{N}\rightarrow0^{+}$, so we deduce that $\operatorname*{Tr}(u^{-})=f_{0}^{-}$ on $\Gamma^{-}$.

For $i=1,\ldots,N-1$ we compute 
\begin{equation}\label{first derivative}
\begin{split}
\partial_{i}u^{-}(x)  &  =\int_{\mathbb{R}^{N-1}}\partial_{i}f_{0}^{-}(x^{\prime}-x_{N}z^{\prime})\varphi\left(  z^{\prime}\right)\,dz^{\prime} 
+\int_{\mathbb{R}^{N-1}}f_{1}^{-}(y^{\prime})x_{N}\frac{\partial }{\partial x_{i}}(\varphi_{x_{N}}\left(  x^{\prime}-y^{\prime}\right))\,dy^{\prime} \\
&  =\int_{\mathbb{R}^{N-1}}\partial_{i}f_{0}^{-}(y^{\prime})\varphi_{x_{N}}\left(  x^{\prime}-y^{\prime}\right)  \,dy^{\prime} 
+\int_{\mathbb{R}^{N-1}}f_{1}^{-}(y^{\prime})x_{N}\frac{\partial}{\partial x_{i}}(\varphi_{x_{N}}\left(  x^{\prime}-y^{\prime}\right))\,dy^{\prime},
\end{split}
\end{equation}
while for $i=N$ 
\begin{equation}\label{first derivative N}
\begin{split}
\partial_{N}u^{-}(x)  &  =\int_{\mathbb{R}^{N-1}}-(\nabla_{\shortparallel }f_{0}^{-}(x^{\prime}-x_{N}z^{\prime})\cdot z^{\prime})\varphi\left( z^{\prime}\right)  \,dz^{\prime} 
+\int_{\mathbb{R}^{N-1}}f_{1}^{-}(y^{\prime})\frac{\partial}{\partial x_{N}}(x_{N}\varphi_{x_{N}}\left(  x^{\prime}-y^{\prime}\right) )\,dy^{\prime} \\
&  =\int_{\mathbb{R}^{N-1}}-\nabla_{\shortparallel}f_{0}^{-}(y^{\prime}) \cdot\frac{x^{\prime}-y^{\prime}}{x_{N}}\varphi_{x_{N}}\left(  x^{\prime} - y^{\prime}\right)  \,dy^{\prime} \\
& \quad +\int_{\mathbb{R}^{N-1}}f_{1}^{-}(y^{\prime})\left[  \varphi_{x_{N} }\left(  x^{\prime}-y^{\prime}\right)  +x_{N}\frac{\partial}{\partial x_{N} }(\varphi_{x_{N}}\left(  x^{\prime}-y^{\prime}\right)  )\right] \,\,dy^{\prime} 
\end{split}
\end{equation}
In turn, for $j=1,\ldots,N-1$ we may compute
\begin{equation}\label{second derivative mixed}
\partial_{i,j}^{2}u^{-}(x)    = \int_{\mathbb{R}^{N-1}}\partial_{i}f_{0}^{-}(y^{\prime})\frac{\partial}{\partial x_{j}}(\varphi_{x_{N}} \left(x^{\prime}-y^{\prime}\right)  )\,dy^{\prime} 
+\int_{\mathbb{R}^{N-1}}f_{1}^{-}(y^{\prime})x_{N}\frac{\partial^{2}}{\partial x_{j}\partial x_{i}}(\varphi_{x_{N}}\left(  x^{\prime}-y^{\prime }\right)  )\,dy^{\prime}, 
\end{equation}
while for $j=N$
\begin{equation}\label{second derivative semi-mixed}
\begin{split}
\partial_{i,N}^{2}  &  u^{-}(x)=\int_{\mathbb{R}^{N-1}}\partial_{i}f_{0}
^{-}(y^{\prime})\frac{\partial}{\partial x_{N}}(\varphi_{x_{N}}\left( x^{\prime}-y^{\prime}\right)  )\,dy^{\prime} \\
& +\int_{\mathbb{R}^{N-1}}f_{1}^{-}(y^{\prime})\left[ \frac{\partial}{\partial x_{i}}(\varphi_{x_{N}}\left( x^{\prime} - y^{\prime}\right) ) + x_{N} \frac{\partial^{2}}{\partial x_{i}\partial x_{N}}(\varphi_{x_{N} } \left(  x^{\prime} - y^{\prime}\right)  )\right]  \,dy^{\prime}. 
\end{split} 
\end{equation}
Finally, when $i=j=N$ we have that
\begin{equation}\label{second derivative pure}
\begin{split}
\partial_{N}^{2}  &  u^{-}(x)=\int_{\mathbb{R}^{N-1}}-\nabla_{\shortparallel} f_{0}^{-}(y^{\prime}) \cdot\frac{\partial}{\partial x_{N}}\left( \frac{x^{\prime}-y^{\prime}}{x_{N}}\varphi_{x_{N}}\left(  x^{\prime} - y^{\prime} \right)  \right)  \,dy^{\prime} \\
& +\int_{\mathbb{R}^{N-1}}f_{1}^{-}(y^{\prime})\left[  2\frac{\partial}{\partial x_{N}}(\varphi_{x_{N}} \left(  x^{\prime}-y^{\prime}\right) )+x_{N}\frac{\partial^{2}}{\partial x_{N}^{2}}(x_{N}\varphi_{x_{N}} \left(x^{\prime}-y^{\prime}\right)  )\right]  \,\,dy^{\prime}. 
\end{split}
\end{equation}

Owing to Proposition \ref{proposition derivatives mollifiers}, we are in a position to apply Proposition \ref{proposition structure} to conclude that for $i=1,\ldots,N-1$ and $j=1,\ldots,N$ we have the bounds
\begin{align*}
\int_{\Omega}|\partial_{i,j}^{2}u^{-}(x)|^{p}dx  &  \leq c\int_{\mathbb{R}^{N-1}}\int_{B^{\prime}(x^{\prime},a)}\frac{|\partial_{i}f_{0}^{-}(y^{\prime})-\partial_{i}f_{0}^{-}(x^{\prime})|^{p}}{|x^{\prime}-y^{\prime}|^{N+p-2}}\,dy^{\prime}dx^{\prime}\\
&  \quad+c\int_{\mathbb{R}^{N-1}}\int_{B^{\prime}(x^{\prime},a)}\frac{|f_{1}^{-}(y^{\prime})-f_{1}^{-}(x^{\prime})|^{p}}{|x^{\prime}-y^{\prime}|^{N+p-2}}\,dy^{\prime}dx^{\prime}
\end{align*}
while for two $N$ derivatives we have that 
\begin{align*}
\int_{\Omega}|\partial_{N}^{2}u^{-}(x)|^{p}dx  
&  \leq c\int_{\mathbb{R}^{N-1}}\int_{B^{\prime}(x^{\prime},a)}\frac{|\nabla_{\shortparallel}f_{0}^{-}(y^{\prime})-\nabla_{\shortparallel}f_{0}^{-}(x^{\prime})|^{p}}{|x^{\prime}-y^{\prime}|^{N+p-2}}\,dy^{\prime}dx^{\prime}\\
&  \quad+c\int_{\mathbb{R}^{N-1}}\int_{B^{\prime}(x^{\prime},a)}\frac{|f_{1}^{-}(y^{\prime})-f_{1}^{-}(x^{\prime}) |^{p}}{|x^{\prime}-y^{\prime}|^{N+p-2}}\,dy^{\prime}dx^{\prime},
\end{align*}
which shows that $u^{-}\in\dot{W}^{2,p}(\Omega)$.

Next we write 
\begin{align*}
\partial_{N}u^{-}(x)  &  =\int_{\mathbb{R}^{N-1}}-\nabla_{\shortparallel} f_{0}^{-}(y^{\prime}) \cdot \frac{x^{\prime}-y^{\prime}}{x_{N}}\varphi_{x_{N} }\left(  x^{\prime}-y^{\prime}\right)  \,dy^{\prime} \\
&  \quad+\int_{\mathbb{R}^{N-1}}f_{1}^{-}(y^{\prime})\left[  \varphi_{x_{N} }\left(  x^{\prime}-y^{\prime}\right)  +\frac{1}{x_{N}^{N-1}}\psi_{N} \left( \frac{x^{\prime}-y^{\prime}}{x_{N}}\right)  \right]  \,dy^{\prime}, 
\end{align*}
where $\psi_{N}$ is given in \eqref{psi i}.  Standard properties of mollifiers imply that as $x_{N}\rightarrow0^{+}$ we have that
\begin{equation*}
\partial_{N}u^{-}(x)\rightarrow-\nabla_{\shortparallel}f_{0}^{-}(x^{\prime
})\cdot\int_{\mathbb{R}^{N-1}}z^{\prime}\varphi\left(  z^{\prime}\right)
\,dz^{\prime}+f_{1}^{-}(x^{\prime})\left[  1+\int_{\mathbb{R}^{N-1}}\psi
_{N}(z^{\prime})\,dz^{\prime}\right]  =f_{1}^{-}(x^{\prime}),
\end{equation*}
where in the second equality we used the fact that $\varphi$ is even to see that $\int_{\mathbb{R}^{N-1}}z^{\prime}\varphi \left(z^{\prime}\right)  \,dz^{\prime}=0$ and Proposition \ref{proposition derivatives mollifiers} to see that $\int_{\mathbb{R}^{N-1}}\psi_{N}(z^{\prime})\,dz^{\prime}=0$.  Hence, $\operatorname*{Tr}(\partial_{N}u^{-})=f_{1}^{-}$\ on $\Gamma^{-}$.

Similarly, if we define $u^+ : \Omega \to \mathbb{R}$ via 
\begin{equation*}
u^{+}(x) =\int_{\mathbb{R}^{N-1}}[f_{0}^{+}(y^{\prime})+f_{1}^{+}(y^{\prime}) (b-x_{N})] \varphi_{b-x_{N}}\left(  x^{\prime}-y^{\prime}\right) \,dy^{\prime},
\end{equation*}
then we find that $u^{+}\in\dot{W}^{2,p}(\Omega)$ with estimates of the same form as above (with $f_i^+$ replacing $f_i^-$). Moreover,  $\operatorname*{Tr}(u^{+})=f_{0}^{+}$ and $\operatorname*{Tr}(\partial_{N}u^{+})=f_{1}^{+}$ on
$\Gamma^{+}$.

Let $\theta\in C^{\infty}([0,b])$ be such that $\theta=1$ in a neighborhood of $0$ and $\theta=0$ in a neighborhood of $b$ and define $u: \Omega \to \mathbb{R}$ via
\begin{equation*}
u(x) = \theta(x_{N})u^{-}(x)+(1-\theta(x_{N}))u^{+}(x).
\end{equation*}
Then for $i,j=1,\ldots,N-1$, 
\begin{align*}
\partial_{i,j}^{2}u(x)  &  =\theta(x_{N})\partial_{i,j}^{2}u^{-}%
(x)+(1-\theta(x_{N}))\partial_{i,j}^{2}u^{+}(x),\\
\partial_{i,N}^{2}u(x)  &  =\theta^{\prime}(x_{N})(\partial_{i}u^{-}%
(x)-\partial_{i}u^{+}(x))+\theta(x_{N})\partial_{i,N}^{2}u^{-}(x)+(1-\theta
(x_{N}))\partial_{i,N}^{2}u^{+}(x),\\
\partial_{N}^{2}u(x)  &  =\theta^{\prime\prime}(x_{N})(u^{-}(x)-u^{+}%
(x))+2\theta^{\prime}(x_{N})(\partial_{N}u^{-}(x)-\partial_{N}u^{+}(x))\\
&  \quad+\theta(x_{N})\partial_{N}^{2}u^{-}(x)+(1-\theta(x_{N}))\partial
_{N}^{2}u^{+}(x).
\end{align*}
From these calculations and the above bounds for $u^{\pm}\in\dot{W}^{2,p}(\Omega)$, we see that in order to prove the desired estimate for $u$ it suffices to estimate $\theta^{\prime} (\nabla u^{-}-\nabla u^{+})$ and $\theta^{\prime\prime}(u^{-}-u^{+})$. 

By Taylor's formula and \eqref{formula taylor inverted m=2} applied to $u^{-}(x^{\prime},\cdot)$ and $u^{+}(x^{\prime},\cdot)$, respectively, we have that
\begin{align*}
u^{-}(x)  &  =f_{0}^{-}(x^{\prime})+f_{1}^{-}(x^{\prime})x_{N}+\frac{1}{2} \int_{0}^{x_{N}}\partial_{N}^{2}u^{-}(x^{\prime},s)(x_{N}-s)\,ds,\\
u^{+}(x)  &  =f_{0}^{+}(x^{\prime})-f_{1}^{+}(x^{\prime})(b-x_{N})+\frac{1}{2} \int_{x_{N}}^{b}\partial_{N}^{2}u^{+}(x^{\prime},s)(x_{N}-s)\,ds,
\end{align*}
where we used the fact that $\operatorname*{Tr}(u^{\pm})=f_{0}^{\pm}$ and $\operatorname*{Tr}(\partial_{N}u^{\pm})=f_{1}^{\pm}$\ on $\Gamma^{\pm}$.  Hence,
\begin{align*}
u^{+}(x)-u^{-}(x)  &  =f_{0}^{+}(x^{\prime})-f_{0}^{-}(x^{\prime}) - (f_{1}^{+}(x^{\prime}) + f_{1}^{-}(x^{\prime}))\frac{b}{2} - (f_{1}^{+}(x^{\prime})-f_{1}^{-}(x^{\prime}))(\frac{b}{2}-x_{N})\\
&  \quad + \frac{1}{2}\int_{x_{N}}^{b}\partial_{N}^{2}u^{+}(x^{\prime},s)(x_{N}-s)\,ds-\frac{1}{2}\int_{0}^{x_{N}}\partial_{N}^{2}u^{-}(x^{\prime},s)(x_{N}-s)\,ds.
\end{align*}
From H\"{o}lder's inequality and the fact that $|\theta^{\prime\prime}(x_{N})|\leq cb^{-2}$ we see that
\begin{align*}
|\theta^{\prime\prime}(x_{N})(u^{+}(x)-u^{-}(x))|^{p}  &  \leq cb^{-2p} \left\vert f_{0}^{+}(x^{\prime})-f_{0}^{-}(x^{\prime})-(f_{1}^{+}(x^{\prime}) + f_{1}^{-}(x^{\prime})) \frac{b}{2}\right\vert ^{p}\\
&  \quad+cb^{-p}|f_{1}^{+}(x^{\prime})-f_{1}^{-}(x^{\prime})|^{p}\\
&  \quad+cb^{-1}\int_{x_{N}}^{b}|\partial_{N}^{2}u^{+}(x^{\prime},s)|^{p}ds+cb^{-1}\int_{0}^{x_{N}}|\partial_{N}^{2}u^{-}(x^{\prime},s)|^{p}ds.
\end{align*}
Integrating over $x \in \Omega$ and using Tonelli's theorem then shows that
\begin{align*}
\int_{\Omega}|\theta^{\prime\prime}(x_{N})(u^{+}(x)-u^{-}(x))|^{p}dx  
&  \leq cb^{-2p+1}\int_{\mathbb{R}^{N-1}}\left\vert f_{0}^{+}(x^{\prime})-f_{0}^{-}(x^{\prime})-(f_{1}^{+}(x^{\prime})+f_{1}^{-}(x^{\prime}))\frac{b}{2}\right\vert ^{p}dx^{\prime}\\
&  \quad+cb^{-p+1}\int_{\mathbb{R}^{N-1}}|f_{1}^{+}(x^{\prime})-f_{1}^{-}(x^{\prime})|^{p}dx^{\prime}\\
&  \quad+c\int_{\Omega}|\partial_{N}^{2}u^{+}(x)|^{p}dx+c\int_{\Omega}|\partial_{N}^{2}u^{-}(x)|^{p}dx.
\end{align*}
This is the desired estimate for $\theta^{\prime\prime}(u^{-}-u^{+})$.

Finally, we derive the estimate for $\theta^{\prime} (\nabla u^{-}-\nabla u^{+})$.  By the fundamental theorem of calculus for every $i=1,\ldots,N-1$ we have that
\begin{align*}
\partial_{i}u^{-}(x)  &  =\partial_{i}f_{0}^{-}(x^{\prime})+\int_{0}^{x_{N} }\partial_{N,i}^{2}u^{-}(x^{\prime},s)\,ds, \\
\partial_{i}u^{+}(x)  &  =\partial_{i}f_{0}^{+}(x^{\prime})-\int_{x_{N}}%
^{b}\partial_{N,i}^{2}u^{+}(x^{\prime},s)\,ds.
\end{align*}
Reasoning as above then shows that
\begin{align*}
\int_{\Omega}|\theta^{\prime}(x_{N})    (\partial_{i}u^{+}(x)-\partial_{i} u^{-} (x))|^{p}dx 
& \leq cb^{-p+1} \int_{\mathbb{R}^{N-1}}|\partial_{i} f_{0}^{+}(x^{\prime})-\partial_{i}f_{0}^{-}(x^{\prime})|^{p}dx^{\prime} \\
& \quad + c\int_{\Omega}(|\partial_{N,i}^{2}u^{-}(x)|^{p}+|\partial_{N,i}^{2} u^{+}(x)|^{p})\,dx.
\end{align*}
The term $2\theta^{\prime}(\partial_{N}u^{-}-\partial_{N}u^{+})$ can be estimated in a similar way, which completes the proof.
\end{proof}

\subsection{The case $m\geq2$, $p>1$}\label{subsection m>2 strip}

Finally, in this subsection  we treat the case in which $m\geq2$ and $p>1$. Given a function $u\in\dot{W}^{m,p}\left(  \Omega\right)  $, we have that $\nabla^{k}u\in\dot{W}^{m-k,p}\left(  \Omega\right)  $ for every $k=1,\ldots,m-1$ and thus by Theorem \ref{theorem trace strip} there exists $\operatorname*{Tr}(\nabla^{k}u)$.  By a density argument, it can be shown that $\operatorname*{Tr}(\nabla^{k}u)(\cdot,b^{\pm})\in W_{\operatorname*{loc}}^{m-k-1,p}(\mathbb{R}^{N-1})$ for every $k=1,\ldots,m-2$ and that for every $j=1,\ldots,m-k-1$, $(\nabla_{\shortparallel})^{j}\operatorname*{Tr}(\nabla^{k}u)(\cdot,b^{\pm})=\operatorname*{Tr}((\nabla_{\shortparallel})^{j}(\nabla^{k}u))(\cdot,b^{\pm})$, where we recall that $(\nabla_{\shortparallel})^{j}$ denotes the vector of all multi-indices $\alpha =(\alpha^{\prime},0)\in\mathbb{N}_{0}^{N-1}\times\mathbb{N}_{0}$ with $|\alpha|=j$. Thus, to characterize the trace space of $\dot{W}^{m,p}\left( \Omega\right)$, it suffices to study $\operatorname*{Tr}(u)$ and the trace of the normal derivatives $\operatorname*{Tr}(\partial_{N}^{k}u)$ for $k=1,\ldots,m-1$.

In what follows we use the notation established in \eqref{gradient zero} and \eqref{gradient prime zero}.

\begin{theorem}\label{theorem trace strip m>2}
Let $\Omega$ be as in \eqref{omega infinite strip},  $m\in\mathbb{N}$ with $m\geq2$, and $1<p<\infty$.  There exists a constant $c=c(m,N,p)>0$ such that for every $u\in\dot{W}^{m,p}\left(  \Omega\right)$ the following estimates hold:  for every $i,l\in\mathbb{N}_{0}$ with $0\leq i+l\leq m-1$,
\begin{equation}\label{theorem trace strip m>2 eq 1}
\begin{split}
\int_{\mathbb{R}^{N-1}}\left\vert \sum_{k=0}^{m-l-i-1}\frac{(-1)^{k}}{k!}(\nabla_{\shortparallel}^{i}\operatorname*{Tr}(\partial_{N}^{k+l} u)(x^{\prime},b^{+})+(-1)^{k+1}\nabla_{\shortparallel}^{i}\operatorname*{Tr}(\partial_{N}^{k+l}u)(x^{\prime},b^{-}))\left(  \frac{b^{+}-b^{-}}{2}\right)^{k}\right\vert ^{p}dx^{\prime}\\
\leq (b^{+}-b^{-})^{(m-i-l)p-1}\int_{\Omega}|\nabla^{m}u(x)|^{p}dx,
\end{split}
\end{equation}
and 
\begin{equation}\label{theorem trace strip m>2 eq 2}
\int_{\mathbb{R}^{N-1}}\int_{B^{\prime}(0,b^{+}-b^{-})}\frac{|\operatorname*{Tr}(\nabla^{m-1}u)(x^{\prime} + h^{\prime},b^{\pm})-\operatorname*{Tr}(\nabla^{m-1}u)(x^{\prime},b^{\pm})|^{p}}{|h^{\prime} |^{p+N-2}}\,dh^{\prime}dx^{\prime}
\leq c\int_{\Omega}|\nabla^{m}u(x)|^{p}dx. 
\end{equation}
\end{theorem}

\begin{proof}
As in the proof of Theorem \ref{theorem trace strip} we take $b^{-}=0$, set $b:=b^{+}$.  Let $l, i \in \mathbb{N}_{0}$ with $0\leq l+i\leq m-1$. By Theorem \ref{theorem AC} and Proposition \ref{proposition by parts} (in particular \eqref{formula by parts m=2}) applied to $\nabla_{\shortparallel}^{i} \partial_{N}^{l} u(x^{\prime},\cdot)$ for $\mathcal{L}^{N-1}$ a.e. $x' \in \mathbb{R}^{N-1}$, it follows that
\begin{align*}
\sum_{k=0}^{m-l-i-1}  &  \frac{(-1)^{k}}{k!}(\nabla_{\shortparallel}^{i}\partial_{N}^{k+l}u(x^{\prime},b)+(-1)^{k+1}\nabla_{\shortparallel}^{i}\partial_{N}^{k+l}u(x^{\prime},0))\left(  \frac{b}{2}\right)^{k}   \\
&  =\frac{1}{(m-l-i-1)!}\int_{0}^{b}\nabla_{\shortparallel}^{i}\partial_{N}^{m-i}u(x^{\prime},t)\left(  b-\frac{t}{2}\right)^{m-l-i-1}dt,
\end{align*}
and so H\"{o}lder's inequality implies that
\begin{align*}
&  \left\vert \sum_{k=0}^{m-l-i-1}\frac{(-1)^{k}}{k!}(\nabla_{\shortparallel}^{i}\partial_{N}^{k+l}u(x^{\prime},b)+(-1)^{k+1}\nabla_{\shortparallel}^{i}\partial_{N}^{k+l}u(x^{\prime},0))\left(  \frac{b}{2}\right)  ^{k}\right\vert ^{p} \\
&  \quad\quad\leq b^{(m-l-i)p-1}\int_{0}^{b}|\nabla^{m}u(x^{\prime},x_{N})|^{p}dx_{N}.
\end{align*}
Integrating both sides with respect to $x^{\prime}\in \mathbb{R}^{N-1}$ and using Tonelli's theorem then shows that
\begin{align*}
\int_{\mathbb{R}^{N-1}}  &  \left\vert \sum_{k=0}^{m-l-i-1}\frac{(-1)^{k}}{k!}(\nabla_{\shortparallel}^{i}\partial_{N}^{k+l}u(x^{\prime},b)+(-1)^{k+1}\nabla_{\shortparallel}^{i} \partial_{N}^{k+l}u(x^{\prime},0))\left(  \frac{b}{2}\right)  ^{k}\right\vert
^{p}dx^{\prime}\\
&  \leq b^{(m-l-i)p-1}\int_{\Omega}|\nabla^{m}u(x)|^{p}dx,
\end{align*}
which is \eqref{theorem trace strip m>2 eq 1}.  On the other hand, since $\nabla^{m}u\in\dot{W}^{1,p}\left(  \Omega;\mathbb{R}^{M_{m}}\right) $, where $M_{m}$ is the number of multi-indices of length $m$, inequality \eqref{theorem trace strip m>2 eq 2} follows by applying Theorem \ref{theorem trace strip} to each component of $\nabla^{m}u$.
\end{proof}

Next we prove the corresponding lifting result. The proof is significantly more involved than the one for inhomogeneous Sobolev spaces (see, for example, \cite{burenkov1998book}, \cite{mazya-mitrea-shaposhnikova2010}, or \cite{necas2012book}) because we only have control of the screened homogeneous fractional seminorm of the traces of the derivatives of order $m-1$ and we have no control of the seminorm of the traces of the derivatives of order less than $m-1$.  Thus, we are forced to employ several integrations by parts, which makes the proof quite technical.

Before stating the result, we establish some notation.  Given $f_{k}^{\pm}\in W_{\operatorname*{loc}}^{m-1-k,p}(\mathbb{R}^{N-1})$ for  $k=0,\ldots,m-2$, and $f_{m-1}^{\pm}\in L_{\operatorname*{loc}}^{,p}(\mathbb{R}^{N-1})$ we define
\begin{equation*}
Q_{i,m,n}(x^{\prime}):=\sum_{k=0}^{m-i-n-1}\frac{(-1)^{k}}{k!}(\nabla
_{\shortparallel}^{i}f_{k+n}^{\pm}(x^{\prime})+(-1)^{k+1}\nabla
_{\shortparallel}^{i}f_{k+n}^{\pm}(x^{\prime}))\left(  \frac{b^{+}-b^{-}}%
{2}\right)  ^{k}.
\end{equation*}

\begin{theorem}\label{theorem lifting strip m>2}
Let $a>0$, $m\in\mathbb{N}$ with $m\geq2$, and $1<p<\infty$.  Suppose that  $f_{k}^{\pm}\in W_{\operatorname*{loc}}^{m-1-k,p}(\mathbb{R}^{N-1})$ for $k=0,\ldots,m-2$ and $f_{m-1}^{\pm}\in L_{\operatorname*{loc}}^{,p}(\mathbb{R}^{N-1})$ satisfy
\begin{equation}\label{compatibility conditions} 
\int_{\mathbb{R}^{N-1}}\left\vert Q_{i,m,n}(x^{\prime})\right\vert^{p}dx^{\prime}<\infty
\end{equation}
for all $n,i\in\mathbb{N}_{0}$ with $0\leq n+i\leq m-1$ and
\begin{equation*}
\int_{\mathbb{R}^{N-1}}\int_{B^{\prime}(0,a)}\frac{|\nabla_{\shortparallel}^{m-k-1}f_{k}^{\pm}(x^{\prime}+h^{\prime})-\nabla_{\shortparallel}^{m-k-1}f_{k}^{\pm}(x^{\prime})|^{p}}{|h^{\prime}|^{p+N-2}}\,dh^{\prime
}dx^{\prime}<\infty
\end{equation*}
for all $k=0,\ldots,m-1$.   Then there exists $u\in\dot{W}^{m,p}(\Omega)$ such that $\operatorname*{Tr}(u)=f_{0}^{\pm}$ and $\operatorname*{Tr}(\partial_{N}^{k}u)=f_{k}^{\pm}$ on $\Gamma^{\pm}$ for $k=1,\ldots,m-1$, and
\begin{align*}
\int_{\Omega}  &  |\nabla^{m}u(x)|^{p}dx\leq c\sum_{n=0}^{m-1} \sum_{i=0}^{m-1-n}(b^{+}-b^{-})^{(n+i-m)p+1}\int_{\mathbb{R}^{N-1}}\left\vert Q_{i,m,n}(x^{\prime})\right\vert^{p}dx^{\prime}\\
&  +c\sum_{k=0}^{m-1}\int_{\mathbb{R}^{N-1}}\int_{B^{\prime}(0,a)} \frac{|\nabla_{\shortparallel}^{m-k-1}f_{k}^{-}(x^{\prime}+h^{\prime}) - \nabla_{\shortparallel}^{m-k-1}f_{k}^{-}(x^{\prime})|^{p}}{|h^{\prime}|^{p+N-2}}\,dh^{\prime}dx^{\prime}\\
& + c\sum_{k=0}^{m-1}\int_{\mathbb{R}^{N-1}}\int_{B^{\prime}(0,a)} \frac{|\nabla_{\shortparallel}^{m-k-1}f_{k}^{+}(x^{\prime}+h^{\prime} )-\nabla_{\shortparallel}^{m-k-1}f_{k}^{+}(x^{\prime})|^{p}}{|h^{\prime}|^{p+N-2}}\,dh^{\prime}dx^{\prime}
\end{align*}
for some constant $c=c(a,m,N,p)>0$. Moreover, the map $(f_0^-,\dotsc,f_{m-1}^-,f_0^+,\dotsc,f_{m-1}^+) \mapsto u$ is linear.
\end{theorem}

\begin{proof}

For simplicity we again take $b^{-}=0$ and set $b:=b^{+}$. Without loss of generality, we may assume that $0<a<b$.  Proposition \ref{proposition mollifier} provides us with a function $\varphi\in C_{c}^{m}(\mathbb{R}^{N-1})$ such that $\operatorname*{supp}\varphi\subseteq B^{\prime}(0,ab^{-1})$ and
\begin{equation*}
\int_{\mathbb{R}^{N-1}}\varphi(y^{\prime})\,dy^{\prime}=1 \text{ and } \int_{\mathbb{R}^{N-1}}(y^{\prime})^{\alpha}\varphi(y^{\prime})\,dy^{\prime}=0
\end{equation*}
for all multi-indices $\alpha\in\mathbb{N}_{0}^{N-1}$ with $1\leq|\alpha|\leq m$.  We define $u^- : \Omega \to \mathbb{R}$ via 
\begin{align*}
u^{-}(x):  &  =\sum_{k=0}^{m-1}\frac{x_{N}^{k}}{k!}\int_{\mathbb{R}^{N-1}} f_{k}^{-}(y^{\prime})\varphi_{x_{N}}\left(  x^{\prime}-y^{\prime}\right) \,dy^{\prime}\\
&  =\sum_{k=0}^{m-1}\frac{x_{N}^{k}}{k!}\int_{\mathbb{R}^{N-1}}f_{k}^{-}(x^{\prime} x_{N}z^{\prime})\varphi\left(  z^{\prime}\right) \,dz^{\prime}.
\end{align*}
We divide the remainder of the proof into several steps.

\textbf{Step 1 -- Computing traces:} We claim that $\operatorname*{Tr}(u^{-})=f_{0}^{-}$ on
$\Gamma^{-}$ and that $\operatorname*{Tr}(\partial_{N}^{l}u^{-})=f_{l}^{-}$ on $\Gamma^{-}$ for $1\leq l\leq m-1$. By standard properties of mollifiers we have that $\varphi_{x_{N}}\ast f_{i}^{-}\rightarrow f_{i}^{-}$ in $L_{\operatorname*{loc}}^{p}(\mathbb{R}^{N-1})$ as $x_{N}\rightarrow0^{+}$ for all $i=0,\ldots,m-1$. Hence, $\operatorname*{Tr}(u^{-})=f_{0}^{-}$ on $\Gamma^{-}$. Next we compute $\partial_{N}^{l}u^{-}$ for $1\leq l\leq m$. We
have%
\begin{equation}\label{partial N l}
\begin{split}
\partial_{N}^{l}u^{-}(x)  &  =\sum_{k=0}^{m-1}\frac{1}{k!}\sum_{i=0}^{l} \binom{l}{i}\partial_{N}^{l-i}\left(  x_{N}^{k}\right)  \partial_{N}^{i}\left(  \int_{\mathbb{R}^{N-1}}f_{k}^{-}(x^{\prime}-x_{N}z^{\prime})\varphi\left(  z^{\prime}\right)  \,dz^{\prime}\right)   \\
&  =\sum_{k=0}^{m-1}\sum_{i=\max\{0,l-k\}}^{l}c_{i,k,l}x_{N}^{k-l+i} \partial_{N}^{i}\left(  \int_{\mathbb{R}^{N-1}}f_{k}^{-}(x^{\prime} -x_{N}z^{\prime})\varphi\left(  z^{\prime}\right)  \,dz^{\prime}\right),
\end{split}
\end{equation}
where $c_{i,k,l}:=\frac{1}{k!}\binom{l}{i}\frac{k!}{(k-l+i)!}$. Note that
\begin{equation}\label{c0}
c_{0,l,l}=1. 
\end{equation}
To compute $\partial_{N}^{i}\left(  \int_{\mathbb{R}^{N-1}}f_{k}^{-} (x^{\prime}-x_{N}z^{\prime})\varphi\left(  z^{\prime}\right)  \,dz^{\prime} \right)$ we distinguish two cases.

\emph{Case 1:} If $1\leq i\leq m-1-k$, then we use the fact that $f_{k}^{-} \in W_{\operatorname*{loc}}^{m-1-k,p}(\mathbb{R}^{N-1})$ to see that
\begin{align}\label{partial N i}
\begin{split}
\partial_{N}^{i}  &  \int_{\mathbb{R}^{N-1}}f_{k}^{-}(x^{\prime} - x_{N}z^{\prime})\varphi\left(  z^{\prime}\right)  \,dz^{\prime} 
= \int_{\mathbb{R}^{N-1}}\frac{\partial^{i}}{\partial x_{N}^{i}}\left( f_{k}^{-}(x^{\prime}-x_{N}z^{\prime})\right)  \varphi\left(  z^{\prime }\right)  \,dz^{\prime}    \\
&  =\sum_{|\alpha|=i}c_{\alpha}\int_{\mathbb{R}^{N-1}}\partial^{\alpha} f_{k}(x^{\prime}-x_{N}z^{\prime})(z^{\prime})^{\alpha}\varphi\left( z^{\prime}\right)  \,dz^{\prime}=\sum_{|\alpha|=i}\int_{\mathbb{R}^{N-1} }\partial^{\alpha}f_{k}(y^{\prime})\psi_{x_{N}}^{\alpha}(x^{\prime}-y^{\prime
})\,dy^{\prime}, 
\end{split}
\end{align}
where $\psi^{\alpha}(y^{\prime}):=c_{\alpha}(y^{\prime})^{\alpha} \varphi\left(  y^{\prime}\right)  $ satisfies
\begin{equation}\label{av zero}
\int_{\mathbb{R}^{N-1}}\psi^{\alpha}(y^{\prime})\,dy^{\prime}=0.
\end{equation}
Observe that since $l-k\leq i\leq m-1-k$, this first case only happens if $l\leq m-1$.

\emph{Case 2:} If $i>m-1-k$, then we write
\begin{align*}
\partial_{N}^{i}  &  \int_{\mathbb{R}^{N-1}}f_{k}^{-}(x^{\prime}%
-x_{N}z^{\prime})\varphi\left(  z^{\prime}\right)  \,dz^{\prime}=\partial
_{N}^{i-(m-1-k)}\int_{\mathbb{R}^{N-1}}\frac{\partial^{m-1-k}}{\partial
x_{N}^{m-1-k}}\left(  f_{k}^{-}(x^{\prime}-x_{N}z^{\prime})\right)
\varphi\left(  z^{\prime}\right)  \,dz^{\prime}\\
&  =\partial_{N}^{i-(m-1-k)}\sum_{|\beta|=m-1-k}c_{\beta}\int_{\mathbb{R}%
^{N-1}}\partial^{\beta}f_{k}(x^{\prime}-x_{N}z^{\prime})(z^{\prime})^{\beta
}\varphi\left(  z^{\prime}\right)  \,dz^{\prime}\\
&  =\partial_{N}^{i-(m-1-k)}\sum_{|\beta|=m-1-k}c_{\beta}\int_{\mathbb{R}%
^{N-1}}\partial^{\beta}f_{k}(y^{\prime})\left(  \frac{x^{\prime}-y^{\prime}%
}{x_{N}}\right)  ^{\beta}\varphi_{x_{N}}\left(  x^{\prime}-y^{\prime}\right)
\,dy^{\prime}\\
&  =\sum_{|\beta|=m-1-k}c_{\beta}\int_{\mathbb{R}^{N-1}}\partial^{\beta}%
f_{k}(y^{\prime})\frac{\partial^{i-(m-1-k)}}{\partial x_{N}^{i-(m-1-k)}%
}\left[  \left(  \frac{x^{\prime}-y^{\prime}}{x_{N}}\right)  ^{\beta}%
\varphi_{x_{N}}\left(  x^{\prime}-y^{\prime}\right)  \right]  \,dy^{\prime}\\
&  =\sum_{|\beta|=m-1-k}\frac{1}{x_{N}^{i-(m-1-k)}}\int_{\mathbb{R}^{N-1}%
}\partial^{\beta}f_{k}(y^{\prime})\psi_{x_{N}}^{\beta,i}(x^{\prime}-y^{\prime
})\,dy^{\prime},
\end{align*}
where the last equality follows from Proposition \ref{proposition derivatives mollifiers}, and where
\begin{equation}\label{av zero 1}
\int_{\mathbb{R}^{N-1}}\psi^{\beta,i}(y^{\prime})\,dy^{\prime}=0.
\end{equation}

By combining the two cases we can write
\begin{align*}
\partial_{N}^{l}    u^{-}(x) = &\sum_{k=0}^{m-1}\sum_{i=\max\{0,l-k\}}^{m-1-k}c_{i,k,l}x_{N}^{k-l+i}\sum_{|\alpha|=i}\int_{\mathbb{R}^{N-1}} \partial^{\alpha}f_{k}(y^{\prime})\psi_{x_{N}}^{\alpha}(x^{\prime}-y^{\prime })\,dy^{\prime} \\
&  +x_{N}^{m-l-1}\sum_{k=0}^{m-1}\sum_{|\beta|=m-1-k}\int_{\mathbb{R}^{N-1}} \partial^{\beta}f_{k}(y^{\prime})\psi_{x_{N}}^{\beta,k,l}(x^{\prime} - y^{\prime})\,dy^{\prime},
\end{align*}
where $\psi^{0}:=\varphi$ and
\begin{equation*}
\psi^{\beta,k,l}(y^{\prime}):=\sum_{i=\max\{0,m-k\}}^{l}c_{i,k,l}\psi^{\beta,i}(y^{\prime}).
\end{equation*}
Assume now that $1\leq l\leq m-1$. Then by \eqref{av zero} and \eqref{av zero 1}, $\psi_{x_{N}}^{\alpha}\ast\partial^{\alpha}f_{k}\rightarrow0$ if $\alpha\neq 0$ and $\psi_{x_{N}}^{\beta,k,l}\ast \partial^{\beta}f_{k}\rightarrow0$ in $L_{\operatorname*{loc}}^{p}(\mathbb{R}^{N-1})$ as $x_{N}\rightarrow0^{+}$. On the other hand, if $\alpha=0$, then $x_{N}^{k-l}\rightarrow0$ if $k>l$ as $x_{N}\rightarrow0^{+}$. Hence, the only term in the first term that does not converge to zero is $k=l$ and $i=0$. In view of \eqref{c0} and of the fact that $\psi^{0}:=\varphi$, we have that $\operatorname*{Tr}(\partial_{N}^{l}u^{-})=f_{l}^{-}$ on $\Gamma^{-}$.

\textbf{Step 2 -- Derivative estimates:}  We claim that $u^{-}\in\dot{W}^{m,p}(\Omega)$. We begin by estimating $\partial_{N}^{m}u^{-}$. Since Case 1 in Step 1 does not happen when $l=m$, we can write
\begin{equation*}
\partial_{N}^{m}u^{-}(x)=\sum_{k=0}^{m-1}\sum_{|\beta|=m-1-k} \int_{\mathbb{R}^{N-1}}\partial^{\beta}f_{k}(y^{\prime})\frac{1}{x_{N}^{N}} \psi^{\beta,k,l} \left(  \frac{x^{\prime}-y^{\prime}}{x_{N}}\right) \,dy^{\prime}.
\end{equation*}
Due to \eqref{av zero 1}, we are in a position to apply Proposition \ref{proposition structure} to conclude that $\partial_{N}^{m}u^{-}\in L^{p}(\Omega)$, with
\begin{equation*}
\int_{\Omega}|\partial_{N}^{m}u^{-}(x)|^{p}dx
\leq c\int_{\mathbb{R}^{N-1}}  \int_{B^{\prime}(0,a)}\frac{|\nabla_{\shortparallel}^{m-k-1}f_{k}^{-}(x^{\prime}+h^{\prime})-\nabla_{\shortparallel}^{m-k-1}f_{k}^{-} (x^{\prime})|^{p}}{|h^{\prime}|^{p+N-2}}\,dh^{\prime}dx^{\prime}.
\end{equation*}
On the other hand, if $0\leq l<m$ in Step 1, then for every multi-index $\gamma=(\gamma^{\prime},0)$ with $|\gamma^{\prime}|=m-l$ we have
\begin{align*}
\partial^{\gamma}\partial_{N}^{l}  &  u^{-}(x)=\sum_{k=0}^{m-1} \sum_{i=\max\{0,l-k\}}^{m-1-k}c_{i,k,l}x_{N}^{k-l+i}\sum_{|\alpha|=i} \int_{\mathbb{R}^{N-1}}\partial^{\alpha}f_{k}(y^{\prime})\frac{\partial^{\gamma}}{\partial x^{\gamma}} \left(  \psi_{x_{N}}^{\alpha}(x^{\prime
}-y^{\prime})\right)  \,dy^{\prime}\\
&  +x_{N}^{m-l-1}\sum_{k=0}^{m-1}\sum_{|\beta|=m-1-k}\int_{\mathbb{R}^{N-1}}\partial^{\beta}f_{k}(y^{\prime})\frac{\partial^{\gamma}}{\partial x^{\gamma}} \left(  \psi_{x_{N}}^{\beta,k,l}(x^{\prime}-y^{\prime})\right) \,dy^{\prime}.
\end{align*}
In the first term for every $i<m-k-1$ we write $\gamma=s_{i}+t_{i}$, where $|s_{i}|=m-k-1-i$ and $|t_{i}|=k-l+i+1$, and integrate by parts $m-k-1-i$ times to see that
\begin{align*}
\partial^{\gamma}\partial_{N}^{l}  &  u^{-}(x)=\sum_{k=0}^{m-1} \sum_{i=\max\{0,l-k\}}^{m-1-k}c_{i,k,l}x_{N}^{k-l+i}\sum_{|\alpha|=i} \int_{\mathbb{R}^{N-1}}\partial^{\alpha+s_{i}}f_{k}(y^{\prime}) \frac{\partial^{t_{i}}}{\partial x^{t_{i}}} \left(  \psi_{x_{N}}^{\alpha}(x^{\prime
}-y^{\prime})\right)  \,dy^{\prime}\\
&  +x_{N}^{m-l-1}\sum_{k=0}^{m-1}\sum_{|\beta|=m-1-k}\int_{\mathbb{R}^{N-1} }\partial^{\beta}f_{k}(y^{\prime})\frac{\partial^{\gamma}}{\partial x^{\gamma}}\left(  \psi_{x_{N}}^{\beta,k,l}(x^{\prime}-y^{\prime})\right) \,dy^{\prime}.
\end{align*}
According to Proposition \ref{proposition derivatives mollifiers}, we may write
\begin{align*}
\partial^{\gamma}\partial_{N}^{l}  &  u^{-}(x)=\sum_{k=0}^{m-1} \sum_{i=\max\{0,l-k\}}^{m-1-k}c_{i,k,l}\sum_{|\alpha|=i}\int_{\mathbb{R}^{N-1}} \partial^{\alpha+s_{i}}f_{k}(y^{\prime})\frac{1}{x_{N}^{N}}\psi^{\alpha,t_{i}}\left(  \frac{x^{\prime}-y^{\prime}}{x_{N}}\right)  \,dy^{\prime}\\
&  +\sum_{k=0}^{m-1}\sum_{|\beta|=m-1-k}\int_{\mathbb{R}^{N-1}}\partial^{\alpha} f_{k}(y^{\prime})\frac{1}{x_{N}^{N}}\psi^{\beta,\gamma} \left( \frac{x^{\prime}-y^{\prime}}{x_{N}}\right)  \,dy^{\prime},
\end{align*}
where $\int_{\mathbb{R}^{N-1}}\psi^{\alpha,t_{i}}(y^{\prime})\,dy^{\prime
}=\int_{\mathbb{R}^{N-1}}\psi^{\beta,\gamma}(y^{\prime})\,dy^{\prime}=0$.
Hence, Proposition \ref{proposition structure} tells us that $\partial^{\gamma} \partial_{N}^{l}u^{-}\in L^{p}(\Omega)$ with
\begin{equation*}
\int_{\Omega}|\partial^{\gamma}\partial_{N}^{l}u^{-}(x)|^{p}dx
\leq c\int_{\mathbb{R}^{N-1}}\int_{B^{\prime}(0,a)}\frac{|\nabla_{\shortparallel}^{m-k-1} f_{k}^{-}(x^{\prime}+h^{\prime})-\nabla_{\shortparallel}^{m-k-1} f_{k}^{-}(x^{\prime})|^{p}}{|h^{\prime}|^{p+N-2}}\,dh^{\prime}dx^{\prime}.
\end{equation*}
Thus $u^- \in \dot{W}^{m,p}(\Omega)$, as claimed.

\textbf{Step 3 -- Defining $u^+$ and $u$:} Reasoning as in Steps 1 and 2, we can show that the function $u^+: \Omega \to \mathbb{R}$ defined by
\begin{equation*}
u^{+}(x):=\sum_{k=0}^{m-1}\frac{(b-x_{N})^{k}}{k!}\int_{\mathbb{R}^{N-1}} f_{k}^{+}(y^{\prime})\varphi_{b-x_{N}}\left(  x^{\prime}-y^{\prime}\right) \,dy^{\prime}
\end{equation*}
belongs to $\dot{W}^{m,p}(\Omega)$, obeys the estimate
\begin{equation*}
\int_{\Omega}|\nabla^{m}u^{+}(x)|^{p}dx
\leq c\sum_{k=0}^{m-1}\int_{\mathbb{R}^{N-1}}\int_{B^{\prime}(0,a)} \frac{|\nabla_{\shortparallel
}^{m-k-1}f_{k}^{+}(x^{\prime}+h^{\prime})-\nabla_{\shortparallel}^{m-k-1} f_{k}^{+}(x^{\prime})|^{p}}{|h^{\prime}|^{p+N-2}}\,dh^{\prime}dx^{\prime}, 
\end{equation*}
and satisfies $\operatorname*{Tr}(u^{+})=f_{0}^{+}$ on $\Gamma^{+}$ and $\operatorname*{Tr}(\partial_{N}^{l}u^{+})=f_{l}^{+}$ on $\Gamma^{+}$ for $1\leq l\leq m-1$.

Now let $\theta\in C^{\infty}([0,b])$ be such that $\theta=1$ in a neighborhood of $0$ and $\theta=0$ in a neighborhood of $b$. We then define $u : \Omega \to \mathbb{R}$ via 
\begin{equation*}
u(x):=\theta(x_{N})u^{-}(x)+(1-\theta(x_{N}))u^{+}(x).
\end{equation*}
For every $1\leq l\leq m$ we have that
\begin{align*}
\partial_{N}^{l}u(x)  &  =\theta(x_{N})\partial_{N}^{l}u^{-}(x)+(1-\theta(x_{N})) \partial_{N}^{l}u^{+}(x)\\
&  \quad+\sum_{i=0}^{l}\binom{l}{i}\theta^{(l-i)}(x_{N})(\partial_{N}^{i} u^{-}(x)-\partial_{N}^{i}u^{+}(x)).
\end{align*}
If $0\leq l<m$, then for every multi-index $\gamma=(\gamma^{\prime},0)$ with $|\gamma^{\prime}|=m-l$ we have
\begin{align*}
\partial^{\gamma}\partial_{N}^{l}u(x)  
&  =\theta(x_{N})\partial^{\gamma} \partial_{N}^{l}u^{-}(x) + (1-\theta(x_{N})) \partial^{\gamma} \partial_{N}^{l} u^{+}(x)\\
&  \quad+\sum_{i=0}^{l}\binom{l}{i}\theta^{(l-i)}(x_{N})(\partial^{\gamma}\partial_{N}^{i}u^{-}(x)-\partial^{\gamma}\partial_{N}^{i}u^{+}(x))
\end{align*}
if $l\geq1$, and
\begin{equation*}
\partial^{\gamma}u(x)=\theta(x_{N})\partial^{\gamma}u^{-}(x)+(1-\theta(x_{N})) \partial^{\gamma}u^{+}(x)
\end{equation*}
if $l=0$.  Since $u^{\pm}\in\dot{W}^{m,p}(\Omega)$, to prove that $u \in \dot{W}^{m,p}(\Omega)$ it suffices to estimate the lower order terms $\theta^{(l-i)}\partial^{\gamma}\partial_{N}^{i}(u^{-}-u^{+})$. 

By Taylor's formula and \eqref{formula taylor inverted m=2} applied to $\partial_{N}^{i}u^{-}(x^{\prime},\cdot)$ and $\partial_{N}^{i}u^{+}(x^{\prime},\cdot)$, respectively, we may write
\begin{equation*}
\partial_{N}^{i}u^{-}(x)    =\sum_{k=i}^{l-1}\frac{1}{(k-i)!}f_{k}^{-}(x^{\prime})x_{N}^{k-i}+\frac{1}{(l-i+1)!}\int_{0}^{x_{N}}\partial_{N}^{l}u^{-}(x^{\prime},s)\left(  x_{N}-s\right)  ^{l-i+1}ds 
\end{equation*}
and
\begin{equation*}
\partial_{N}^{i}u^{+}(x)    
=\sum_{k=i}^{l-1}\frac{(-1)^{k-i}}{(k-i)!}f_{k}^{+}(x^{\prime}) (b-x_{N})^{k-i}+\frac{1}{(l-i+1)!}\int_{x_{N}}^{b}\partial_{N}^{l}u^{+}(x^{\prime},s)\left(  x_{N}-s\right)  ^{l-i+1}ds,
\end{equation*}
where we have used the fact that $\operatorname*{Tr}(u^{\pm})=f_{0}^{\pm}$ and $\operatorname*{Tr}(\partial_{N}^{k}u^{\pm})=f_{k}^{\pm}$\ on $\Gamma^{\pm}$. Since $f_{k}^{\pm}\in W_{\operatorname*{loc}}^{m-1-k,p}(\mathbb{R}^{N-1})$ and $\gamma=(\gamma^{\prime},0)$ with $|\gamma^{\prime}|=m-l\leq m-1-k$ for $k\leq l-1$, we can apply $\partial^{\gamma}$ to both sides of the previous two identities to see that
\begin{equation*}
\partial^{\gamma}\partial_{N}^{i}u^{-}(x)    =\sum_{k=i}^{l-1}\frac
{1}{(k-i)!}\partial^{\gamma}f_{k}^{-}(x^{\prime})x_{N}^{k-i}+\frac
{1}{(l-i+1)!}\int_{0}^{x_{N}}\partial^{\gamma}\partial_{N}^{l}u^{-}(x^{\prime
},s)\left(  x_{N}-s\right)  ^{l-i+1}ds 
\end{equation*}
and 
\begin{equation*}
\partial^{\gamma}\partial_{N}^{i}u^{+}(x)    =\sum_{k=i}^{l-1}\frac
{(-1)^{k-i}}{(k-i)!}\partial^{\gamma}f_{k}^{+}(x^{\prime})(b-x_{N}%
)^{k-i}+\frac{1}{(l-i+1)!}\int_{x_{N}}^{b}\partial^{\gamma}\partial_{N}%
^{l}u^{+}(x^{\prime},s)\left(  x_{N}-s\right)  ^{l-i+1}ds.
\end{equation*}
Hence,
\begin{align*}
\partial^{\gamma}\partial_{N}^{i}(u^{+}(x)-u^{-}(x))  
&  =\sum_{k=i}^{l-1}\frac{1}{(k-i)!}[(-1)^{k-i}\partial^{\gamma}f_{k}^{+}(x^{\prime}) (b-x_{N})^{k-i}-\partial^{\gamma}f_{k}^{-}(x^{\prime})x_{N}^{k-i}]\\
&  \quad+\frac{1}{(l-i+1)!}\int_{x_{N}}^{b}\partial^{\gamma}\partial_{N}^{l}u^{+}(x^{\prime},s)\left(  x_{N}-s\right)  ^{l-i+1}ds\\
&  \quad-\frac{1}{(l-i+1)!}\int_{0}^{x_{N}}\partial^{\gamma}\partial_{N}^{l}u^{-}(x^{\prime},s)\left(  x_{N}-s\right)  ^{l-i+1}ds.
\end{align*}
Now define
\begin{equation*}
P_{i,l,\gamma}(x)    :=\sum_{k=i}^{l-1}\frac{1}{(k-i)!}[(-1)^{k-i} \partial^{\gamma}f_{k}^{+}(x^{\prime})(b-x_{N})^{k-i}-\partial^{\gamma} f_{k}^{-}(x^{\prime})x_{N}^{k-i}] 
\end{equation*}
and 
\begin{equation*}
Q_{j,l,\gamma}(x^{\prime})    :=\sum_{k=0}^{l-1-j}\frac{(-1)^{k}}{k!} (\partial^{\gamma}f_{k+j}^{\pm}(x^{\prime})+(-1)^{k+1}\partial^{\gamma} f_{k+j}^{\pm}(x^{\prime}))\left(  \frac{b}{2}\right)^{k} 
\end{equation*}
and note that \eqref{compatibility conditions} guarantees that $Q_{j,l,\gamma}\in L^{p}(\mathbb{R}^{N-1})$. Using the facts that for $j=1,\ldots,l-1$, with $j\leq k-i$,
\begin{equation*}
\partial_{N}^{j}((b-x_{N})^{k-i})=(-1)^{j}\frac{(k-i)!}{(k-i-j)!} (b-x_{N})^{k-i-j},\quad\partial_{N}^{j}(x_{N}{}^{k-i})=\frac{(k-i)!}{(k-i-j)!}x_{N}^{k-i-j},
\end{equation*}
we have that
\begin{align*}
\partial_{N}^{j}P_{i,l,\gamma}\left(  x^{\prime},\frac{b}{2}\right)   
&=\sum_{k=j+i}^{l-1}\frac{1}{(k-i-j)!}[(-1)^{k-i-j}\partial^{\gamma}f_{k}^{+}(x^{\prime})-\partial^{\gamma}f_{k}^{-}(x^{\prime})]\left(  \frac{b}{2}\right)^{k-i-j}\\
&  =\sum_{n=0}^{l-1-j-i}\frac{1}{n!}[(-1)^{n}\partial^{\gamma}f_{n+j+i}^{+}(x^{\prime})-\partial^{\gamma}f_{n+j+i}^{-}(x^{\prime})] \left(  \frac{b}{2}\right)^{n}=Q_{j+i,l,\gamma}(x^{\prime}).
\end{align*}
Also $P_{i,l,\gamma}\left(  x^{\prime},\frac{b}{2}\right)  =Q_{i,l,\gamma}(x^{\prime})$. Hence, if we fix $x^{\prime}\in\mathbb{R}^{N-1}$ and apply Taylor's formula centered at $\frac{b}{2}$ to the function $P_{i,l,\gamma}(x^{\prime},\cdot)$, we get
\begin{equation*}
P_{i,l,\gamma}(x)=\sum_{j=0}^{k-1}\partial_{N}^{j}P_{i,l,\gamma} \left(x^{\prime},\frac{b}{2}\right)  \left(  x_{N}-\frac{b}{2}\right)^{j}
=\sum_{j=0}^{k-1}Q_{j+i,l,\gamma}(x^{\prime})\left(  x_{N}-\frac{b}{2}\right)^{j}.
\end{equation*}
In turn,  H\"{o}lder's inequality and the fact that $|\theta^{(l-i)}(x_{N})|\leq cb^{-(l-i)}$ imply that
\begin{align*}
|\theta^{(l-1)}(x_{N})  &  \partial^{\gamma}\partial_{N}^{i}(u^{+}(x)-u^{-}(x))|^{p}\leq c\sum_{j=0}^{k-1}b^{(j-l+i)p}|Q_{j+i,l,\gamma}(x^{\prime})|^{p} \\
&  \quad+cb^{-1}\int_{x_{N}}^{b}|\partial^{\gamma}\partial_{N}^{l} u^{+}(x^{\prime},s)|^{p}ds+cb^{-1}\int_{0}^{x_{N}}|\partial^{\gamma} \partial_{N}^{l}u^{-}(x^{\prime},s)|^{p}ds.
\end{align*}
Integrating over $\Omega$ and using Tonelli's theorem then shows that
\begin{align*}
\int_{\Omega}|\theta^{(l-1)}(x_{N})\partial^{\gamma}\partial_{N}^{i} (u^{+}(x)-u^{-}(x))|^{p}dx  &  \leq c\sum_{j=0}^{k-1}b^{(j-l+i)p+1} \int_{\mathbb{R}^{N-1}}|Q_{j+i,l,\gamma}(x^{\prime})|^{p}dx^{\prime}\\
&  \quad+c\int_{\Omega}|\nabla^{m}u^{+}(x)|^{p}dx+c\int_{\Omega}|\nabla^{m}u^{-}(x)|^{p}dx.
\end{align*}
This completes the proof.

\end{proof}

\section{Traces in general strip-like domains}\label{section traces strip-like}

In this section we consider the case in which $\Omega$ is as in
\eqref{omega two graphs}. We define
\begin{equation*}
\Gamma^{\pm}:=\left\{  x\in\mathbb{R}^{N}\,|\,x_{N}=\eta^{\pm}(x^{\prime })\right\}  .
\end{equation*}

\subsection{The case $m=1$, $p>1$}

We begin with the case $m=1$ and $p>1$.  First we prove the trace estimate, the analog of Theorem \ref{theorem trace strip}.

\begin{theorem}\label{theorem trace m=1}
Let $\Omega$ be as in \eqref{omega two graphs}, where $\eta^{\pm}:\mathbb{R}^{N-1}\rightarrow\mathbb{R}$ are Lipschitz continuous functions such that $\eta^{-}<\eta^{+}$.  Set $L:=|\eta^{-}|_{0,1} + |\eta^{+}|_{0,1}$, and let $1<p<\infty$.  There exists a unique linear operator
\begin{equation*}
\operatorname*{Tr}:\dot{W}^{1,p}(\Omega)\rightarrow L_{\operatorname*{loc}}^{p}(\partial\Omega)
\end{equation*}
such that the following hold.
\begin{enumerate}
 \item $\operatorname*{Tr}(u)=u$ on $\partial\Omega$ for all $u\in\dot {W}^{1,p}(\Omega)\cap C^0(\overline{\Omega})$
 
 \item There exists a constant
$c=c(N,p)>0$ such that%
\begin{gather}
\int_{\mathbb{R}^{N-1}}\frac{|\operatorname*{Tr}(u)(x^{\prime},\eta^{+}(x^{\prime}))-\operatorname*{Tr}(u)(x^{\prime},\eta^{-}(x^{\prime}))|^{p} }{(\eta^{+}(x^{\prime})-\eta^{-}(x^{\prime}))^{p-1}}\,dx^{\prime}\leq
c\int_{\Omega}|\partial_{N}u(x)|^{p}dx,\label{trace 1}\\
\int_{\mathbb{R}^{N-1}}\int_{B^{\prime}(0,(\eta^{+}(x^{\prime})-\eta^{-}(x^{\prime}))/(2L))} \frac{|\operatorname*{Tr}(u)(x^{\prime}+h^{\prime}, \eta^{\pm}(x^{\prime}+h)) - \operatorname*{Tr}(u)(x^{\prime},\eta^{\pm}(x^{\prime}))|^{p}}{|h^{\prime}|^{p+N-2}} \, dh^{\prime}dx^{\prime}\nonumber\\
\leq(1+L)^{p}c\int_{\Omega}|\nabla u(x)|^{p}dx \label{trace2}%
\end{gather}
for all $u\in\dot{W}^{1,p}(\Omega)$.
 
 \item The integration by parts formula
\begin{equation*}
\int_{\Omega}u\partial_{i}\psi\,dx=-\int_{\Omega}\psi\partial_{i}%
u\,dx+\int_{\partial\Omega}\psi\operatorname*{Tr}(u)\nu_{i}\,d\mathcal{H}%
^{N-1}%
\end{equation*}
holds for all $u\in\dot{W}^{1,p}(\Omega)$, all $\psi\in C_{c}^{1}(\mathbb{R}^{N})$, and all $i=1,\ldots,N$.
\end{enumerate}
\end{theorem}

\begin{proof}
 
We will only prove the second item.  The proof of the first and third follow as in the proof of Theorem \ref{theorem trace strip}. We divide the proof into steps.

\textbf{Step 1 -- Special case of \eqref{trace 1}}:   Suppose for now that $\eta^{-}=0$ and $\eta:=\eta^{+}$ and consider $u\in\dot{W}^{1,p}(\Omega)$.  By Theorem \ref{theorem AC} and the fundamental theorem of calculus, for $\mathcal{L}^{N-1}$ a.e. $x^{\prime}\in \mathbb{R}^{N-1}$  we have that
\begin{equation*}
u(x^{\prime},\eta(x^{\prime})) - u(x^{\prime},0) = \int_{0}^{\eta(x^{\prime})} \partial_{N} u(x^{\prime},x_{N}) \,dx_{N}.
\end{equation*}
It then follows from H\"{o}lder's inequality that
\begin{equation*}
|u(x^{\prime},\eta(x^{\prime})) - u(x^{\prime},0)|^{p}\leq(\eta(x^{\prime}))^{p-1} \int_{0}^{\eta(x^{\prime})} |\partial_{N} u(x^{\prime},x_{N})|^{p}dx_{N}.
\end{equation*}
By integrating in $x^{\prime}$ over $\mathbb{R}^{N-1}$ and using Tonelli's theorem, we then see that
\begin{equation*}
\int_{\mathbb{R}^{N-1}}\frac{|u(x^{\prime},\eta(x^{\prime}))-u(x^{\prime},0)|^{p}}{(\eta(x^{\prime}))^{p-1}} \,dx^{\prime} \leq \int_{U_{R}}|\partial_{N} u(x)|^{p}dx.
\end{equation*}
Hence, 
\begin{equation*}
\int_{\mathbb{R}^{N-1}}\frac{|\operatorname*{Tr}(u)(x^{\prime},\eta(x^{\prime
}))-\operatorname*{Tr}(u)(x^{\prime},0)|^{p}}{(\eta(x^{\prime}))^{p-1}%
}\,dx^{\prime}\leq\int_{\Omega}|\partial_{N}u(x)|^{p}dx,
\end{equation*}
which proves \eqref{trace 1} in the special case in which $\eta^-=0$. 

\textbf{Step 2 -- Special case of \eqref{trace2} }:  Again suppose $\eta^-=0$ and $\eta = \eta^+$. Let $u\in\dot{W}^{1,p}(\Omega)$.  We may suppose, without loss of generality, that $L := |\eta|_{0,1} \ge 1$.   Let $R>0$ and $M_{R}\geq 2R+\max_{\overline{B^{\prime}(0,R)}}\eta$.  

For $x^{\prime}\in B^{\prime}(0,R)$ we make two observations, both of which use the fact that $L \ge 1$.  First, $B^{\prime}(x^{\prime},\eta(x^{\prime})/2L) \subset B^\prime(0, M_{R}/2)$.  Second, if $h^{\prime}\in B^{\prime}(0,\eta(x^{\prime})/2L)$ and $0<x_{N}<\eta(x^{\prime})/(2L)$, then by the Lipschitz continuity of $\eta$,
\begin{equation*}
\eta(x^{\prime}+h^{\prime})\geq\eta(x^{\prime})-L|h^{\prime}|>\tfrac{1}{2} \eta(x^{\prime})>x_{N}.
\end{equation*}
With these observations in hand, we may use Theorem \ref{theorem AC} and the fundamental theorem of calculus for $\mathcal{L}^{N-1}$ a.e. $x^\prime \in B^\prime(0,R)$, $h^\prime \in B^\prime(0,\eta(x^{\prime})/2L)$ and $0<x_{N}< |h^\prime| < \eta(x^{\prime})/(2L)$  to see that
\begin{align*}
u(x^{\prime}+h^{\prime},0)-u(x^{\prime},0)  
& = u(x^{\prime}+h^{\prime},x_{N})-u(x^{\prime},x_{N})
-\int_{0}^{x_{N}}\partial_{N}u(x^{\prime} +h^{\prime},s)\,ds
 -\int_{0}^{x_{N}}\partial_{N}u(x^{\prime},s)\,ds \\
 &=\int_{0}^{1} \nabla_{\shortparallel}u(x^{\prime}+sh^{\prime},x_{N})\cdot h^{\prime}ds 
-\int_{0}^{x_{N}}\partial_{N}u(x^{\prime}+h^{\prime},s)\,ds
-\int_{0}^{x_{N}}\partial_{N}u(x^{\prime},s)\,ds.
\end{align*}
In turn, we may estimate
\begin{equation*}
|u(x^{\prime}+h^{\prime},0)-u(x^{\prime},0)|    \leq|h^{\prime}|\int_{0}^{1}|\nabla u(x^{\prime}+sh^{\prime},x_{N})|ds
 +\int_{0}^{|h^{\prime}|}|\nabla u(x^{\prime}+h^{\prime},s)|\,ds
 +\int_{0}^{|h^{\prime}|}|\nabla u(x^{\prime},s)|\,ds.
\end{equation*}
By averaging both sides in $x_{N}$ over the interval $(0,|h^{\prime}|)$ and raising to the power $p$ we then find that
\begin{align*}
|u(x^{\prime}+h^{\prime},0)-u(x^{\prime},0)|^{p}  
&  \leq 3^{p-1}\left(\int_{0}^{|h^{\prime}|}\int_{0}^{1}|\nabla u(x^{\prime}+sh^{\prime},x_{N}) |\,dsdx_{N}\right)^{p}\\
&  +3^{p-1}\left(  \int_{0}^{|h^{\prime}|}|\nabla u(x^{\prime}+h^{\prime},s)|\,ds\right)^{p}
  +3^{p-1}\left(  \int_{0}^{|h^{\prime}|}|\nabla u(x^{\prime},s)|\,ds\right)^{p}.
\end{align*}
We now divide both sides by $|h^{\prime}|^{p+N-2}$, integrate in $h^{\prime}$
over $B^{\prime}(0,\eta(x^{\prime})/2L)$ and in $x^{\prime}$ over $B^{\prime
}(0,R)$, and use Tonelli's theorem to arrive at the estimate
\begin{align*}
&  \int_{B^{\prime}(0,R)}\int_{B^{\prime}(0,\eta(x^{\prime})/2L)} 
\frac{|u(x^{\prime}+h^{\prime},0)-u(x^{\prime},0)|^{p}}{|h^{\prime}|^{p+N-2}} dh^{\prime}dx^{\prime}\\
&  \leq3^{p-1}\int_{B^{\prime}(0,R)}\int_{B^{\prime}(0,\eta(x^{\prime})/2L)} \frac{1}{|h^{\prime}|^{p+N-2}}\left(  \int_{0}^{1}\int_{0}^{|h^{\prime}|}| \nabla u(x^{\prime}+sh^{\prime},x_{N})|\,dx_{N}ds\right)  ^{p}dh^{\prime} dx^{\prime}\\
&  \quad+3^{p-1}\int_{B^{\prime}(0,R)}\int_{B^{\prime}(0,\eta(x^{\prime})/2L)}\frac{1}{|h^{\prime}|^{p+N-2}}\left(  \int_{0}^{|h^{\prime}|}|\nabla u(x^{\prime}+h^{\prime},s)|\,ds\right)  ^{p}dh^{\prime}dx^{\prime} \\
&  \quad+3^{p-1}\int_{B^{\prime}(0,R)}\int_{B^{\prime}(0,\eta(x^{\prime})/2L)} \frac{1}{|h^{\prime}|^{p+N-2}}\left(  \int_{0}^{|h^{\prime}|}|\nabla u(x^{\prime},s)|\,ds\right)^{p} dh^{\prime}dx^{\prime}\\
&  =:3^{p-1}I+3^{p-1}II+3^{p-1}III.
\end{align*}
By H\"{o}lder's inequality,
\begin{equation*}
\int_{0}^{1}\int_{0}^{|h^{\prime}|}|\nabla u(x^{\prime}+sh^{\prime},x_{N})|\,dx_{N}ds
\leq\left(  \int_{0}^{1}\left(  \int_{0}^{|h^{\prime}|} |\nabla u(x^{\prime}+sh^{\prime},x_{N})|\,dx_{N}\right)  ^{p}ds\right)^{1/p}.
\end{equation*}
In turn, by Tonelli's theorem and using spherical coordinates $h^{\prime} = r\sigma^{\prime}$, where $r>0$ and $\sigma^{\prime}\in \mathbb{S}^{N-2}$, and the change of variables $y^{\prime}=x^{\prime}+sr\sigma^{\prime}$, we may estimate the first term by
\begin{align*}
I  &  \leq\int_{0}^{1}\int_{\mathbb{S}^{N-2}}\int_{B^{\prime}(0,R)} \int_{0}^{\eta(x^{\prime})/2L}\frac{1}{r^{p}}\left(  \int_{0}^{r}|\nabla u(x^{\prime} + sr\sigma^{\prime},x_{N})|\,dx_{N}\right)^{p}drdx^{\prime}d\mathcal{H}^{N-2}(\sigma^{\prime})ds\\
&  =\int_{0}^{1}\int_{\mathbb{S}^{N-2}}\int_{0}^{\infty}\int_{B^{\prime}(0,R)} \chi_{(0,\eta(x^{\prime})/2L)}(r)\frac{1}{r^{p}}\left(  \int_{0}^{r}|\nabla u(x^{\prime}+sr\sigma^{\prime},x_{N})|\,dx_{N}\right)^{p} dx^{\prime
}drd\mathcal{H}^{N-2}(\sigma^{\prime})ds\\
&  \leq \beta_{N-1}\int_{0}^{\infty}\int_{B^{\prime}(0,M_{R})}\chi _{(0,\eta(y^{\prime}))}(r) \frac{1}{r^{p}}\left(  \int_{0}^{r}|\nabla u(y^{\prime},x_{N})|\,dx_{N}\right)^{p} dy^{\prime}dr\\
&  =\beta_{N-1}\int_{B^{\prime}(0,M_{R})}\int_{0}^{\eta(y^{\prime})}\frac{1}{r^{p}} \left(  \int_{0}^{r}|\nabla u(y^{\prime},x_{N})|\,dx_{N}\right)^{p}drdy^{\prime},
\end{align*}
where here we used the facts that $\eta(x^{\prime})/2L<\eta(x^{\prime}+sh^{\prime})$ for $x^{\prime}\in B^{\prime}(0,R)$ and $|h^{\prime}|\leq\eta(x^{\prime})/2L$ and that $|x^{\prime}+sr\sigma^{\prime}|<R+r\leq R+\eta(x^{\prime})/2L\leq M_{R}$.  Applying Hardy's inequality (see, for instance, Theorem C.41 in \cite{leoni2017book}) to the right-hand side and using Tonelli's theorem yields the bound
\begin{equation*}
I\leq\frac{\beta_{N-1}p^{p}}{(p-1)^{p}}\int_{B^{\prime}(0,M_{R})} \int_{0}^{\eta(y^{\prime})}|\nabla u(y^{\prime},x_{N})|^{p}dx_{N}dy^{\prime} 
\le \frac{\beta_{N-1}p^{p}}{(p-1)^{p}}\int_{\Omega }|\nabla u(y)|^{p}dy.
\end{equation*}
Similarly,
\begin{align*}
II  &  =\int_{\mathbb{S}^{N-2}}\int_{B^{\prime}(0,R)}\int_{0}^{\eta(x^{\prime})/2L} \frac{1}{r^{p}}\left(  \int_{0}^{r}|\nabla u(x^{\prime}+r\sigma^{\prime},s)|\,ds\right)^{p} dr dx^{\prime}d\mathcal{H}^{N-2}(\sigma^{\prime})\\
&  =\int_{\mathbb{S}^{N-2}}\int_{0}^{\infty}\int_{B^{\prime}(0,R)}\chi_{(0,\eta (x^{\prime})/2L)}(r)\frac{1}{r^{p}}\left(  \int_{0}^{r}|\nabla u(x^{\prime}+r\sigma^{\prime},s)|\,ds\right)^{p} dx^{\prime}drd\mathcal{H}^{N-2}(\sigma^{\prime})\\
&  \leq \beta_{N-1}\int_{0}^{\infty}\int_{B^{\prime}(0,M_{R})}\chi_{(0,\eta(y^{\prime}))}(r)\frac{1}{r^{p}}\left(  \int_{0}^{r}|\nabla u(y^{\prime},s)|\,ds\right)^{p} dx^{\prime} dr\\
&  =\beta_{N-1}\int_{B^{\prime}(0,M_{R})}\int_{0}^{\eta(y^{\prime})} \frac{1}{r^{p}}\left(  \int_{0}^{r}|\nabla u(y^{\prime},s)|\,ds\right)^{p}drdx^{\prime}.
\end{align*}
Again by Hardy's inequality and Tonelli's theorem, we get
\begin{equation*}
II\leq\frac{\beta_{N-1}p^{p}}{(p-1)^{p}}\int_{\Omega}|\nabla u(y)|^{p}dy.
\end{equation*}
The term $III$ is even simpler and can be bounded from above by the same expression. Hence, we arrive at the bound
\begin{equation*}
\int_{B^{\prime}(0,R)}\int_{B^{\prime}(0,\eta(x^{\prime})/2L)}\frac
{|u(x^{\prime}+h^{\prime},0)-u(x^{\prime},0)|^{p}}{|h^{\prime}|^{p+N-2}%
}dh^{\prime}dx^{\prime}\leq3^{p}\frac{\beta_{N-1}p^{p}}{(p-1)^{p}}%
\int_{\Omega}|\nabla u(y)|^{p}dy.
\end{equation*}
Letting $R\rightarrow\infty$ and using the Lebesgue monotone convergence theorem gives \eqref{trace2} in this special case.

\textbf{Step 3 -- General case: }Assume now that $\Omega$ is as in \eqref{omega two graphs} and let
\begin{equation*}
U =\left\{  y\in\mathbb{R}^{N}\,|\,0<y_{N}<\eta^{+}(y^{\prime})-\eta
^{-}(y^{\prime})\right\} . 
\end{equation*}
Consider the Lipschitz transformation $\Phi:\Omega\rightarrow U$ given by
\begin{equation*}
\Phi(x)  =(x^{\prime},x_{N}-\eta^{-}(x^{\prime})).
\end{equation*}
The function $\Phi$ is invertible with Lipschitz inverse $\Phi^{-1} : U \to \Omega$ given by
\begin{equation*}
\Phi^{-1}(y)=(y^{\prime},y_{N}+\eta^{-}(y^{\prime})).
\end{equation*}
Moreover, $\det J_{\Phi^{-1}}(y)=1$ at each point $y \in U$ where $\Phi^{-1}$ is differentiable.  Given $u\in\dot{W}^{1,p}(\Omega)$, we define $w : U \to \mathbb{R}$ via 
\begin{equation*}
w(y) = u(\Phi^{-1}(y))=u(y^{\prime},y_{N}+\eta^{-}(y^{\prime})).
\end{equation*}
Then  we compute
\begin{align*}
\frac{\partial w}{\partial y_{i}}(y)  &  =\frac{\partial u}{\partial x_{i}}(y^{\prime},y_{N}+\eta^{-}(y^{\prime}))+\frac{\partial u}{\partial x_{N}}(y^{\prime},y_{N} + \eta^{-}(y^{\prime}))\frac{\partial\eta^{-}}{\partial y_{i}}(y^{\prime}) \text{ for }  i=1,\ldots,N-1 \text{ and }\\
\frac{\partial w}{\partial y_{N}}(y)  &  =\frac{\partial u}{\partial x_{N}}(y^{\prime},y_{N} + \eta^{-}(y^{\prime})),
\end{align*}
which shows that $w\in\dot{W}^{1,p}(U)$. Hence, by Step 1,
\begin{align*}
\int_{\mathbb{R}^{N-1}}  &  \frac{|\operatorname*{Tr}(u)(x^{\prime},\eta
^{+}(x^{\prime}))-\operatorname*{Tr}(u)(x^{\prime},\eta^{-}(x^{\prime}))|^{p}}{(\eta^{+}(x^{\prime})-\eta^{-}(x^{\prime}))^{p-1}}\,dx^{\prime}\\
&  =\int_{\mathbb{R}^{N-1}}\frac{|\operatorname*{Tr}(w)(x^{\prime},\eta^{+}(x^{\prime})-\eta^{-}(x^{\prime}))-\operatorname*{Tr}(w)(x^{\prime},0)|^{p}}{(\eta^{+}(x^{\prime})-\eta^{-}(x^{\prime}))^{p-1}}\,dx^{\prime}\\
&  \leq\int_{U}|\partial_{N}w(y)|^{p}dy=\int_{U}|\partial_{N}u(\Phi^{-1}(y))|^{p}dy
= \int_{\Omega}|\partial_{N}u(x)|^{p}dx,
\end{align*}
where we used the fact that $\det J_{\Phi^{-1}}(y)=1$.  This proves \eqref{trace 1} in the general case.

Similarly, Step 2 shows that
\begin{align*}
\int_{\mathbb{R}^{N-1}}  &  \int_{B^{\prime}(0,(\eta^{+}(x^{\prime})-\eta
^{-}(x^{\prime}))/(2L))}\frac{|\operatorname*{Tr}(u)(x^{\prime}+h^{\prime
},\eta^{-}(x^{\prime}+h))-\operatorname*{Tr}(u)(x^{\prime},\eta^{-}(x^{\prime
}))|^{p}}{|h^{\prime}|^{p+N-2}}\,dh^{\prime}dx^{\prime}\\
&  =\int_{\mathbb{R}^{N-1}}\int_{B^{\prime}(0,(\eta^{+}(x^{\prime})-\eta
^{-}(x^{\prime}))/(2L))}\frac{|\operatorname*{Tr}(w)(x^{\prime}+h^{\prime
},0)-\operatorname*{Tr}(w)(x^{\prime},0)|^{p}}{|h^{\prime}|^{p+N-2}%
}\,dh^{\prime}dx^{\prime}\\
&  \leq3^{p}\frac{\beta_{N-1}p^{p}}{(p-1)^{p}}\int_{U}|\nabla w(y)|^{p}%
dy\leq(1+L)^{p}(2N)^{1/2}\frac{\beta_{N-1}(3p)^{p}}{(p-1)^{p}}\int_{\Omega
}|\nabla u(x)|^{p}dx.
\end{align*}
This proves \eqref{trace2} in the general case for the lower trace.  The analogous estimate for $\operatorname*{Tr}(u)(\cdot,\eta^{+}(\cdot))$ can be obtained by flattening the top and arguing similarly.  We omit the details for the sake of brevity.
\end{proof}

\begin{remark}
In Step 1 we could have written instead
\begin{align*}
v(x^{\prime}+h^{\prime},0)-v(x^{\prime},0)  &  =v(x^{\prime}+h^{\prime},0)-v(x^{\prime},|h^{\prime}|)+v(x^{\prime},|h^{\prime}|)-v(x^{\prime},0)\\
&  =\int_{0}^{1}\nabla v((x^{\prime},|h^{\prime}|)+s(h^{\prime},-|h^{\prime}|)) \cdot(h^{\prime},-|h^{\prime}|)\,ds
+\int_{0}^{|h^{\prime}|}\partial_{N}v(x^{\prime}+h^{\prime},s)\,ds.
\end{align*}
This would have lead to the better estimate
\begin{equation*}
\int_{\mathbb{R}^{N-1}}\int_{B^{\prime}(0,\eta(x^{\prime})/L)} \frac{|\operatorname*{Tr}(u)(x^{\prime}+h^{\prime},0)-\operatorname*{Tr}(u)(x^{\prime},0)|^{p}}{|h^{\prime}|^{p+N-2}} dh^{\prime} dx^{\prime}
\leq c(N,p)\int_{\Omega}|\nabla u(x)|^{p}dx.
\end{equation*}
However, the proof in Step 1 is simpler to extend to higher order Sobolev spaces.
\end{remark}

Next we prove the corresponding lifting result, which is the analog of Theorem {\ref{theorem lifting strip}.

\begin{theorem}\label{theorem lifting m=1}
Let $\Omega$ be as in \eqref{omega two graphs}, where $\eta^{\pm}:\mathbb{R}^{N-1}\rightarrow\mathbb{R}$ are Lipschitz continuous functions, with $\eta^{-}<\eta^{+}$ and $L:=|\eta^{-}|_{0,1} +|\eta^{+}|_{0,1}$. Let  $0<a<1$ and $1<p<\infty$.  Suppose that $f^{\pm} \in L_{\operatorname*{loc}}^{p}(\mathbb{R}^{N-1})$ satisfy
\begin{equation}\label{difference f plus of minus two graphs}
\int_{\mathbb{R}^{N-1}}\frac{|f^{+}(x^{\prime})-f^{-}(x^{\prime})|^{p}}{(\eta^{+}(x^{\prime}) - \eta^{-}(x^{\prime}))^{p-1}}\,dx^{\prime} <\infty 
\end{equation}
and 
\begin{equation} \label{seminorn two graphs}
|f^{-}|_{\W_{(a(\eta^{+}-\eta^{-}))}^{s,p}(\mathbb{R}^{N-1})}<\infty
,\quad|f^{+}|_{\W_{(a(\eta^{+}-\eta^{-}))}^{s,p}(\mathbb{R}^{N-1})}%
<\infty. 
\end{equation}
Then there exists $u\in\dot{W}^{1,p}(\Omega)$ such that $\operatorname*{Tr}(u)(x^{\prime},\eta^{\pm}(x^{\prime}))=f^{\pm}(x^{\prime})$ for $x^{\prime} \in\mathbb{R}^{N-1}$, and
\begin{align*}
\int_{\Omega}|\nabla u(x)|^{p}dx  &  \leq c(1+L)^{p}\int_{\mathbb{R}^{N-1}}\frac{|f^{+}(x^{\prime})-f^{-}(x^{\prime})|^{p}}{(\eta^{+}(x^{\prime})-\eta^{-}(x^{\prime}))^{p-1}}\,dx^{\prime}\\
&  \quad+c(1+L)^{p}|f^{-}|_{\W_{(a(\eta^{+}-\eta^{-}))}^{s,p}(\mathbb{R}^{N-1})}^{p}
+c(1+L)^{p}|f^{+}|_{\W_{(a(\eta^{+}-\eta^{-}))}^{s,p}(\mathbb{R}^{N-1})}^{p}
\end{align*}
for some constant $c=c(a,N,p)>0$.   Moreover, the map $(f^-,f^+) \mapsto u$ is linear. 
\end{theorem}

\begin{proof}
We divide the proof into steps.

\textbf{Step 1 -- The case $\eta^-=0$}:  Assume first that $\eta^- =0$ and $\eta :=\eta^+ >0$.   Given $0<b\leq a$, let $\varphi\in C_{c}^{\infty}(\mathbb{R}^{N-1})$ be a nonnegative function such that $\int_{\mathbb{R}^{N-1}} \varphi(y^{\prime})\,dy^{\prime}=1$ and $\operatorname*{supp}\varphi\subseteq B^{\prime}(0,b)$. Define $u^- : \Omega \to \mathbb{R}$ via 
\begin{equation*}
u^{-}(x):=(\varphi_{x_{N}}\ast f^{-})(x^{\prime})
=\frac{1}{x_{N}^{N-1}} \int_{\mathbb{R}^{N-1}}f^{-}(y^{\prime})\varphi\left(  \frac{x^{\prime} - y^{\prime}}{x_{N}}\right)  \,dy^{\prime}.
\end{equation*}
As in the proof of Theorem \ref{theorem lifting strip}, for every $i=1,\ldots,N$, we have that
\begin{equation*}
|\partial_{i}u^{-}(x)|^{p}\leq\frac{c}{x_{N}^{N+p-1}}\int_{B^{\prime}(x^{\prime}, x_{N}b)} |f^{-}(y^{\prime})-f^{-}(x^{\prime})|^{p}dy^{\prime}.
\end{equation*}
Integrating both sides over $\Omega$ and using Tonelli's theorem shows that
\begin{equation*}
\begin{split}
\int_{\Omega}    |\partial_{i}u^{-}(x)|^{p}dx  
& \leq c\int_{\mathbb{R}^{N-1}}\int_{0}^{\eta(x^{\prime})}\frac{1}{x_{N}^{N+p-1}} \int_{B^{\prime}(x^{\prime},x_{N}b)}|f^{-}(y^{\prime}) - f^{-}(x^{\prime})|^{p}dy^{\prime}dx_{N} dx^{\prime} \\
&  =c\int_{\mathbb{R}^{N-1}}\int_{B^{\prime}(x^{\prime},b\eta(x^{\prime}))}|f^{-}(y^{\prime})-f^{-}(x^{\prime})|^{p}\int_{b^{-1}|x^{\prime}-y^{\prime}|}^{\eta(x^{\prime})}\frac{1}{x_{N}^{N+p-1}}dx_{N}dy^{\prime}dx^{\prime
}\\
&  \leq c\int_{\mathbb{R}^{N-1}}\int_{B^{\prime}(x^{\prime},b\eta(x^{\prime}))}\frac{|f^{-}(y^{\prime})-f^{-}(x^{\prime})|^{p}}{|x^{\prime}-y^{\prime}|^{N+p-2}}\,dy^{\prime}dx^{\prime}.
\end{split}
\end{equation*}
This shows that $u^{-}\in\dot{W}^{1,p}(\Omega)$. Moreover, since by standard properties of mollifiers $\varphi_{x_{N}}\ast f^{-}\rightarrow f^{-}$ in $L_{\operatorname*{loc}}^{p}(\mathbb{R}^{N-1})$ as $x_{N}\rightarrow0^{+}$, we have that $\operatorname*{Tr}(u^{-})(x^{\prime},0)=f^{-}(x^{\prime})$ for
$x^{\prime}\in\mathbb{R}^{N-1}$.

\textbf{Step 2 -- The general case} Assume now that $\Omega$ is as in \eqref{omega two graphs} and let $U$ and $\Phi$ be as in Step 2 of the proof of Theorem \ref{theorem trace m=1}. Let $v^{-}\in W^{1,p}(U)$ be the function constructed in the previous step with $\operatorname*{Tr}(v^{-})(y^{\prime
},0)=f^{-}(y^{\prime})$ for $y^{\prime}\in\mathbb{R}^{N-1}$. Set $u^- : \Omega \to \mathbb{R}$ via
\begin{equation*}
u^{-}(x):=v^{-}(\Phi(x))=v^{-}(x^{\prime},x_{N}-\eta^{-}(x^{\prime})).
\end{equation*}
Then $\operatorname*{Tr}(u^{-})(x^{\prime},\eta^{-}(x^{\prime}))=\operatorname*{Tr}(v^{-})(x^{\prime},0)=f^{-}(x^{\prime})$ for $x^{\prime}\in\mathbb{R}^{N-1}$, and 
\begin{align*}
\frac{\partial u^{-}}{\partial x_{i}}(x)  &  =\frac{\partial v^{-}}{\partial y_{i}}(x^{\prime},x_{N}-\eta^{-}(x^{\prime}))-\frac{\partial v^{-}}{\partial y_{N}}(x^{\prime},x_{N}-\eta^{-}(x^{\prime}))\frac{\partial\eta^{-}}{\partial x_{i}}(x^{\prime}) \text{ for }i=1,\ldots,N-1, \text{ and } \\
\frac{\partial u^{-}}{\partial x_{N}}(x)  &  =\frac{\partial v^{-}}{\partial
y_{N}}(x^{\prime},x_{N}-\eta^{-}(x^{\prime})).
\end{align*}
Then Step 1 allows us to bound
\begin{align*}
\int_{\Omega}|\partial_{i}u^{-}(x)|^{p}dx  &  \leq c(1+L)^{p}\int_{\Omega}|\nabla_{y}v^{-}(\Phi(x))|^{p}dx=c(1+L)^{p}\int_{U}|\nabla_{y}v^{-} (y)|^{p}dy\\
&  \leq c(1+L)^{p}\int_{\mathbb{R}^{N-1}}\int_{B^{\prime}(x^{\prime},(\eta^{+}(x^{\prime})-\eta^{-}(x^{\prime}))b)}\frac{|f^{-}(y^{\prime})-f^{-}(x^{\prime})|^{p}}{|x^{\prime}-y^{\prime}|^{N+p-2}}\,dy^{\prime
}dx^{\prime},
\end{align*}
and
\begin{align*}
\int_{\Omega}|\partial_{N}u^{-}(x)|^{p}dx  &  =\int_{\Omega}|\partial_{N} v^{-}(\Phi(x))|^{p}dx=\int_{U}|\partial_{N}v^{-}(y)|^{p}dy\\
&  \leq c\int_{\mathbb{R}^{N-1}}\int_{B^{\prime}(x^{\prime},(\eta^{+}(x^{\prime})-\eta^{-}(x^{\prime}))b)}\frac{|f^{-}(y^{\prime} ) - f^{-}(x^{\prime})|^{p}}{|x^{\prime}-y^{\prime}|^{N+p-2}}\,dy^{\prime
}dx^{\prime},
\end{align*}
which in particular means that $u^- \in \dot{W}^{1,p}(\Omega)$.

Arguing similarly, we can construct a function $u^{+}\in\dot{W}^{1,p}(\Omega)$ such that $\operatorname*{Tr}u^{+}(x^{\prime},\eta^{+}(x^{\prime}))=f^{+}(x^{\prime})$ for $x^{\prime}\in\mathbb{R}^{N-1}$ and which obeys the estimates 
\begin{equation*}
 \int_{\Omega}|\partial_{i}u^{+}(x)|^{p}dx    \leq c(1+L)^{p} \int_{\mathbb{R}^{N-1}}\int_{B^{\prime}(x^{\prime},(\eta^{+}(x^{\prime})-\eta^{-} (x^{\prime}))b)}\frac{|f^{+}(y^{\prime}) - f^{+}(x^{\prime})|^{p}%
}{|x^{\prime}-y^{\prime}|^{N+p-2}}\,dy^{\prime}dx^{\prime}
\end{equation*}
and 
\begin{equation*}
\int_{\Omega}|\partial_{N}u^{+}(x)|^{p}dx    \leq c\int_{\mathbb{R}^{N-1}} \int_{B^{\prime}(x^{\prime},(\eta^{+}(x^{\prime})-\eta^{-}(x^{\prime}))b)}\frac{|f^{+}(y^{\prime})-f^{+}(x^{\prime})|^{p}}{|x^{\prime}-y^{\prime }|^{N+p-2}}\,dy^{\prime}dx^{\prime}. 
\end{equation*}
Consider now a function $\vartheta\in C^{\infty}([0,1])$ such that $\vartheta=1$ in $[0,\delta]$ and $\vartheta=0$ in $[1-\delta,1]$, and define $\theta: \Omega \to \mathbb{R}$ via
\begin{equation*}
\theta(x) =\vartheta\left(  \frac{x_{N}-\eta^{-}(x^{\prime})}{\eta^{+}(x^{\prime}) - \eta^{-}(x^{\prime})}\right).
\end{equation*}
We then define the function $u: \Omega \to \mathbb{R}$ via
\begin{equation*}
u(x) = \theta(x)u^{-}(x)+(1-\theta(x))u^{+}(x).
\end{equation*}
Then for $i=1,\ldots,N$,
\begin{equation*}
\partial_{i}u(x)=\partial_{i}\theta(x)(u^{-}(x)-u^{+}(x))+\theta
(x)\partial_{i}u^{-}(x)+(1-\theta(x))\partial_{i}u^{+}(x).
\end{equation*}
From this and the fact that $u^\pm \in \dot{W}^{1,p}(\Omega)$ we see that in order to prove that $u \in \dot{W}^{1,p}(\Omega)$ we must only prove that  $\partial_{i}\theta(u^{-}-u^{+})\in L^{p}(\Omega)$.

For $i=1,\ldots,N-1$, we compute 
\begin{equation*}
\partial_{i}\theta(x)    =\vartheta^{\prime}\left(  \frac{x_{N}-\eta^{-}(x^{\prime})}{\eta^{+}(x^{\prime})-\eta^{-}(x^{\prime})}\right)  
\left[ \frac{-\partial_{i}\eta^{-}(x^{\prime})}{\eta^{+}(x^{\prime})-\eta^{-}(x^{\prime})}  -  \frac{(x_{N}-\eta^{-}(x^{\prime}))(\partial_{i}\eta^{+}(x^{\prime}) - \partial_{i} \eta^{-}(x^{\prime}))}{(\eta^{+}(x^{\prime})-\eta^{-} (x^{\prime}))^{2}}\right]  ,
\end{equation*}
while for $i=N$ we compute 
\begin{equation*}
 \partial_{N}\theta(x)    =\vartheta^{\prime}\left(  \frac{x_{N}-\eta^{-}(x^{\prime})}{\eta^{+}(x^{\prime})-\eta^{-}(x^{\prime})}\right)  \frac{1}{\eta^{+}(x^{\prime})-\eta^{-}(x^{\prime})}.
\end{equation*}
Hence, for $x\in\Omega$,
\begin{equation*}
|\nabla\theta(x)|\leq\frac{c(1+L)}{\eta^{+}(x^{\prime})-\eta^{-}(x^{\prime})}%
\end{equation*}
for some constant $c=c(N,\delta)>0$. By the fundamental theorem of calculus, which can be applied thanks to Theorem \ref{theorem AC}, we have that
\begin{equation*}
u^{-}(x)    =f^{-}(x^{\prime})+\int_{\eta^{-}(x^{\prime})}^{x_{N}} \partial_{N}u^{-}(x^{\prime},s)\,ds
\end{equation*}
and 
\begin{equation*}
u^{+}(x)    =f^{+}(x^{\prime})-\int_{x_{N}}^{\eta^{+}(x^{\prime})}\partial_{N}u^{+}(x^{\prime},s)\,ds.
\end{equation*}
Hence,
\begin{equation*}
|u^{+}(x)-u^{-}(x)|    \leq|f^{+}(x^{\prime})-f^{-}(x^{\prime})|
+\int_{\eta^{-}(x^{\prime})}^{\eta^{+}(x^{\prime})}|\partial_{N}u^{-}(x^{\prime
},s)|\,ds
 +\int_{\eta^{-}(x^{\prime})}^{\eta^{+}(x^{\prime})}|\partial_{N} u^{+}(x^{\prime},s)|\,ds.
\end{equation*}
In turn, we may use H\"{o}lder's inequality to bound
\begin{align*}
|u^{+}(x)  &  -u^{-}(x)|^{p}\leq c|f^{+}(x^{\prime})-f^{-}(x^{\prime})|^{p}\\
&  +c(\eta^{+}(x^{\prime})-\eta^{-}(x^{\prime}))^{p-1}\int_{\eta^{-}(x^{\prime})}^{\eta^{+}(x^{\prime})}(|\partial_{N}u^{-}(x^{\prime},s)|^{p}+|\partial_{N}u^{+}(x^{\prime},s)|^{p})\,ds,
\end{align*}
and it follows that
\begin{align*}
& |\partial_{i}\theta(x)    (u^{+}(x)-u^{-}(x))|^{p}\leq c(1+L)^{p} \frac{|(u^{+}(x)-u^{-}(x))|^{p}}{(\eta^{+}(x^{\prime})-\eta^{-}(x^{\prime}))^{p}}\\
&  \leq c(1+L)^{p}\frac{|f^{+}(x^{\prime})-f^{-}(x^{\prime})|^{p}}{(\eta^{+}(x^{\prime})-\eta^{-}(x^{\prime}))^{p}}  +\frac{c(1+L)^{p}}{\eta^{+}(x^{\prime})-\eta^{-}(x^{\prime})} \int_{\eta^{-}(x^{\prime})}^{\eta^{+}(x^{\prime})}(|\partial_{N} u^{-}(x^{\prime},s)|^{p}+|\partial_{N}u^{+}(x^{\prime},s)|^{p})\,ds.
\end{align*}
Integrating both sides over $\Omega$ and using Tonelli's theorem gives%
\begin{align*}
\int_{\Omega}|\partial_{i}\theta(x)  &  (u^{+}(x)-u^{-}(x))|^{p}dx
\leq c(1+L)^{p}\int_{\mathbb{R}^{N-1}}\frac{|f^{+}(x^{\prime})-f^{-}(x^{\prime})|^{p}} {(\eta^{+}(x^{\prime})-\eta^{-}(x^{\prime}))^{p-1}}dx^{\prime}\\
&  +c(1+L)^{p}\int_{\Omega}(|\partial_{N}u^{-}(x)|^{p}+|\partial_{N} u^{+}(x)|^{p})\,dx.
\end{align*}
Thus $u \in \dot{W}^{1,p}(\Omega)$, and the desired estimate follows readily from this and the above analysis.

\end{proof}

\subsection{The case $m\geq2$, $p>1$}

Finally, in this subsection we consider the case $m \ge 2$ and $p >0$.   For $1\leq n\leq m-1$, we define
\begin{equation*} 
\operatorname*{Tr}\Bigl(\frac{\partial^{n}u}{\partial\nu^{n}}\Bigr)
:=\sum_{|\alpha|=n}\frac{1}{\alpha!}\operatorname*{Tr}(\partial^{\alpha} u) \nu^{\alpha}. 
\end{equation*}
We will study the linear mapping%
\begin{equation}
u\in W^{m,p}(\Omega)\mapsto\operatorname*{Tr}\nolimits_{m}%
(u):=\Bigl(\operatorname*{Tr}(u),\operatorname*{Tr}\Bigl(\frac{\partial
u}{\partial\nu}\Bigr),\cdots,\operatorname*{Tr}\Bigl(\frac{\partial^{m-1}%
u}{\partial\nu^{m-1}}\Bigr)\Bigr). \label{tr linear mapping}%
\end{equation}
\bigskip

We begin with the case $m=2$.   

\begin{theorem}\label{theorem trace m=2}
Let $\Omega$ be as in \eqref{omega two graphs}, where $\eta^{\pm}:\mathbb{R}^{N-1}\rightarrow\mathbb{R}$ are Lipschitz continuous functions, with $\eta^{-}<\eta^{+}$ and $L:=|\eta^{-}|_{0,1}+|\eta^{+}|_{0,1}$. Let $1<p<\infty$.  Then there exists a constant $c=c(N,p)>0$ such that for every $u\in\dot{W}^{2,p}\left(  \Omega\right)$ we have the estimates
\begin{equation*}
\sum_{i=1}^{N}  \int_{\mathbb{R}^{N-1}} \frac{|\operatorname*{Tr}(\partial_{i}u)(x^{\prime},\eta^{+}(x^{\prime})) - \operatorname*{Tr}(\partial_{i}u)(x^{\prime},\eta^{-}(x^{\prime}))|^{p}}{(\eta^{+}(x^{\prime})-\eta^{-}(x^{\prime}))^{p-1}} dx 
\leq c\int_{\Omega}|\nabla^{2}u(x)|^{p}dx, 
\end{equation*}
\begin{align*}
\int_{\mathbb{R}^{N-1}}
\frac{
\left\vert \sum\limits_{k=0}^1 \left(
  \operatorname*{Tr}(\partial_N^k u)(x^{\prime},\eta^{+}(x^{\prime})) 
+ (-1)^{k+1} \operatorname*{Tr}(\partial_N^k u)(x^{\prime},\eta^{-}(x^{\prime}))
\right)
\left(\frac{\eta^{+}(x^{\prime})-\eta^{-}(x^{\prime})}{2} \right)^k
\right\vert ^{p}
}
{(\eta^{+}(x^{\prime})-\eta^{-}(x^{\prime}))^{2p-1}}dx^{\prime}\\
\leq c\int_{\Omega}|\nabla^{2}u(x)|^{p}dx,
\end{align*}
and
\begin{align*}
\sum_{i=1}^{N} \int_{\mathbb{R}^{N-1}}\int_{B^{\prime}(0,(\eta^{+}(x^{\prime}) - \eta^{-}(x^{\prime}))/(2L))}\frac{|\operatorname*{Tr}(\partial_{i}u)(x^{\prime} + h^{\prime},\eta^{\pm}(x^{\prime}+h))-\operatorname*{Tr}(\partial_{i}u) (x^{\prime},\eta^{\pm}(x^{\prime}))|^{p}}{|h^{\prime}|^{p+N-2}
}\,dh^{\prime}dx^{\prime}\\
\leq(1+L)^{p}c\int_{\Omega}|\nabla^{2}u(x)|^{p}dx.
\end{align*}
\end{theorem}

\begin{proof}
Since $\partial_{i}u\in\dot{W}^{1,p}\left(  \Omega\right)$, the first and third inequalities follow by applying Theorem \ref{theorem trace m=1} to $\partial_{i}u$.  It remains to prove the second inequality.

For $\mathcal{L}^{N-1}$ a.e. $x' \in \mathbb{R}^{N-1}$ Theorem \ref{theorem AC} allows us to use the identity \eqref{formula by parts m=2} applied to the function  $u(x^{\prime},\cdot+\eta^{-}(x^{\prime}))$ with $b = \eta^{+}(x^{\prime}) - \eta^{-}(x^{\prime})$ in order to compute
\begin{align*}
u(x^{\prime},\eta^{+}(x^{\prime}))  &  -u(x^{\prime},\eta^{-}(x^{\prime
}))-(\partial_{N}u(x^{\prime},\eta^{+}(x^{\prime}))+\partial_{N}u(x^{\prime
},\eta^{-}(x^{\prime})))\frac{\eta^{+}(x^{\prime})-\eta^{-}(x^{\prime})}{2}\\
&  =\int_{0}^{\eta^{+}(x^{\prime})-\eta^{-}(x^{\prime})}\partial_{N}^{2}u(x^{\prime},t+\eta^{-}(x^{\prime}))\left(  \frac{\eta^{+}(x^{\prime})-\eta^{-}(x^{\prime})}{2}-t\right)  \,dt.
\end{align*}
From this we estimate
\begin{align*}
&  \left\vert u(x^{\prime},\eta^{+}(x^{\prime}))-u(x^{\prime},\eta^{-}(x^{\prime}))-(\partial_{N}u(x^{\prime},\eta^{+}(x^{\prime}))+\partial_{N}u(x^{\prime},\eta^{-}(x^{\prime})))\frac{\eta^{+}(x^{\prime})-\eta
^{-}(x^{\prime})}{2}\right\vert \\
&  \leq\frac{\eta^{+}(x^{\prime})-\eta^{-}(x^{\prime})}{2}\int_{0}^{\eta^{+}(x^{\prime})-\eta^{-}(x^{\prime})}|\partial_{N}^{2}u(x^{\prime},t+\eta^{-}(x^{\prime}))|\,dt=\frac{\eta^{+}(x^{\prime})-\eta^{-}(x^{\prime})}{2}%
\int_{\eta^{-}(x^{\prime})}^{\eta^{+}(x^{\prime})}|\partial_{N}^{2} u(x^{\prime},s)|\,ds.
\end{align*}
where in the second equality we have made the change of variables $s=t+\eta^{-}(x^{\prime})$. We then apply H\"{o}lder's inequality to see that
\begin{gather*}
\left\vert u(x^{\prime},\eta^{+}(x^{\prime}))-u(x^{\prime},\eta^{-}(x^{\prime}))-(\partial_{N}u(x^{\prime},\eta^{+}(x^{\prime}))+\partial_{N}u(x^{\prime},\eta^{-}(x^{\prime})))\frac{\eta^{+}(x^{\prime})-\eta^{-}(x^{\prime})}{2}\right\vert^{p}\\
\leq\frac{(\eta^{+}(x^{\prime})-\eta^{-}(x^{\prime}))^{2p-1}}{2}\int_{\eta^{-}(x^{\prime})}^{\eta^{+}(x^{\prime})}|\partial_{N}^{2}u(x^{\prime},s)|^{p}ds.
\end{gather*}
We then divide by  $(\eta^{+}(x^{\prime})-\eta^{-}(x^{\prime}))^{2p-1}$, integrate over $x^{\prime} \in \mathbb{R}^{N-1}$, and use Tonelli's theorem to arrive at the bound
\begin{gather*}
\int_{\mathbb{R}^{N-1}}\frac{\left\vert u(x^{\prime},\eta^{+}(x^{\prime} ) ) - u(x^{\prime},\eta^{-}(x^{\prime}))-(\partial_{N}u(x^{\prime},\eta^{+}(x^{\prime}))+\partial_{N}u(x^{\prime},\eta^{-}(x^{\prime})))\frac{\eta^{+}(x^{\prime}) - \eta^{-}(x^{\prime})}{2}\right\vert ^{p}}{( \eta^{+}(x^{\prime})-\eta^{-}(x^{\prime}))^{2p-1}}\,dx^{\prime}\\
\leq\int_{U_{R}}|\partial_{N}^{2}u(x)|^{p}dx,
\end{gather*}
which is the desired second estimate.
 \end{proof}

\begin{remark}\label{remark no lifting}
We have not been able to prove the extension of Theorem \ref{theorem lifting strip m=2} for domains of the form
\eqref{omega two graphs}, even assuming that $\eta^{\pm}$ are more regular.  If we follow the strategy of our proof of Theorem \ref{theorem lifting m=1} and consider first the case in which $\eta^-=0$ and $\eta= \eta^+$, then the desired function $v^{-}$ has the form (see the proof of Theorem \ref{theorem lifting strip m=2})
\begin{equation*}
v^{-}(x):=\int_{\mathbb{R}^{N-1}}[f_{0}^{-}(y^{\prime})+f_{1}^{-}(y^{\prime
})x_{N}]\varphi_{x_{N}}\left(  x^{\prime}-y^{\prime}\right)  \,dy^{\prime}.
\end{equation*}
However, to return to the general case \eqref{omega two graphs} one should define $u^- :\Omega \to \mathbb{R}$ via
\begin{equation*}
u^{-}(x):=v^{-}(x^{\prime},x_{N}-\eta^{-}(x^{\prime})).
\end{equation*}
Then for $i,j=1,\ldots,N-1$,
\begin{align*}
\frac{\partial^{2}u^{-}}{\partial x_{j}\partial x_{i}} &  (x) = \frac{\partial^{2}v^{-}}{\partial y_{j}\partial y_{i}}(x^{\prime},x_{N}-\eta^{-}(x^{\prime}))-\frac{\partial^{2}v^{-}}{\partial y_{N}\partial y_{i}%
}(x^{\prime},x_{N}-\eta^{-}(x^{\prime}))\frac{\partial\eta^{-}}{\partial x_{j}}(x^{\prime}) \\
&  -\frac{\partial^{2}v^{-}}{\partial y_{j}\partial y_{N}}(x^{\prime},x_{N}+\eta^{-}(x^{\prime}))\frac{\partial\eta^{-}}{\partial x_{i}}(x^{\prime})+\frac{\partial^{2}v^{-}}{\partial y_{N}^{2}}(x^{\prime},x_{N}-\eta^{-}(x^{\prime}))\frac{\partial\eta^{-}}{\partial x_{i}}(x^{\prime}) \frac{\partial\eta^{-}}{\partial x_{j}}(x^{\prime}) \\
&  -\frac{\partial v^{-}}{\partial y_{N}}(x^{\prime},x_{N}-\eta^{-}(x^{\prime})) \frac{\partial^{2}\eta^{-}}{\partial x_{j}\partial x_{i}}(x^{\prime}),
\end{align*}
while
\begin{align*}
\frac{\partial^{2}u^{-}}{\partial x_{j}\partial x_{N}}(x) &  = \frac{\partial^{2}v^{-}}{\partial y_{j}\partial y_{N}}(x^{\prime},x_{N}-\eta^{-}(x^{\prime}))-\frac{\partial^{2}v^{-}}{\partial y_{N}^{2}}(x^{\prime},x_{N}-\eta^{-}(x^{\prime}))\frac{\partial\eta^{-}}{\partial x_{j}}(x^{\prime}), \\
\frac{\partial^{2}u^{-}}{\partial x_{N}^{2}}(x) &  =\frac{\partial^{2}v^{-}}{\partial y_{N}^{2}}(x^{\prime},x_{N}-\eta^{-}(x^{\prime})). 
\end{align*}
These computations show that all the second derivatives of $v^{-}$ are in $L^{p}$, but unfortunately $\frac{\partial v^{-}}{\partial y_{N}}$ is not in general, so it is unclear if $u^- \in \dot{W}^{2,p}(\Omega)$. Other types of lifting arguments (see, for instance, Proposition 7.3 in \cite{mazya-mitrea-shaposhnikova2010}) seem to present the same pathology.
\end{remark}

Next we present the trace estimate in the general case $m\geq2$.

\begin{theorem} \label{theorem trace m>2}
Let $\Omega$ be as in \eqref{omega two graphs}, where $\eta^{\pm}:\mathbb{R}^{N-1}\rightarrow\mathbb{R}$ are Lipschitz continuous functions, with $\eta^{-}<\eta^{+}$ and $L:=|\eta^{+}|_{0,1}+|\eta^{-}|_{0,1}$.  Let $m\in\mathbb{N}$ with $m\geq 2$ and  $1<p<\infty$.  Then there exists a constant $c=c(m,N,p)>0$ such that for every $u\in\dot{W}^{m,p}\left(\Omega\right)$ and for every $i,l\in\mathbb{N}_{0}$ with $0\leq i+l\leq m-1$,
\begin{gather*}
\int_{\mathbb{R}^{N-1}}\frac{\left\vert \sum\limits_{k=0}^{m-l-i-1}\frac{(-1)^{k}}{k!}(\nabla_{\shortparallel}^{i}\operatorname*{Tr}(\partial_{N}^{k+l}u)(x^{\prime},\eta^{+}(x^{\prime}))+(-1)^{k+1}\nabla_{\shortparallel }^{i}\operatorname*{Tr}(\partial_{N}^{k+l}u)(x^{\prime},\eta^{-}(x^{\prime}))) \left(  \frac{\eta^{+}(x^{\prime})-\eta^{-}(x^{\prime})}{2}\right)^{k} \right\vert^{p}}{(\eta^{+}(x^{\prime})-\eta^{-}(x^{\prime}))^{(m-i-l)p-1}} dx^{\prime}\\
\leq c\int_{\Omega}|\nabla^{m}u(x)|^{p}dx,\\
\int_{\mathbb{R}^{N-1}}\int_{B^{\prime}(0,(\eta^{+}(x^{\prime})-\eta^{-}(x^{\prime}))/(2L))}\frac{|\operatorname*{Tr}(\nabla^{m-1}u)(x^{\prime} + h^{\prime},\eta^{\pm}(x^{\prime}+h^{\prime})) - \operatorname*{Tr}(\nabla^{m-1}u)(x^{\prime},\eta^{\pm}(x^{\prime}))|^{p}}{|h^{\prime}|^{p+N-2}} \,dh^{\prime}dx^{\prime}\\
\leq(1+L)^{p}c\int_{\Omega}|\nabla^{m}u(x)|^{p}dx.
\end{gather*}

\end{theorem}

\begin{proof}
The proof is very similar to the that of Theorem \ref{theorem trace m=2}, using \eqref{formula by parts} instead of \eqref{formula by parts m=2}. We omit the details for the sake of brevity.
\end{proof}

\section{Applications}\label{section applications}

In this section we present some applications of the previous results to quasilinear elliptic PDEs in domains $\Omega$ of the form \eqref{omega two graphs} or \eqref{omega infinite strip}.

\subsection{Lagrangians}

We record here some properties of Lagrangians that will be used to define our quasilinear PDEs.  We begin with a notion of admissibility.

\begin{definition}\label{def admissible lagrangian}
Let $\Omega \subseteq \mathbb{R}^N$ be open and $1 < p < \infty$.  We say that a function $G: \Omega \times \mathbb{R}^N \to \mathbb{R}$ is a $p-$admissible Lagrangian if the following conditions are satisfied.
\begin{enumerate}
 \item For each $\xi \in \mathbb{R}^N$ the function $G(\cdot,\xi)$ is measurable on $\Omega$, and for almost every $x \in \Omega$ the function $G(x,\cdot)$ is continuously differentiable on $\mathbb{R}^N$.
 \item For almost every $x \in \Omega$ the function $G(x,\cdot)$ is convex on $\mathbb{R}^N$.
 \item There exists a constant $A_- \in (0,\infty)$ and a function $\psi_- \in L^1(\Omega)$ such that 
\begin{equation*}
 A_- |\xi|^p - \psi_-(x) \le G(x,\xi) \text{ for almost every } x  \in \Omega \text{ and every } \xi \in \mathbb{R}^N.
\end{equation*}
 \item  There exists a constant $A_+ \in (0,\infty)$ and a function $\psi_+ \in L^{p'}(\Omega)$ such that 
\begin{equation*}
 |\nabla_\xi G(x,\xi) | \le \psi_+(x) + A_+ |\xi|^{p-1} \text{ for almost every } x  \in \Omega \text{ and every } \xi \in \mathbb{R}^N.
\end{equation*}
 \item $G(\cdot,0) \in L^1(\Omega)$.
\end{enumerate}
\end{definition}

The next lemma records an upper bound that follows from the assumptions in Definition \ref{def admissible lagrangian}.

\begin{lemma}\label{lemma admissible bound}
Let $\Omega \subseteq \mathbb{R}^N$ be open and $1 < p < \infty$.  If $G: \Omega \times \mathbb{R}^N \to \mathbb{R}$ is a $p-$admissible Lagrangian in the sense of Definition \ref{def admissible lagrangian}, then 
\begin{equation*}
|G(x,\xi)| \le |G(x,0)| + \frac{|\psi_+(x)|^{p'}}{p'}  + \frac{(1+A_+)}{p}|\xi|^p
\end{equation*}
for almost every $x \in \Omega$ and each $\xi \in \mathbb{R}^N$.
\end{lemma}
\begin{proof}
The fundamental theorem of calculus and the fourth admissibility condition for $G$ allow us to estimate 
\begin{equation*}
 |G(x,\xi)| \le |G(x,0)| + \int_0^1 |\xi| |\nabla_\xi G(x,t\xi)| dt  \le |G(x,0)| + \int_0^1 [|\xi| |\psi_+(x)| + A_+ t^{p-1} |\xi|^p] dt 
\end{equation*}
The stated estimate follows easily from this and Young's inequality.

\end{proof}

\subsection{The Dirichlet problem}

We now turn our attention to the Dirichlet problem associated to $p-$admissible Lagrangians.  Our main result establishes necessary and sufficient conditions for the solvability of the associated Dirichlet problem.

\begin{theorem}\label{theorem dirichlet G}
Let $\Omega \subset \mathbb{R}^N$ be as in \eqref{omega infinite strip}, $a>0$, $1<p<\infty$, and $f^{\pm}\in L_{\operatorname*{loc}}^{1}(\mathbb{R}^{N-1})$.  Suppose that $G: \Omega \times \mathbb{R}^N \to \mathbb{R}$ is a $p-$admissible Lagrangian in the sense of Definition \ref{def admissible lagrangian}.  Then the quasilinear Dirichlet problem
\begin{equation}\label{dirichlet G}
\left\{
\begin{array}
[c]{ll}
-\operatorname{div}(\nabla_\xi G(\cdot,\nabla u)  ) = 0 & \text{in
}\Omega,\\
u=f^{\pm} & \text{on }\Gamma^{\pm}
\end{array}
\right.  
\end{equation}
admits a solution in $\dot{W}^{1,p}(\Omega)$ if and only if $f^{\pm}$ satisfy \eqref{difference f plus of minus} and \eqref{seminorm strip}.  In either case, there exists a constant $c=c(a,N,p,A_-,A_+)>0$ such that 
\begin{align}\label{dirichlet G bounds}
\begin{split}
\int_{\Omega}|\nabla u(x)|^{p}dx  &  \leq c\int_{\Omega} [|G(x,0)| + |\psi_-(x)| +    |\psi_+(x)|^{p^{\prime}}] dx + c\int_{\mathbb{R}^{N-1}}|f^{+}(x^{\prime})-f^{-}(x^{\prime})|^{p}\, dx^{\prime}\\
&  \quad+c|f^{-}|_{\W_{(a)}^{1-1/p,p}(\mathbb{R}^{N-1})}^{p} 
+ c|f^{+}|_{\W_{(a)}^{1-1/p,p}(\mathbb{R}^{N-1})}^{p},
\end{split}
\end{align}
where $\psi_\pm$ are as in Definition \ref{def admissible lagrangian}, and $u$ minimizes the energy functional  $F : \{v\in\dot{W}^{1,p}(\Omega)\,|\,\operatorname*{Tr}(v)=f^{\pm} \text{ on }\Gamma^{\pm}\} \to \mathbb{R}$ given by 
\begin{equation}\label{dirichlet energy}
 F(v) = \int_{\Omega} G(x,\nabla v(x))  dx.
\end{equation}

\end{theorem}
\begin{proof}
The proof is a standard application of the direct method of the calculus of variations. We present it for the convenience of the reader.

Assume that $f^{\pm}$ satisfy \eqref{difference f plus of minus} and \eqref{seminorm strip}.  Let $X:=\{v\in\dot{W}^{1,p}(\Omega)\,|\,\operatorname*{Tr}(v)=f^{\pm} \text{ on }\Gamma^{\pm}\}$ and note that $X$ is nonempty thanks to Theorem \ref{theorem lifting strip}. Consider the functional $F: X \to \mathbb{R}$ given by \eqref{dirichlet energy}, which is well-defined in light of Lemma \ref{lemma admissible bound}.  Since $G$ is $p-$admissible, we can bound
\begin{equation*}
F(v)\geq \int_\Omega A_- |\nabla v(x)|^p dx - \int_\Omega  \psi_-(x)  dx.
\end{equation*}
It follows that
\begin{equation*}
\ell:=\inf_{v\in X}F(v)>-\infty
\end{equation*}
and that if $\{u_{n}\}_{n\in \mathbb{N}}$ is a minimizing sequence in $X$, i.e. 
\begin{equation*}
\lim_{n\rightarrow\infty}F(u_{n})=\ell,
\end{equation*}
then $\{\nabla u_{n}\}_{n \in \mathbb{N}}$ is bounded in $L^{p}(\Omega;\mathbb{R}^{N})$. By
applying Poincar\'{e}'s inequality on an increasing sequence of bounded Lipschitz domains and employing a diagonalization argument, we can find a subsequence, not relabeled, and $u\in\dot{W}^{1,p}(\Omega)$ such that $\nabla u_{n}\rightharpoonup\nabla u$ in $L^{p}(\Omega;\mathbb{R}^{N})$.  The $p-$admissibility conditions guarantee that $F$ is sequentially weakly lower semicontinuous (see, for instance, Theorem 6.54 of \cite{fonseca-leoni2007}), so 
\begin{equation*}
F(u)\leq\lim_{n\rightarrow\infty}F(u_{n})=\ell.
\end{equation*}
On the other hand, by applying the Rellich--Kondrachov theorem on an increasing sequence of bounded Lipschitz domains $\Omega_{k}$, we can assume that $u_{n}\rightarrow u$ in $L^{p}(\Omega_{k})$ as $n\rightarrow\infty$ for every $k$. Hence, $u_{n}\rightharpoonup u$ in $W^{1,p}(\Omega_{k})$, which
implies that $\operatorname*{Tr}(u)=f^{\pm}$ on $\Gamma^{\pm}\cap \partial\Omega_{k}$. Since this is true for every $k$, we have that $u\in X$ and so $F(u)=\ell$.

Since $u$ is a minimizer of $F$ we may take variations to see that $u$ satisfies
\begin{equation*}
\int_{\Omega}\nabla_\xi G (x,\nabla u)\cdot \nabla v\,dx=0
\end{equation*}
for every $v\in\dot{W}^{1,p}(\Omega)$ with $\operatorname*{Tr}(v)=0$ on $\Gamma^{\pm}$. This shows that $u$ is a weak solution of \eqref{dirichlet G}.

Conversely, if $u\in\dot{W}^{1,p}(\Omega)$ is a weak solution of \eqref{dirichlet G}, then since $\operatorname*{Tr}(u)=f^{\pm}$ on $\Gamma^{\pm}$, it follows from Theorem \ref{theorem trace strip} that $f^{\pm}$ satisfies conditions \eqref{difference f plus of minus} and \eqref{seminorm strip}.

It remains to prove \eqref{dirichlet G bounds} when either (and hence both) of the equivalent conditions is satisfied.  By Theorem \ref{theorem lifting strip} there exists $w\in\dot{W}^{1,p}(\Omega)$ with $\operatorname*{Tr}(w)=f^{\pm}$ on $\Gamma^{\pm}$ and
\begin{equation*}
\Vert\nabla w\Vert_{L^{p}(\Omega)}\leq c|(f^{-},f^{+})|_{\dot{X}%
_{(a)}^{1-1/p,p}(\mathbb{R}^{N-1})},
\end{equation*}
where we recall that $\dot{X}_{(a)}^{1-1/p,p}(\mathbb{R}^{N-1})$ is the space given in Definition \ref{definition space X}, with $\sigma = a$.  Then $w \in X$, and so by minimality, $p-$admissibility, and Lemma \ref{lemma admissible bound} we have that 
\begin{equation*}
\int_\Omega [A_- |\nabla u(x)|^p  -  \psi_-(x) ] dx  \le F(u) \le F(w) 
 \le \int_\Omega [|G(x,0)| + |\psi_+(x)|^p +(1+A_+) |\nabla w(x)|^p ] dx
\end{equation*}
Then \eqref{dirichlet G bounds} follows immediately by combining these bounds. 
\end{proof}

\begin{remark}
Theorem \ref{theorem dirichlet G} continues to hold for domains of the type \eqref{omega two graphs}, provided we replace \eqref{difference f plus of minus} and \eqref{seminorm strip} with \eqref{difference f plus of minus two graphs} and \eqref{seminorn two graphs}. The proof remains unchanged.
\end{remark}

We can now present the proof of Theorem \ref{theorem dirichlet laplacian}.

\begin{proof}[Proof of Theorem \ref{theorem dirichlet laplacian}]
We simply observe that the map $G: \Omega \times \mathbb{R}^N \to \mathbb{R}$ given by $G(x,\xi) = |\xi|^p/p + g(x) \cdot \xi$ satisfies the conditions of Definition \ref{def admissible lagrangian}, and so we may apply Theorem \ref{theorem dirichlet G}.
\end{proof}

\subsection{The Neumann problem}

We now turn our attention to the Neumann problem associated to an admissible Lagrangian $G$, i.e. the problem
\begin{equation*}
\left\{
\begin{array}
[c]{ll}%
-\operatorname{div}(\nabla_\xi G(\cdot,\nabla u)  ) = \psi & \text{in }\Omega,\\
\nabla_\xi G(\cdot ,\nabla u) \cdot\nu=h & \text{on }\partial\Omega.
\end{array}
\right.
\end{equation*}
Assuming for the moment that $u$, $\psi$, and $h$ are sufficiently smooth and have the right decaying properties, by multiplying the equation by $v \in C_{c}^{\infty}(\mathbb{R}^{N})$ and integrating by parts we get
\begin{equation*}
\int_{\Omega}\nabla_\xi G(x,\nabla u(x)) \cdot\nabla v(x) \,dx
=\int_{\Omega} v(x)\psi(x) \,dx + \int_{\partial\Omega}v(x)h(x) \,d\mathcal{H}^{N-1}(x).
\end{equation*}
More generally, we may replace the linear functionals
\begin{equation*}
v\mapsto\int_{\Omega}v\psi\,dx \text{ and }
v\mapsto\int_{\partial\Omega }vh\,d\mathcal{H}^{N-1}
\end{equation*}
with generic linear functionals  
\begin{equation}\label{functional Lambda}
\Psi:\dot{W}^{1,p}(\Omega)\rightarrow\mathbb{R} \text{ and }
\Lambda : \dot{X}_{(1)}^{1-1/p,p}(\mathbb{R}^{N-1})  = \operatorname*{Tr}(\dot{W}^{1,p}(\Omega))\rightarrow\mathbb{R}.
\end{equation}
Thus a weak solution $u\in\dot{W}^{1,p}(\Omega)$ of this generalized problem must satisfy
\begin{equation} \label{weak solution Neumann}
\int_{\Omega} \nabla_\xi G(x,\nabla u(x)) \cdot\nabla v(x) \,dx=\Psi(v)+\Lambda
(\operatorname*{Tr}(v))
\end{equation}
for all $v\in\dot{W}^{1,p}(\Omega)$. Note that since a constant function $v\equiv c$ belongs to $\dot{W}^{1,p}(\Omega)$, from \eqref{weak solution Neumann} we get the compatibility condition
\begin{equation} \label{compatibility conditions neumann}
\Psi(c)+\Lambda(c)=0
\end{equation}
for every $c\in\mathbb{R}$. Thus, if we define $\Lambda_1 : \dot{W}^{1,p}(\Omega) \to \mathbb{R}$ via 
\begin{equation}\label{functional Lambda 1}
\Lambda_{1}(v):=\Lambda(\operatorname*{Tr}(v)), 
\end{equation}
then
\begin{equation*}
\Psi(v+c)+\Lambda_{1}(v+c)=\Psi(v)+\Lambda_{1}(v)
\end{equation*}
for every $v\in\dot{W}^{1,p}(\Omega)$ and every $c\in\mathbb{R}$, and so we can define $\Psi+\Lambda_{1}$ on the quotient space $\dot{W}^{1,p}(\Omega)/\mathbb{R}.$

Given $v\in\dot{W}^{1,p}(\Omega)$, we set
\begin{equation*}
\operatorname*{Tr}\nolimits^{\pm}(v):=\left.  \operatorname*{Tr}(v)\right\vert_{\Gamma^{\pm}}.
\end{equation*}
In view of Theorem \ref{theorem trace strip}, the pair $(\operatorname*{Tr} \nolimits^{-}(v),\operatorname*{Tr}\nolimits^{+}(v))$ belongs to $\dot{X}_{(a)}^{1-1/p,p}(\mathbb{R}^{N-1})$, where $\dot{X}_{(a)}^{1-1/p,p}(\mathbb{R}^{N-1})$ is the space given in Definition \ref{definition space X} with $\sigma= a$ for $0<a\leq b$.  In what follows, with a slight abuse of notation, we identify
$\operatorname*{Tr}(v)$ with $(\operatorname*{Tr}\nolimits^{-}(v),\operatorname*{Tr}\nolimits^{+}(v))$. Conversely, given $(f^{-},f^{+})\in\dot{X}_{(a)}^{1-1/p,p}(\mathbb{R}^{N-1})$ by Theorem
\ref{theorem lifting strip} we have that there is $v\in\dot{W}^{1,p}(\Omega)$ such that $\operatorname*{Tr}\nolimits^{\pm}(f)=f^{\pm}$. If we set
\begin{equation}\label{function f}
f(x^{\prime},x_{N}):=\left\{
\begin{array}
[c]{ll}%
f^{-}(x^{\prime}) & \text{if }x^{\prime}\in\mathbb{R}^{N-1},\,x_{N}=b^{-},\\
f^{+}(x^{\prime}) & \text{if }x^{\prime}\in\mathbb{R}^{N-1},\,x_{N}=b^{+},
\end{array}
\right.  
\end{equation}
then $f=\operatorname*{Tr}(v)$.

We can now state our main result on the Neumann problem.

\begin{theorem}\label{theorem neumann laplacian}
Let $\Omega\subset \mathbb{R}^n$ be as in \eqref{omega infinite strip}, $1<p<\infty$, and assume that $\Psi:\dot{W}^{1,p}(\Omega)\rightarrow\mathbb{R}$ and $\Lambda:\operatorname*{Tr}(\dot
{W}^{1,p}(\Omega))\rightarrow\mathbb{R}$ are linear functionals satisfying \eqref{compatibility conditions neumann}.  Suppose that $G: \Omega \times \mathbb{R}^N \to \mathbb{R}$ is a $p-$admissible Lagrangian in the sense of Definition \ref{def admissible lagrangian}.  Then the quasilinear Neumann problem
\begin{equation} \label{neumann problem}
\left\{
\begin{array}
[c]{ll}%
-\operatorname{div}(\nabla_\xi G(\cdot,\nabla u) )=\Psi & \text{in }\Omega,\\
\nabla_\xi G(\cdot,\nabla u)\cdot\nu=\Lambda & \text{on }\partial\Omega
\end{array}
\right. 
\end{equation}
admits a weak solution in $\dot{W}^{1,p}(\Omega)$ if and only if there exists $c>0$ such that
\begin{equation}\label{functionals bounded}
|\Psi(v)+\Lambda_{1}(v)|\leq c\Vert\nabla v\Vert_{L^{p}(\Omega)}
\end{equation}
for every $v\in\dot{W}^{1,p}(\Omega)$, where $\Lambda_{1}$ is given in \eqref{functional Lambda 1}. In either case, there exists a constant $c=c(a,N,p,A_-)$ such that
\begin{equation} \label{neumann bounds}
\int_{\Omega}|\nabla u(x)|^{p}dx   \leq c\int_{\Omega} [|G(x,0)| + |\psi_-(x)| ] dx +  c\Vert\Psi+\Lambda_{1}\Vert_{(\dot{W}^{1,p}(\Omega))\prime}^{p^{\prime}}, 
\end{equation}
and $u$ minimizes of the energy functional $F: \dot{W}^{1,p}(\Omega) \to \mathbb{R}$ given by 
\begin{equation}\label{neumann energy}
F(v) = \int_{\Omega}G(x,\nabla v(x)) dx-\Psi(v)-\Lambda_{1}(v).
\end{equation}

\end{theorem}

\begin{proof}
Assume that $\Psi$ and $\Lambda$ satisfy \eqref{functionals bounded}. Consider
the functional \eqref{neumann energy} defined for $v\in\dot{W}^{1,p}(\Omega)$. By \eqref{functionals bounded} and the $p$-admissibility conditions, 
\begin{equation*}
F(v)\geq   \int_\Omega A_- |\nabla v(x)|^p dx - \int_\Omega  \psi_-(x)  dx -c\Vert\nabla
v\Vert_{L^{p}(\Omega)}.
\end{equation*}
We can now proceed as in the proof of Theorem \ref{theorem dirichlet G} to show that there exists a minimizer $u\in\dot{W}^{1,p}(\Omega)$ of $F$. Using the continuity and linearity of $\Lambda_{1}$, taking variations we get that \eqref{weak solution Neumann} holds for every $v\in\dot{W}^{1,p}(\Omega)$, which shows that $u$ is a weak solution of \eqref{neumann problem}.

Conversely, assume that $u\in\dot{W}^{1,p}(\Omega)$ satisfies \eqref{weak solution Neumann}. Since $(p-1)p^{\prime}=p$, it follows from the $p-$admissibility conditions and H\"older's inequality that
\begin{equation*}
|\Psi(v)+\Lambda_{1}(v)| \le \int_\Omega (|\psi_+(x)| + A_+ |\nabla u(x)|^{p-1} )|\nabla v(x)| dx \le c(\psi_+,u,A_+)  \Vert\nabla v\Vert_{L^{p}(\Omega)}
\end{equation*}
for every $v\in\dot{W}^{1,p}(\Omega)$. 

To conclude the proof it remains to show \eqref{neumann bounds}.  Using the $p-$admissibility conditions, we may bound
\begin{equation*}
\int_\Omega [A_- |\nabla u(x)|^p  -  \psi_-(x) ] dx  -  (\Psi(u) + \Lambda_1(u) ) \le F(u) \le F(0) 
 \le \int_\Omega |G(x,0)| dx.
\end{equation*}
Then \eqref{neumann bounds} follows directly from this and  \eqref{functionals bounded}.
\end{proof}

We next record a corollary about the Neumann problem with $\Psi =0$.

\begin{corollary}
Let $\Omega\subset \mathbb{R}^n$ be as in \eqref{omega infinite strip}, $0<a\leq b^{+}-b^{-}$, $1<p<\infty$, and assume that  $\Lambda:\operatorname*{Tr}(\dot{W}^{1,p}(\Omega))\rightarrow\mathbb{R}$ is a linear functionals satisfying $\Lambda(1) =0$.   Suppose that $G: \Omega \times \mathbb{R}^N \to \mathbb{R}$ is a $p-$admissible Lagrangian in the sense of Definition \ref{def admissible lagrangian}. Then the quasilinear Neumann problem
\begin{equation}\label{neumann problem zero}
\left\{
\begin{array}
[c]{ll}%
-\operatorname{div}(\nabla_\xi G(\cdot,\nabla u))=0 & \text{in }\Omega,\\
\nabla_\xi G(\cdot,\nabla u) \cdot\nu=\Lambda & \text{on }\partial\Omega
\end{array}
\right.  
\end{equation}
admits a weak solution in $\dot{W}^{1,p}(\Omega)$ if and only if there exists
$c>0$ such that
\begin{equation}\label{Lambda bounded}
|\Lambda(f)|\leq c|(f^{-},f^{+})|_{\dot{X}_{(a)}^{1-1/p,p}(\mathbb{R}^{N-1})}
\end{equation}
for every $(f^{-},f^{+})\in\dot{X}_{(a)}^{1-1/p,p}(\mathbb{R}^{N-1})$, where $f$ is given in \eqref{function f}. In either case, there exists a constant $c=c(a,N,p,A_-)$ such that
\begin{equation}\label{neumann bounds zero}
\Vert\nabla u\Vert_{L^{p}(\Omega)}^p \leq c\int_{\Omega} [|G(x,0)| + |\psi_-(x)| ] dx + c|\Lambda|_{(\dot{X}_{(a)}^{1-1/p,p}(\mathbb{R}^{N-1}))^{\prime}}^{p^{\prime}},
\end{equation}
where $|\Lambda|_{(\dot{X}_{(a)}^{1-1/p,p}(\mathbb{R}^{N-1}))^{\prime}}$ is defined in \eqref{seminorm T}.
\end{corollary}

\begin{proof}
Since $\Psi=0$  Theorem \ref{theorem neumann laplacian} tells us that \eqref{neumann problem zero} admits a weak solution $u \in \dot{W}^{1,p}(\Omega)$ if and only if
\begin{equation}\label{Lamda1 bounded}
|\Lambda_{1}(v)|\leq c\Vert\nabla v\Vert_{L^{p}(\Omega)}
\end{equation}
for every $v\in\dot{W}^{1,p}(\Omega)$. Assume that \eqref{Lambda bounded}
holds. By Theorem \ref{theorem trace strip}, given $v\in\dot{W}^{1,p}(\Omega
)$, we have that $(\operatorname*{Tr}\nolimits^{-}(v),\operatorname*{Tr}%
\nolimits^{+}(v))\in\dot{X}_{(a)}^{1-1/p,p}(\mathbb{R}^{N-1})$ with
$|(\operatorname*{Tr}\nolimits^{-}(v),\operatorname*{Tr}\nolimits^{+}%
(v))|_{\dot{X}_{(a)}^{1-1/p,p}(\mathbb{R}^{N-1})}\leq c\Vert\nabla
v\Vert_{L^{p}(\Omega)}$, and so by \eqref{Lambda bounded},
\begin{equation*}
|\Lambda_{1}(v)|=\left\vert \Lambda(\operatorname*{Tr}(v))\right\vert \leq
c|(\operatorname*{Tr}\nolimits^{-}(v),\operatorname*{Tr}\nolimits^{+}%
(v))|_{\dot{X}_{(a)}^{1-1/p,p}(\mathbb{R}^{N-1})}\leq c\Vert\nabla
v\Vert_{L^{p}(\Omega)}.
\end{equation*}
This shows that the functional $\Lambda_{1}$ satisfies \eqref{Lamda1 bounded}.
Conversely, assume that \eqref{Lamda1 bounded} holds. Given $(f^{-},f^{+}%
)\in\dot{X}_{(a)}^{1-1/p,p}(\mathbb{R}^{N-1})$ by Theorem
\ref{theorem lifting strip} we can find $v\in\dot{W}^{1,p}(\Omega)$ such that
$\operatorname*{Tr}(v)=f$ and
\begin{equation*}
\Vert\nabla v\Vert_{L^{p}(\Omega)}\leq c|(f^{-},f^{+})|_{\dot{X}%
_{(a)}^{1-1/p,p}(\mathbb{R}^{N-1})}.
\end{equation*}
It follows from \eqref{Lamda1 bounded} that
\begin{equation*}
\left\vert \Lambda(f)\right\vert =|\Lambda_{1}(v)|\leq c\Vert\nabla
w\Vert_{L^{p}(\Omega)}\leq c|(f^{-},f^{+})|_{\dot{X}_{(a)}^{1-1/p,p}%
(\mathbb{R}^{N-1})},
\end{equation*}
which shows \eqref{Lambda bounded}.

To conclude the proof it remains to show \eqref{neumann bounds zero}. This follows from \eqref{neumann bounds} once we observe that
\begin{align*}
| \Lambda(\operatorname*{Tr}(u))| & \leq|\Lambda|_{(\dot{X}_{(a)}^{1-1/p,p}(\mathbb{R}^{N-1}))^{\prime}%
}|(\operatorname*{Tr}\nolimits^{-}(u),\operatorname*{Tr}\nolimits^{+}%
(u))|_{\dot{X}_{(a)}^{1-1/p,p}(\mathbb{R}^{N-1})}\\
&  \leq c|\Lambda|_{(\dot{X}_{(a)}^{1-1/p,p}(\mathbb{R}^{N-1}))^{\prime}}%
\Vert\nabla u\Vert_{L^{p}(\Omega)},
\end{align*}
where the last inequality follows from Theorem \ref{theorem trace strip}.
\end{proof}

\begin{remark}
In view of Theorems \ref{theorem trace m=1} and \ref{theorem lifting m=1}, the
previous corollary continues to hold for domains of the type \ref{omega two graphs}, provided we replace $\dot{X}_{(a)}^{1-1/p,p}(\mathbb{R}^{N-1})$ with $\dot{X}_{(\eta)}^{1-1/p,p}(\mathbb{R}^{N-1})$, where
$\eta:=(\eta^{+}-\eta^{-})/(2L)$. The proof remains unchanged.
\end{remark}

\bibliographystyle{abbrv}
\bibliography{traces-references}

\ifx \cedla \undefined \let \cedla = \c \fi \ifx \cyr \undefined \let \cyr =
  \relax \fi \ifx \cprime \undefined \def \cprime {$\mathsurround=0pt '$}\fi
  \ifx \prime \undefined \def \prime {'} \fi
\begin{thebibliography}{10}

\bibitem{adams-fournier2003book}
R.~A. Adams and J.~J.~F. Fournier.
\newblock {\em {Sobolev Spaces}}, volume 140 of {\em {Pure and Applied
  Mathematics}}.
\newblock Elsevier/Academic Press, Amsterdam, second edition, 2003.

\bibitem{benci-fortunato1979}
V.~Benci and D.~Fortunato.
\newblock {Weighted {S}obolev spaces and the nonlinear {D}irichlet problem in
  unbounded domains}.
\newblock {\em Ann. Mat. Pura Appl. (4)}, 121:319--336, 1979.

\bibitem{bennett-sharpley1988book}
C.~Bennett and R.~Sharpley.
\newblock {\em {Interpolation of Operators}}, volume 129 of {\em {Pure and
  Applied Mathematics}}.
\newblock Academic Press, Inc., Boston, MA, 1988.

\bibitem{berger-schechter1972}
M.~S. Berger and M.~Schechter.
\newblock {Embedding theorems and quasi-linear elliptic boundary value problems
  for unbounded domains}.
\newblock {\em Trans. Amer. Math. Soc.}, 172:261--278, 1972.

\bibitem{bergh-lofstrom-1976book}
J.~Bergh and J.~L{\"o}fstr{\"o}m.
\newblock {\em {Interpolation Spaces. An Introduction}}.
\newblock Springer-Verlag, Berlin-New York, 1976.
\newblock Grundlehren der Mathematischen Wissenschaften, No. 223.

\bibitem{besov-ilin-nikolskii1979book}
O.~V. Besov, V.~P. {Il\cprime in}, and S.~M. {Nikol\cprime ski\u\i}.
\newblock {\em {Integral Representations of Functions and Imbedding Theorems.
  {V}ol. {I}}}.
\newblock V. H. Winston \&\ Sons, Washington, D.C.; Halsted Press [John Wiley
  \&\ Sons], New York-Toronto, Ont.-London, 1978.
\newblock Translated from the Russian, Scripta Series in Mathematics, Edited by
  Mitchell H. Taibleson.

\bibitem{besov-ilin-nikolskii1978book}
O.~V. Besov, V.~P. {Il\cprime in}, and S.~M. {Nikol\cprime ski\u\i}.
\newblock {\em {Integral Representations of Functions and Imbedding Theorems.
  {V}ol. {II}}}.
\newblock V. H. Winston \&\ Sons, Washington, D.C.; Halsted Press [John Wiley
  \&\ Sons], New York-Toronto, Ont.-London, 1979.
\newblock Scripta Series in Mathematics, Edited by Mitchell H. Taibleson.

\bibitem{bourgain-brezis-mironescu2001}
J.~Bourgain, H.~Brezis, and P.~Mironescu.
\newblock Another look at {S}obolev spaces.
\newblock In {\em Optimal control and partial differential equations}, pages
  439--455. IOS, Amsterdam, 2001.

\bibitem{bourgain-brezis-mironescu2002}
J.~Bourgain, H.~Brezis, and P.~Mironescu.
\newblock Limiting embedding theorems for {$W^{s,p}$} when {$s\uparrow1$} and
  applications.
\newblock {\em J. Anal. Math.}, 87:77--101, 2002.
\newblock Dedicated to the memory of Thomas H. Wolff.

\bibitem{brezis_functional}
H.~Brezis.
\newblock {\em {Functional Analysis, Sobolev Spaces and Partial Differential
  Equations}}.
\newblock {Universitext}. Springer, New York, 2011.

\bibitem{brezis-nguyen2018}
H.~Brezis and H.-M. Nguyen.
\newblock Non-local functionals related to the total variation and connections
  with image processing.
\newblock {\em Ann. PDE}, 4(1):Art. 9, 77, 2018.

\bibitem{burenkov1998book}
V.~I. Burenkov.
\newblock {\em {Sobolev Spaces on Domains}}, volume 137 of {\em {Teubner-Texte
  zur Mathematik [Teubner Texts in Mathematics]}}.
\newblock B. G. Teubner Verlagsgesellschaft mbH, Stuttgart, 1998.

\bibitem{dinezza-palatucci-valdinoci2012}
E.~Di~Nezza, G.~Palatucci, and E.~Valdinoci.
\newblock Hitchhiker's guide to the fractional {S}obolev spaces.
\newblock {\em Bull. Sci. Math.}, 136(5):521--573, 2012.

\bibitem{du_etal2012}
Q.~Du, M.~Gunzburger, R.~B. Lehoucq, and K.~Zhou.
\newblock {Analysis and approximation of nonlocal diffusion problems with
  volume constraints}.
\newblock {\em SIAM Rev.}, 54(4):667--696, 2012.

\bibitem{edmunds-evans1973}
D.~E. Edmunds and W.~D. Evans.
\newblock {Elliptic and degenerate-elliptic operators in unbounded domains}.
\newblock {\em Ann. Scuola Norm. Sup. Pisa (3)}, 27:591--640 (1974), 1973.

\bibitem{felsinger-mathieu-kassman2015}
M.~Felsinger, M.~Kassmann, and P.~Voigt.
\newblock {The {D}irichlet problem for nonlocal operators}.
\newblock {\em Math. Z.}, 279(3-4):779--809, 2015.

\bibitem{fonseca-leoni2007}
I.~Fonseca and G.~Leoni.
\newblock {\em {Modern Methods in the Calculus of Variations: {$L^p$} Spaces}}.
\newblock {Springer Monographs in Mathematics}. Springer, New York, 2007.

\bibitem{gagliardo1957}
E.~Gagliardo.
\newblock {Caratterizzazioni delle tracce sulla frontiera relative ad alcune
  classi di funzioni in {$n$} variabili}.
\newblock {\em Rend. Sem. Mat. Univ. Padova}, 27:284--305, 1957.

\bibitem{galdi-simader1990}
G.~P. Galdi and C.~G. Simader.
\newblock {Existence, uniqueness and {$L^q$}-estimates for the {S}tokes problem
  in an exterior domain}.
\newblock {\em Arch. Rational Mech. Anal.}, 112(4):291--318, 1990.

\bibitem{grafakos2014book}
L.~Grafakos.
\newblock {\em {Modern Fourier Analysis}}, volume 250 of {\em {Graduate Texts
  in Mathematics}}.
\newblock Springer, New York, third edition, 2014.

\bibitem{grisvard2011book}
P.~Grisvard.
\newblock {\em {Elliptic Problems in Nonsmooth Domains}}, volume~69 of {\em
  {Classics in Applied Mathematics}}.
\newblock Society for Industrial and Applied Mathematics (SIAM), Philadelphia,
  PA, 2011.
\newblock Reprint of the 1985 original.

\bibitem{kufner1985book}
A.~Kufner.
\newblock {\em {Weighted {S}obolev Spaces}}.
\newblock {A Wiley-Interscience Publication}. John Wiley \& Sons, Inc., New
  York, 1985.
\newblock Translated from the Czech.

\bibitem{leoni2017book}
G.~Leoni.
\newblock {\em {A First Course in {S}obolev Spaces}}, volume 181 of {\em
  {Graduate Studies in Mathematics}}.
\newblock American Mathematical Society, Providence, RI, second edition, 2017.

\bibitem{mazya2011book}
V.~Maz'ya.
\newblock {\em {Sobolev Spaces with Applications to Elliptic Partial
  Differential Equations}}, volume 342 of {\em {Grundlehren der Mathematischen
  Wissenschaften [Fundamental Principles of Mathematical Sciences]}}.
\newblock Springer, Heidelberg, augmented edition, 2011.

\bibitem{mazya-mitrea-shaposhnikova2010}
V.~Maz'ya, M.~Mitrea, and T.~Shaposhnikova.
\newblock {The {D}irichlet problem in {L}ipschitz domains for higher order
  elliptic systems with rough coefficients}.
\newblock {\em J. Anal. Math.}, 110:167--239, 2010.

\bibitem{meyers-serrin1960}
N.~Meyers and J.~Serrin.
\newblock {The exterior {D}irichlet problem for second order elliptic partial
  differential equations}.
\newblock {\em J. Math. Mech.}, 9:513--538, 1960.

\bibitem{mironescu2015}
P.~Mironescu.
\newblock {Note on {G}agliardo's theorem ``{${\rm tr}W^{1,1}=L^1$}''}.
\newblock {\em Ann. Univ. Buchar. Math. Ser.}, 6(LXIV)(1):99--103, 2015.

\bibitem{mironescu-russ2015}
P.~Mironescu and E.~Russ.
\newblock {Traces of weighted {S}obolev spaces. {O}ld and new}.
\newblock {\em Nonlinear Anal.}, 119:354--381, 2015.

\bibitem{necas2012book}
J.~Ne\v{c}as.
\newblock {\em {Direct Methods in the Theory of Elliptic Equations}}.
\newblock {Springer Monographs in Mathematics}. Springer, Heidelberg, 2012.
\newblock Translated from the 1967 French original.

\bibitem{peetre1976book}
J.~Peetre.
\newblock {\em {New Thoughts on {B}esov spaces}}.
\newblock Mathematics Department, Duke University, Durham, N.C., 1976.
\newblock Duke University Mathematics Series, No. 1.

\bibitem{ponce2004}
A.~C. Ponce.
\newblock An estimate in the spirit of {P}oincar\'e's inequality.
\newblock {\em J. Eur. Math. Soc. (JEMS)}, 6(1):1--15, 2004.

\bibitem{ponce-spector2017}
A.~C. Ponce and D.~Spector.
\newblock On formulae decoupling the total variation of {BV} functions.
\newblock {\em Nonlinear Anal.}, 154:241--257, 2017.

\bibitem{simader-sohr1996}
C.~G. Simader and H.~Sohr.
\newblock {\em {The {D}irichlet Problem for the {L}aplacian in Bounded and
  Unbounded Domains}}, volume 360 of {\em {Pitman Research Notes in Mathematics
  Series}}.
\newblock Longman, Harlow, 1996.
\newblock A new approach to weak, strong and $(2+k)$-solutions in Sobolev-type
  spaces.

\bibitem{strichartz2016}
R.~S. Strichartz.
\newblock {``{G}raph paper'' trace characterizations of functions of finite
  energy}.
\newblock {\em J. Anal. Math.}, 128:239--260, 2016.

\bibitem{taylor_functional}
A.~E. Taylor.
\newblock {\em {Introduction to Functional Analysis}}.
\newblock John Wiley \& Sons, Inc., New York; Chapman \& Hall, Ltd., London,
  1958.

\bibitem{thater2002}
G.~Th{\"a}ter.
\newblock {Neumann problem in domains with outlets of bounded diameter}.
\newblock {\em Acta Appl. Math.}, 73(3):251--274, 2002.

\bibitem{triebel1995book}
H.~Triebel.
\newblock {\em {Interpolation Theory, Function Spaces, Differential
  Operators}}.
\newblock Johann Ambrosius Barth, Heidelberg, second edition, 1995.

\bibitem{uspenskii1961}
S.~V. Uspenski\u{\i}.
\newblock {Imbedding theorems for weighted classes}.
\newblock {\em Amer. Math. Soc., Transl II. Ser. 87}, pages 121--145, 1970.
\newblock Translation from Trudy Mat. Inst. Steklov 60 (1961), 282--303.

\bibitem{zhou-du2010}
K.~Zhou and Q.~Du.
\newblock {Mathematical and numerical analysis of linear peridynamic models
  with nonlocal boundary conditions}.
\newblock {\em SIAM J. Numer. Anal.}, 48(5):1759--1780, 2010.

\end{thebibliography}

\end{document}